\documentclass[10pt,a4paper,reqno]{amsart}
\usepackage{amsthm,amssymb,amstext,graphicx}
\usepackage{bbm}
\usepackage{enumerate}
\usepackage{amscd}
\usepackage{amssymb}
\usepackage{amsthm}
\usepackage{mathtools}
\usepackage{hyperref}
\usepackage[capitalise]{cleveref}
\usepackage{enumitem}
\usepackage[utf8]{inputenc}
\usepackage[T1]{fontenc}
\usepackage{lmodern}
\usepackage[bibliography= common]{apxproof}
\usepackage{hyperref}
\hypersetup{
  colorlinks=true,
  pdfpagelabels=true,
  pdfpagelayout==TwoColumnRight,
  pdftitle={},
  pdfauthor={© Sven Karbach},
  pdfsubject={Project2},
  pdfkeywords={}
}

\DeclareMathOperator{\vect}{vec}
\DeclareMathOperator{\Cov}{Cov}
\newcommand*\D{\mathop{}\!\mathrm{d}}
\newcommand*\E{\mathop{}\!\mathrm{e}}
\newcommand*\I{\mathop{}\!\textnormal{i}}

\newtheorem{theorem}{Theorem}[section]
\newtheorem*{theorem*}{Theorem}
\newtheorem{lemma}[theorem]{Lemma}
\newtheorem{proposition}[theorem]{Proposition}
\newtheorem*{proposition*}{Proposition}
\newtheorem{corollary}[theorem]{Corollary}
\newtheorem*{corollary*}{Corollary}
\theoremstyle{plain}

\theoremstyle{definition}
\newtheorem{definition}[theorem]{Definition}
\newtheorem*{definition*}{Definition}

\newtheoremstyle{example}
  {.3\baselineskip}
  {.3\baselineskip}
  {\normalsize}  
  {0pt}       
  {\bfseries} 
  {.}         
  {5pt plus 1pt minus 1pt} 
  {}          
\theoremstyle{example}
\newtheorem{example}[theorem]{Example}

\newtheoremstyle{remark}
  {.2\baselineskip}
  {.2\baselineskip}
  {\normalfont}
  {}
  {\bfseries}
  {\ifx\thmnote\@gobble.\else\normalfont.\fi}
  {.5em}
  {}
\theoremstyle{remark}
\newtheorem{remark}[theorem]{Remark}

\newlist{propenum}{enumerate}{1}
\setlist[propenum]{label=\roman*), ref=\textup{\thetheorem~\roman*)}}
\crefalias{propenumi}{proposition}

\newlist{theoremenum}{enumerate}{1}
\setlist[theoremenum]{label=\roman*), ref=\textup{\thetheorem~\roman*)}}
\crefalias{theoremenumi}{theorem}

\newlist{lemenum}{enumerate}{1}
\setlist[lemenum]{label=\roman*), ref=\textup{\thelemma~\roman*)}}
\crefalias{lemenumi}{lemma}

\newcommand{\op}[1]{\mathbf{#1}}

\newcommand{\opz}{\mathbf{z}}
\newcommand{\opA}{\mathbf{A}}

\newcommand{\opC}{\mathbf{C}}

\newcommand{\opI}{\mathbf{I}}

\newcommand{\opP}{\mathbf{P}}
\newcommand{\opQ}{\mathbf{Q}}

\newcommand{\bu}{\mathbf{u}}

\newcommand{\bx}{\mathbf{x}}

\newcommand{\bF}{\mathbf{F}}

\renewcommand{\MR}{\mathbb{R}}
\newcommand{\MC}{\mathbb{C}}

\newcommand{\MM}{\mathbb{M}}
\newcommand{\MN}{\mathbb{N}}
\newcommand{\MP}{\mathbb{P}}

\newcommand{\MI}{\mathbb{I}}

\newcommand{\MS}{\mathbb{S}}
\newcommand{\MF}{\mathbb{F}}
\newcommand{\cF}{\mathcal{F}}
\newcommand{\cA}{\mathcal{A}}
\newcommand{\cB}{\mathcal{B}}
\newcommand{\cC}{\mathcal{C}}
\newcommand{\cD}{\mathcal{D}}

\newcommand{\cK}{\mathcal{K}}
\newcommand{\cL}{\mathcal{L}}

\newcommand{\cO}{\mathcal{O}}

\newcommand{\cQ}{\mathcal{Q}}

\newcommand{\cG}{\mathcal{G}}

\newcommand{\df}{\coloneqq}
\newcommand{\one}{\mathbbm{1}}
\DeclareMathAlphabet{\mymathbb}{U}{BOONDOX-ds}{m}{n}
\newcommand{\zero}{\mymathbb{0}}
\newcommand{\interior}[1]{({\kern0pt#1})^{\textnormal{o}}}
\newcommand{\set}[1]{\left\{ #1\right\}}
\newcommand{\norm}[1]{\|#1\|}
\newcommand{\llangle}{\langle\langle}
\newcommand{\rrangle}{\rangle\rangle}

\newcommand{\EX}[1]{\mathbb{E}\left[#1\right]}

\newcommand{\CEX}[2]{\mathbb{E}\left[#1\lvert\;#2\right]}
\newcommand{\Var}{\textnormal{Var}}

\newcounter{Task}

\newcommand{\MRplus}{\MR^{+}}

\begin{document}
\title[MCARMA on cones]{Multivariate continuous-time
  autoregressive moving-average processes on cones}  
\author{Fred Espen Benth}
\address[Fred Espen Benth]{University of Oslo, Department of Mathematics, POBox
  1053 Blindern, N-0316 Oslo, Norway}  
\email{fredb@math.uio.no}

\author{Sven Karbach$^{\ast}$}
\address[Sven Karbach]{University of Hamburg, Department of Mathematics,
  Bundesstr. 55, 20146 Hamburg, Germany} 
\email{sven@karbach.org}

\thanks{$^{\ast}$The corresponding author Sven Karbach received financial
  assistance from The Dutch Research Council (NWO) (Grant No: C.2327.0099)}
\thanks{F. E. Benth acknowledges financial support from SPATUS, a thematic
  research group funded by UiO:Energy.}


\begin{abstract} 
  In this article we study multivariate continuous-time autoregressive
  moving-average (MCARMA) processes with values in convex cones. More
  specifically, we introduce matrix-valued MCARMA processes with L\'evy noise
  and present necessary and sufficient conditions for processes from this
  class to be cone valued. We derive specific hands-on conditions in the
  following two cases: First, for classical MCARMA on $\MR_{d}$ with values in
  the positive orthant $\MR_{d}^{+}$. Second, for MCARMA processes on real
  square matrices assuming values in the cone of symmetric and positive
  semi-definite matrices. Both cases are relevant for applications and we give
  several examples of positivity ensuring parameter specifications. In addition
  to the above, we discuss the capability of positive semi-definite MCARMA
  processes to model the spot covariance process in multivariate
  stochastic volatility models. We justify the relevance of MCARMA based
  stochastic volatility models by an exemplary analysis of the second order
  moment structure of positive semi-definite well-balanced Ornstein-Uhlenbeck based
  models.\par{}
\end{abstract}

\maketitle

\vspace{-5mm}

\section{Introduction}\label{sec:bk22-introduction}
Multivariate continuous-time autoregressive moving-average processes are the
continuous time versions of the classical discrete-time VARMA models and have
been studied thoroughly over the last two decades, see \cite{MS07, SS12, BS13,
  Fas16, Kev18}. Similarly to their univariate analogs, the CARMA processes,
MCARMA processes can be interpreted as solutions to higher-order stochastic
differential equation of the form
\begin{align}\label{eq:bk22-MCARMA-higher-order-SDE}
  \mathrm{D}^{p}X_{t}+\tilde{A}_{1}\mathrm{D}^{p-1}X_{t}+\ldots+\tilde{A}_{p}X_{t}=\tilde{C}_{0}\mathrm{D}^{q+1}L_{t}+\tilde{C}_{2}\mathrm{D}^{q}L_{t}+\ldots+\tilde{C}_{q}\mathrm{D}
  L_{t},
\end{align}
where $\mathrm{D}=\frac{\D}{\D t}$, $(\tilde{A}_{i})_{i=1,\ldots,p}$ and
$(\tilde{C}_{j})_{j=0,\ldots,q}$ for $q,p\in\MN$ are two families of linear operators and
$(L_{t})_{\in\MR}$ denotes a multivariate L\'evy process. Naturally,
equation~\eqref{eq:bk22-MCARMA-higher-order-SDE} asks for a rigorous
definition as the paths of L\'evy processes are in general not
differentiable. Heuristically, however, we can interpret~\eqref{eq:bk22-MCARMA-higher-order-SDE} as the continuous-time version
of a (V)ARMA difference equation and we therefore expect that similar key
features governed by the autoregressive and moving-average structure of the
defining equation find their counterpart in the continuous-time
setting. Indeed, one notable feature of the (M)CARMA class is its
  flexible auto-covariance structure, which has the potential to model a wide
  variety of \emph{short memory} effects observed in many time series in
  applications and justifies the popularity of modeling with (M)CARMA
  processes in subjects ranging from finance over meteorology to natural science
  and engineering, see e.g.~\cite{JBB91, ML04, TT06, BB09}. In general, short
  memory refers to an exponentially fast decaying auto-covariance function such as,
  e.g., Ornstein-Uhlenbeck (OU) processes induce, which, as a matter of fact,
  form the subclass of MCAR processes of order one. MCARMA processes of
  general order, however, exhibit much more nuanced auto-covariance functions
  compared to its OU subclass.\par  
 
In many applications, where (M)CARMA models are employed, a crucial model
feature is positivity, e.g. in modeling wind speed~\cite{BB09, BCR21}, the
velocity field in turbulence~\cite{BNBV18} or volatility in finance~\cite{BNS01,
  TT06, BL13, BNS13}. It is therefore of great importance to understand the
capability of (M)CARMA processes to model phenomena with positive states. In the
univariate case positive CARMA processes were studied in~\cite{TC05, TC09,
  BDY11, BL13, BR19, NR21}. In particular, the authors in~\cite{TC05} give a set
of necessary and/or sufficient parameter conditions such that CARMA processes of
general order driven by a L\'evy subordinator assume only non-negative
values. Positivity of L\'evy-driven (multivariate) OU
  processes and its applications have been investigated in a number of
  articles~\cite{BNS01, BNS07, BDY11, PS09}. In particular, in~\cite{BNS07} the
  authors study OU processes on the cone of positive semi-definite matrices. With the present article we contribute to the growing literature on
MCARMA processes by studying positivity of multivariate CARMA processes
extending the univariate results in~\cite{TC05}. More
precisely, we introduce the class of matrix-valued MCARMA processes which
extends the classical $\MR_{d}$-valued MCARMA processes and present necessary and
sufficient conditions for matrix-valued MCARMA processes to be cone
valued. The particularly interesting cases of $\MR_{d}^{+}$, the positive
orthant in $\MR_{d}$, and  $\MS_{d}^{+}$, the cone of symmetric positive
semi-definite $d\times d$ matrices, are included in our analysis.\medskip

The starting point of our analysis will be the formulation of linear
continuous-time state space models on the space of real $n\times m$-matrices. It
is well known that $\MR_{d}$-valued MCARMA processes are characterized by
certain configurations of continuous-time state space models, see
e.g.~\cite{MS07, SS12, BS13}. Linear continuous-time state space models are
essentially given by an Ornstein-Uhlenbeck type process on the Cartesian product
of the particular state space, in our case the space of all $n\times
m$-matrices, and a linear output operator mapping the values of the
higher-dimensional OU type process into the state space again. We show that for
certain specifications of such models the vectorizations of stationary output
processes are equivalent to MCARMA processes as they were introduced
in~\cite{MS07}. In this sense the matrix-valued state space models give rise to
the novel class of matrix-valued MCARMA processes. The details are given in
Section~\ref{sec:bk22-carma-proc-on-matrices} below. Since in some applications one
is interested in non-stable systems, we explicitly include the class of
non-stable state space models in our analysis. Non-stable models often correspond to
so called non-causal MCARMA processes, i.e. MCARMA processes that are not
adapted to the natural filtration of its driving L\'evy process. A careful distinction between the stable and
non-stable case is justified, since the stability conditions may interact with the
imposed cone-invariance constraints.\par{} 

Once we set up the class of matrix-valued MCARMA processes we study its
invariance with respect to convex cones. In particular, we are looking for necessary and sufficient
parameter conditions such that a matrix-valued MCARMA process driven by a
multivariate cone-increasing L\'evy process takes values solely in this cone. 
As noted above, in the univariate case positivity, i.e. $\MR^{+}$-
invariance, of CARMA processes is well-studied and the relevance in applications is widely recognized.  In the multivariate setting, however,
cone-invariance of MCARMA processes has not yet been studied in a systematical
way. Partial results exist in the recent work~\cite{NR21}, where the authors
derive conditions ensuring the positivity of (univariate) CARMA processes. The
authors claim that some parts of their results could be extended to the
multivariate CARMA case, however, only few information about this extension are
provided. Our Theorem~\ref{thm:bk22-positivite-non-stable-MCARMA} below supports
this claim to some extent.\par{}   

The main motivation for studying matrix-valued MCARMA processes on cones comes from multivariate stochastic volatility modeling, see e.g.~\cite{BNS07, LT08, PS09, BNS13, CFMT11}. A multivariate stochastic volatility model consists of
a $d$ dimensional (logarithmic) price process and a spot covariance
process that modulates the multivariate noise of the former and takes values
in the cone of symmetric and positive semi-definite matrices. Classical choices for the spot covariance process are either pure-jump based models,
such as OU type models and superposition of these, so called
supOU processes, see~\cite{BNS07, PS09}, pure diffusion-based models such as the
Wishart processes, see~\cite{Bru91}, or mixtures of both, e.g. affine jump-diffusion
models, see~\cite{LT08, CFMT11}. We propose to use MCARMA processes on
symmetric and positive semi-definite matrices to model the spot covariance
process in multivariate stochastic volatility models. For one thing, the popular multivariate extension of the
Barndorff--Nielsen and Shepard (BNS) stochastic volatility model in~\cite{PS09} is included
in the much broader class of positive semi-definite MCARMA based models. In
addition, the nuanced short memory structure of the MCARMA class has the
potential to explain phenomena like non-monotone auto-correlation functions with
polynomial diminished exponential decay, that are often
observed in realized (co)variance time-series of financial data and that go
beyond the auto-correlation structure induced by OU-type models, see
Section~\ref{lem:bkk22-r++-well-balanced-ou} below. We demonstrate the capability of MCARMA based stochastic volatility models by means of
\emph{positive semi-definite well-balanced Ornstein-Uhlenbeck processes}. This
class stems from a particular configuration of MCARMA processes on positive
semi-definite matrices and we present a detailed analysis of it in
Section~\ref{sec:bk22-mcarma-based-mult}. Lastly, we want to mention that also in a discrete-time setting studying
autoregressive matrix-valued models is an active area of research, see
e.g.~\cite{ZJZL19, CXY21}. Many articles on (M)CARMA models deal with their
connection to the discrete-time setting, e.g. studying high frequency sampling
for MCARMA in \cite{Fas16, Kev18} or parameter estimation of the driving
noise from discrete observations in~\cite{SS12}. Our findings on the
equivalence of the vectorized matrix-valued MCARMA and the classical MCARMA
suggests that analogous connections to the discrete-time setting continue to
hold in the matrix-valued case and the good accessibility of the classical
MCARMA is maintained accordingly.

\vspace{-1mm}

\subsection{Layout of the article}\label{sec:bk22-layout-article}

This article is structured as follows: In
Section~\ref{sec:bk22-carma-proc-on-matrices} we define a class of continuous-time state
space models on matrices and show the equivalence with the classical MCARMA class. In
Section~\ref{sec:bk22-posit-mcarma-proc} we study the cone invariance of matrix- and
vector-valued MCARMA processes with a particular focus on the positive orthant
in $\MR_{d}$ and the cone of symmetric positive semi-definite
matrices. Lastly, in Section~\ref{sec:bk22-mcarma-based-mult} we propose an MCARMA
based multivariate stochastic volatility model and demonstrate its capability
to capture memory effects by an exemplary analysis of positive semi-definite
well-balanced OU processes.

\subsection{Notation}\label{sec:bk22-notation}
By $\MN$ we denote the set of integers and let $\MN_{0}\df\MN \cup
\set{0}$. For a complex number $z=a+\I b\in\MC$ we denote its real part $a$ by
$\Re(z)$ and its imaginary part $b$ by $\Im(z)$. For $d\in\MN$, we denote by
$\MR_{d}$ the $d$-dimensional Euclidean space equipped with the standard
inner-product $(\cdot, \cdot )_{d}$. The closed positive orthant in $\MR_{d}$
will be denoted by $\MR_{d}^{+}$ and the canonical basis in $\MR_{d}$ by $\set{e_{1},e_{2},\ldots,e_{d}}$.

\paragraph{\textbf{Matrices}}
Let $n,m\in\MN$ and let $\mathbb{K}$ denotes either a field or a ring. Then we
write $\MM_{n,m}(\mathbb{K})$ for the set of all $n\!\times\! m$ matrices over
$\mathbb{K}$. If $n=m$ we write $\MM_{n}(\mathbb{K})$ and if $\mathbb{K}=\MR$
we simplify to $\MM_{n,m}$. If $m=1$ we have $\MM_{n,1}=\MR_{n}$ and use the
latter notation. The $p$-times Cartesian product of $\MM_{n,m}$ will be denoted
by $(\MM_{n,m})^{p}$, which is just equivalent to $\MM_{pn,m}$, but we use the
former notation as it is more suggestive. In the case where $\mathbb{K}=\MR[\lambda]$
is the polynomial ring over $\MR$ the set $\MM_{n}(\MR[\lambda])$ denotes the space
of all matrix polynomials with coefficients in $\MM_{n}$. We refer
to~\cite{GLR09} for a comprehensive analysis of matrix polynomials.
We denote the transpose of a matrix $A\in\MM_{n,m}$ by $A^{\intercal}$, which
is an element in $\MM_{m,n}$, and write $\MS_{d}$ for the subspace in $\MM_{n}$ consisting of all
symmetric $n\times n$ matrices, i.e. all $A\in\MM_{n}$ such that
$A^{\intercal}=A$. The set of all symmetric positive semi-definite $n\times
n$-matrices will be denoted by $\MS_{n}^{+}$,
i.e. $\MS_{n}^{+}=\set{A\in\MS_{n}\colon (A x,x)_{n}\geq 0,\,\forall
  x\in\MR_{n}}$. If necessary, we can express real $n\times m$-matrices in a
component-wise notation by $A=(a_{i,j})_{1\leq i\leq n, 1\leq j\leq m}$ for
$a_{i,j}\in\MR$. For all $n,m\in\MN$ we can identify $\MM_{n,m}$ with
$\MR_{nm}$ through the vectorization operator $\vect\colon \MM_{n,m}\to
\MR_{nm}$ which transforms a matrix into a vector by stacking the columns below each other. Similarly, we
denote by $\mathrm{vech}\colon \MS_{n}\to \MR_{n(n+1)/2}$ the operator that
stacks only the lower triangular part of a symmetric matrix below another. On
$\MM_{n,m}$ we consider the inner-product $\langle\cdot, \cdot\rangle_{nm}$
given by $\langle A,B\rangle_{nm}=(\vect(A),\vect(B))_{nm}$ and denote the
induced norm by $\norm{\cdot}_{nm}$. Let $n_{1},n_{2},m_{1},m_{2}\in\MN$, then
for $A\in\MM_{n_{1},m_{1}}$ and $B\in\MM_{n_{2},m_{2}}$ we denote the
Kronecker product of $A$ and $B$ by $A\otimes B\in
\MM_{n_{2}n_{1},m_{2}m_{1}}$. We denote the Hadamard product of two matrices
$A,B\in\MM_{n,m}$ by $A\odot B$ and let $\one_{n,m}=(1_{i,j})_{1\leq i\leq
  n,1\leq j\leq m}$ stand for the matrix in $\MM_{n,m}$ which is equal to one
in every component and $\zero_{n,m}$ denotes the $n\times m$-zero matrix. If
$n=m$ we write $\one_{n}\df \one_{n,n}$, $\zero_{n}\df\zero_{n,n}$ and denote
the identity matrix by $\MI_{n}$.

\paragraph{\textbf{Linear operators on matrices}}
We denote the algebra of all linear operators from  $\MM_{n_{1},m_{1}}$ to $\MM_{n_{2},m_{2}}$ by
$\cL(\MM_{n_{1},m_{1}},\MM_{n_{2},m_{2}})$. If $n_{1}=n_{2}=n$ and
$m_{1}=m_{2}=m$ we write $\cL(\MM_{n,m})$ and if $m=1$, it is well known that
$\cL(\MM_{n,1})=\cL(\MR_{n})\simeq \MM_{n}$ and we use the latter notation. If
$n,m\in\MN$ are greater than one, we will denote elements of $\cL(\MM_{n,m})$ by bold face letters, e.g. $\opA\in\cL(\MM_{n,m})$ versus $A\in \MM_{n,m}$ and we reserve the calligraphic
letters, e.g. $\cA$, for linear operators mapping from or to
$(\MM_{n,m})^{p}$ for some integer $p>1$. Moreover, we denote
  by $\cA^{*}$ (or $\opA^{*}$) the adjoint operator of $\cA$ (respectively $\opA$). Since in the sequel we will
always make sure that there is no confusion regarding the matrix space that we
are operating in, we denote the identity operator in $\cL(\MM_{n,m})$ simply by
$\opI$ and will only index $\opI$, when we speak of the identity in
$\cL((\MM_{n,m})^{p})$, in which case we write $\opI_{p}$. For every $\opA\in
\cL(\MM_{n,m})$ we denote its spectrum by $\sigma(\opA)$, which, as we work in
finite-dimensions, is just the set of eigenvalues of $\opA$. 
Moreover, we denote the spectral bound of $\opA$ by $\tau(\opA)$,
i.e. $\tau(\opA)\df \max\set{\Re(\lambda)\colon \lambda \in \sigma(\opA)}$.

\paragraph{\textbf{Matrix-valued L\'evy processes}}
Let $(\Omega,\cF,\MF,\MP)$ be a filtered probability space satisfying the
usual conditions and let $(L_{t})_{t\geq 0}$ be an $\MR_{nm}$-valued L\'evy
process defined on this probability basis, see~\cite{Sat99} for a comprehensive
analysis of multivariate L\'evy processes. Since we can always identify
$\MM_{n,m}$ with $\MR_{nm}$ the class of multivariate L\'evy processes easily
extends to matrix-valued L\'evy processes. We recall that the L\'evy
characteristic exponent at $z\in\MM_{n,m}$ is given by  
  \begin{align}
    \label{eq:bk22-chracteristic-Levy}
    \psi_{L}(z)=\I \langle
    \gamma_{L},z\rangle_{nm}\!-\!\frac{1}{2}\langle Q_{1}
    z,z\rangle_{nm}+\int_{\MM_{n,m}}\hspace{-3mm}\big(\E^{\I \langle
    \xi,z\rangle_{nm}}-1-\I\langle \chi(\xi),z\rangle_{nm}\big)\nu_{L}(\D\xi),
  \end{align}
  where $\gamma_{L}\in\MM_{n,m}$, $Q_{1}$ is the \emph{covariance operator} of
  the continuous part of the L\'evy process, $\chi(\xi)\df\xi
  \one_{\norm{\xi}_{nm}\leq 1}(\xi)$ and $\nu_{L}\colon\cB(\MM_{n,m})\to
  \MRplus$ is the \emph{L\'evy measure}. We call a L\'evy process
  $(L_{t})_{t\geq 0}$ \emph{integrable}, if $\EX{\norm{L_{t}}_{nm}}<\infty$ for
  all $t\geq 0$ and
  \emph{square-integrable} whenever $\EX{\norm{L_{t}}^{2}_{nm}}<\infty$ for
  all $t\geq 0$. For a square-integrable L\'evy process $L$ with characteristic
  exponent given by~\eqref{eq:bk22-chracteristic-Levy} the mean of $L_{1}$ is
  denoted by $\mu_{L}$ and we have  $\mu_{L}=\big(\gamma_{L}+\int_{\MM_{n,m}\cap\set{\xi\colon\norm{\xi}_{n,m}>
      1}}\xi\nu_{L}(\D\xi)\big)$. Moreover, we denote by $\cQ\in\cL(\MM_{n,m})$
  the covariance operator of $L_{1}$, given by $\cQ=Q_{1}+\int_{\MM_{n,m}}\xi\otimes\xi \nu_{L}(\D\xi)$. For any L\'evy process
  $(L^{(1)}_{t})_{\geq 0}$ defined on the positive real line $\MRplus$, we can
  choose a second independent and identically distributed L\'evy process
  $(L_{t}^{(2)})_{t\geq 0}$ to define a \emph{two-sided L\'evy process}
  $(L_{t})_{t\in\MR}$ by
  \begin{align}\label{eq:two-sided-Levy}
   L_{t}\df\one_{\MRplus}(t)L^{(1)}_{t}-\one_{\MR^{\!\!\!\!~^{-}}}(t)L_{-t-}^{(2)},  
  \end{align}
  where $L^{(2)}_{t-}\df\lim_{s\nearrow t}L^{(2)}_{s}$ for all $t\geq 0$ and
  $\MR^{-}\df-\MRplus\setminus\set{0}$. Throughout this
  article we use the conventional and intuitive notation of stochastic
  integration with respect to matrix-valued integrators analogous to, e.g.,~\cite{BNS07, BNS13}.

\section{Linear state-space models and matrix-valued MCARMA}\label{sec:bk22-carma-proc-on-matrices}

Throughout this section we fix $m,n\in\MN$ and let $(\Omega,\cF,\MF,\MP)$
denote a filtered probability space satisfying the usual conditions. Moreover,
we assume that $(\Omega,\cF,\MF,\MP)$ is rich enough to carry an
$\MM_{n,m}$-valued two-sided L\'evy process
$L=(L_{t})_{t\in\MR}$. All proofs of this section are either
  deferred to the appendix or can be found in the pertinent literature. We begin this section by introducing a very general class
of linear \emph{continuous-time state space models} defined on real $n\times
m$-matrices:   

\begin{definition}\label{def:bk22-state-space-model}
  Let $p\in\MN$ and let the tuple $(\cA,\cB,\cC,L)$ consist of a \emph{state
    transition operator} $\cA\in \cL\big((\MM_{n,m})^{p}\big)$, an \emph{input operator}
  $\cB\in\cL\big(\MM_{n,m},(\MM_{n,m})^{p}\big)$, an \emph{output operator}
  $\cC\in\cL\big((\MM_{n,m})^{p},\MM_{n,m}\big)$ and an $\MM_{n,m}$-valued two-sided
  L\'evy process $L=(L_{t})_{t\in\MR}$. A \emph{continuous-time linear
  state space model} on $\MM_{n,m}$, associated with the parameter set $(\cA,\cB,\cC,L)$, consists of a
  \textit{state-space equation} given by
  \begin{align}\label{eq:bk22-state-space-Z}
    \D Z_{t}=\cA Z_{t}\D t+\cB \D L_{t},\quad t\in\MR,
  \end{align}
  and an \emph{observation equation} given by
  \begin{align}\label{eq:bk22-output-X}
    X_{t}=\cC Z_{t},\quad t\in\MR.
  \end{align}
  We call the $(\MM_{n,m})^{p}$-valued process $(Z_{t})_{t\in\MR}$ the
  \emph{state process} and the $\MM_{n,m}$-valued process $(X_{t})_{t\in\MR}$
  the \emph{output process} of the continuous-time linear state space model
  (associated with $(\cA,\cB,\cC,L)$). 
\end{definition}

Continuous-time linear state space models have been rigorously studied in the
deterministic and stochastic control literature over many decades, see e.g.~\cite{ZD63, AST06, FR00}. 
We note that if $m=1$ (that means for the state space $\MR_{n}=\MM_{n,1}$) our definition of a linear
continuous-time state space model coincides with the one in~\cite[Definition
3.1]{SS12}.\par{}

We note that the state process $(Z_{t})_{t\in\MR}$ of a linear continuous-time
state space model is simply a L\'evy driven Ornstein-Uhlenbeck type process on the
space $(\MM_{n,m})^{p}$ and a solution to~\eqref{eq:bk22-state-space-Z} is given by
the following variation-of-constant formula: 
\begin{align}\label{eq:bk22-Z}
  Z_{t}=\E^{(t-s)\cA}Z_{s}+\int_{s}^{t}\E^{(t-u)\cA}\cB\D L_{u},\quad s<t\in\MR.
\end{align}
In general, a solution $(Z_{t})_{t\in\MR}$ to~\eqref{eq:bk22-state-space-Z} is not
unique without specifying an initial value. If for some $s\in\MR$ we
are given an $\cF_{s}$-measurable random variable $Z_{s}$, then
$(Z_{t})_{t\geq s}$ given by~\eqref{eq:bk22-Z} is the
unique solution to~\eqref{eq:bk22-state-space-Z} on $[s,\infty)$ adapted to the
filtration $(\cF_{t})_{t\geq s}$. Moreover, if the spectral bound of the transition operator $\cA$ is strictly negative,
i.e. $\tau(\cA)<0$, then it follows from~\cite{SY83, CM87}
that there exists a unique \emph{stationary solution} to~\eqref{eq:bk22-state-space-Z} if
and only if $\EX{\log(\norm{L_{1}}_{nm})}<\infty$. In this
case the unique stationary solution $(Z_{t})_{t\in\MR}$ is adapted and given by
\begin{align}\label{eq:bk22-Z-stationary}
  Z_{t}=\int_{-\infty}^{t}\E^{(t-s)\cA}\cB \D L_{s},\quad t\in\MR.
\end{align}

\begin{remark}[Uniqueness, stationarity and adaptedness]\label{rem:uniqueness}
  Note that in case of $\tau(\cA)\geq 0$, there could still, under certain
  conditions, exist a (unique) stationary solution
  to~\eqref{eq:bk22-state-space-Z} on $\MR$, see also
  Proposition~\ref{prop:bk22-state-space-MCARMA} below. However, it
  might happen that the stationary solution is not adapted to the natural
  filtration $\MF^{L}$, since $Z_{t}$ possibly depends on the generated
  sigma-algebra $\sigma(L_{s}\colon s>t)$. A concrete example of
  a stationary output process with $\tau(\cA)\geq 0$ is given in
  Section~\ref{sec:bk22-posit-semi-defin-2}. If from a modeling
  perspective adaptedness to the natural filtration is required, then we shall
  either assume $\tau(\cA)<0$, for which the existence of a unique
  adapted solution is known by the reasoning above. Or in case of
  $\tau(\cA)\geq 0$, there may exist many solutions, but only one for every
  fixed $\cF_{s}$-measurable initial condition, which also happens to be
  adapted, but possibly non-stationary.
\end{remark}

\vspace{3mm}

In the sequel we often distinguish between the two cases in
Remark~\ref{rem:uniqueness} which we call
the \emph{stable} (where $\tau(A)<0$) and \emph{non-stable} (where
$\tau(A)\geq 0$) case (following the usual nomenclature it is
actually the \emph{exponentially stable} and \emph{non-exponentially stable} case). We find it
appropriate to give a precise definition to prevent confusion:

\begin{definition}\label{def:bk22-stable-processes}
  We call a continuous-time state space model associated with the parameter set
  $(\cA,\cB,\cC,L)$ \emph{stable}, whenever $\tau(\cA)<0$. If
  moreover, $L$ has finite log-moments, i.e. $\EX{\log(\norm{L_{1}}_{n,m})}<\infty$, then whenever we refer to the state process $(Z_{t})_{t\in\MR}$ we mean the unique stationary
  solution given by~\eqref{eq:bk22-Z-stationary}. In this case we call
  $(Z_{t})_{t\in\MR}$ the \emph{stable state process} and
  $(X_{t})_{t\in\MR}$ the \emph{stable output process}. In case
  of $\tau(\cA)\geq 0$, we call the state space model \emph{non-stable} and refer
  to $(Z_{t})_{t\in\MR}$ in~\eqref{eq:bk22-Z} simply as a state process. If
  uniqueness is required, we may fix an initial value $Z_{s}$ at $s\in\MR$ for
  some $\cF_{s}$-measurable random variable $Z_{s}$.
\end{definition}

Given a state process $(Z_{t})_{t\in\MR}$ we now shift our focus to the
output process $(X_{t})_{t\in\MR}$ defined in~\eqref{eq:bk22-output-X}. In the
next proposition we summarize some well known and easy to check properties of
$(X_{t})_{t\in\MR}$.  

\begin{proposition}\label{prop:bk22-output-process-properties}
  Let $(\cA,\cB,\cC,L)$ be as in Definition~\ref{def:bk22-state-space-model} and
  let $\psi_{L}$ be the characteristic exponent of $L$ given
  by~\eqref{eq:bk22-chracteristic-Levy}. Then
  the process $(X_{t})_{t\in\MR}$ in~\eqref{eq:bk22-output-X} satisfies 
  \begin{align}\label{eq:bk22-X-variation-of-constants}
    X_{t}=\cC \E^{(t-s)\cA}Z_{s}+\int_{s}^{t}\cC \E^{(t-u)\cA}\cB\D
    L_{u},\quad s<t\in\MR,   
  \end{align}
  and for every $x\in \MM_{n,m}$ and $s\leq t$ we have
  \begin{align}\label{eq:bk22-characteristic-function}
    \CEX{\E^{\I\langle X_{t},x\rangle_{mn}}}{\mathcal
    \cF_s}&=\exp\Big(\I\langle\cC\E^{(t-s)\cA}Z_s,x\rangle_{mn}+\int_{s}^{t}\psi_L\big(
    \cB^{*}\E^{(t-u)\cA^{*}}\cC^{*}x\big)\,\D u\Big).
\end{align}
If $L$ is integrable, then  the conditional mean of $X_{t}$ is finite and
given by
\begin{align}\label{eq:bk22-X-mean}
  \CEX{X_{t}}{\cF_{s}}=\cC\E^{(t-s)\cA}Z_s+\int_{0}^{t-s}\cC\E^{u\cA}\cB \mu_{L}\,\D
  u,\quad s<t\in\MR.
\end{align}

If moreover $L$ has finite log-moments and $(X_{t})_{t\in\MR}$ is a stable output
process, then it is stationary and given by
  \begin{align}\label{eq:bk22-X-stationary}
    X_{t}=\int_{-\infty}^{t}\cC \E^{(t-s)\cA}\cB \D L_{s},\quad t\in\MR,
  \end{align}
  and if in addition $L$ is integrable, then the mean of $(X_{t})_{t\in\MR}$
  is given by
  \begin{align}\label{eq:bk22-X-mean-stationary}
    \EX{X_{t}}=-\cC\cA^{-1}\cB\mu_{L},\quad t\in\MR.  
  \end{align}  
\end{proposition}

From~\eqref{eq:bk22-X-stationary} we see that the dynamics of a stable output
process $(X_{t})_{t\in\MR}$ are only governed by the parameters
$\cA,\cB$ and $\cC$ and the L\'evy process $L$. More specifically, the
dynamics depend solely on the action of the kernel function
$g\colon\MRplus\to\cL(\MM_{n,m})$ given by $g(t)\df\cC\E^{\cA t}\cB$ applied to
the increments of $(L_{t})_{t\in\MR}$. In contrast, we see that in the
non-stable case, the dynamics of $(X_{t})_{t\geq s}$ also depend on the
initial value $Z_{s}$ for $s<t$ through the term
$\cC\E^{(t-s)\cA}Z_{s}$. As we will see, this has some consequences for the
techniques available to study the cone-invariance of output processes in Section~\ref{sec:bk22-posit-mcarma-proc}.\par{}   

\subsection{Controller canonical form and matrix-valued MCARMA processes}
In this section we introduce a more particular form of linear continuous-time state
space models on $\MM_{n,m}$ which in the discrete-time control literature is
often called the \emph{controller canonical form}. This controller canonical form will
prove itself useful for two reasons: First, it will allow us to interpret the
output processes of certain continuous-time state space models as a linear
transformation of MCARMA processes as they were introduced in the seminal
work~\cite{MS07}. Second, this form is particularly convenient to study
positivity questions in the next section.\par{}

\vspace{3mm}

\begin{definition}
Let $p\in\MN$ and denote by $\op0$ the null operator in $\cL(\MM_{n,m})$,
i.e. $\op0$ maps every $x\in\MM_{n,m}$ to the null matrix $\zero_{n,m}$. Moreover, let
$\opA_{1},\opA_{2},\ldots,\opA_{p}\in\cL(\MM_{n,m})$ and define the state
transition operator $\cA_{p}\colon (\MM_{n,m})^{p}\to(\MM_{n,m})^{p}$ as
follows: 
\begin{align}\label{eq:bk22-Ap}
\cA_{p}\df
  \begin{bmatrix}
   \op0 & \opI & \op0 & \ldots & \op0 \\
   \op0 & \op0 & \opI & \ddots & \vdots \\
   \vdots &  & \ddots   & \ddots & \op0 \\
   \op0 & \ldots & \ldots & \op0 & \opI     \\
   \opA_{p} & \opA_{p-1} & \ldots & \ldots & \opA_{1}
  \end{bmatrix},  
\end{align}
where for every $\bx=(x_{1},\ldots,x_{p})^{\intercal}\in (\MM_{n,m})^{p}$ we understand that
\begin{align*}
\cA_{p}(\bx)=\big(x_{2},\ldots,x_{p},\sum_{i=1}^{p}\opA_{p-i+1}(x_{i})\big)^{\intercal}\in (\MM_{n,m})^{p}.
\end{align*}
Moreover, let $q\in\MN_{0}$ such that $q<p$ and let
$\opC_{0},\opC_{1},\ldots,\opC_{p-1}\in\cL(\MM_{n,m})$ with $\opC_{i}=\op0$
for every $q+1\leq i\leq p-1$. We then define the output operator $\cC_{q}\colon (\MM_{n,m})^{p}\to \MM_{n,m}$ by
\begin{align}\label{eq:bk22-Cq}
\cC_{q}\df[\opC_{0},\opC_{1},\ldots,\opC_{p-1}],  
\end{align}
where $\cC_{q}$ is to be understood as follows: For $\bx=(x_{1},\ldots,x_{p})^{\intercal}\in(\MM_{n,m})^{p}$ we have
\begin{align*}
  \cC_{q}(\bx)=\sum_{i=1}^{q+1}\opC_{i-1}(x_{i}).
\end{align*}
Finally, we define the input operator $E_{p}\in\cL(\MM_{n,m},(\MM_{n,m})^{p})$ by
\begin{align}\label{eq:bk22-Ep}
E_{p}\df e_{p}\otimes\,\opI,  
\end{align}
which for every $x\in\MM_{n,m}$ is defined as $E_{p}(x)=
(\zero_{n,m},\ldots,\zero_{n,m},x)^{\intercal}\in(\MM_{n,m})^{p}$.
\end{definition}

With this specification of $(\cA_{p},E_{p},\cC_{q},L)$ the state process
$(Z_{t})_{t\in\MR}$ becomes 
\begin{align}
  \label{eq:bk22-Z-controller}
  Z_{t}=\E^{(t-s)\cA_{p}}Z_{s}+\int_{s}^{t}\E^{(t-u)\cA_{p}}E_{p}\D L_{u},\quad s<t\in\MR,
\end{align} 
and the output process
\begin{align}\label{eq:bk22-X-controller}
  X_{t}=\cC_{q}\E^{(t-s)\cA_{p}}Z_{s}+\int_{s}^{t}\cC_{q}\E^{(t-u)\cA_{p}}E_{p}\D L_{u},\quad s<t\in\MR,  
\end{align}
and the analogous formulas~\eqref{eq:bk22-Z-stationary} and~\eqref{eq:bk22-X-stationary}
hold in the stable case. Note that the
transition matrix $\cA_{p}$ can be viewed as the companion block operator
matrix of the following operator polynomial $\opP$ with coefficients in
$\cL(\MM_{n,m})$:
\begin{align}
  \label{eq:bk22-P}
  \opP(\lambda)\df
  \opI \lambda^{p}-\opA_{1}\lambda^{p-1}-\opA_{2}\lambda^{p-2}-\ldots-\opA_{p},\quad \lambda\in\MC.
\end{align}
In the same spirit we introduce the operator polynomial $\opQ$ which is
associated with the output operator $\cC_{q}$ and given by
\begin{align}
  \label{eq:bk22-Q}
  \opQ(\lambda)\df \opC_{0}+\opC_{1}\lambda+\opC_{2}\lambda^{2}+\ldots+\opC_{q}\lambda^{q},\quad \lambda\in\MC. 
\end{align}

Now, recall $\vect\colon \MM_{n,m}\to \MR_{nm}$ being the linear isometric
isomorphism that maps every matrix $A\in\MM_{n,m}$ to the vector of its columns by
stacking the columns below each other. If $A\in\cL(\MM_{n,m})$ we denote
by $A^{\vect}$ the \emph{matrix representation} of $A$ given by $A^{\vect}\df\vect\circ
A\circ \vect^{-1}$. Note that $A^{\vect}\in \cL(\MR_{nm})\simeq
\MM_{nm}$ and by identification we consider $A^{\vect}$ as a $nm\times nm$-matrix. Moreover, we denote by $K^{(n,m)}\in\MM_{nm}$ the commutation matrix,
which is the unique matrix in $\MM_{nm}$ such that for every $A\in \MM_{n,m}$
we have $K^{(n,m)}\vect(A)=\vect(A^{\intercal})$. We denote the inverse of
$K^{(n,m)}$ by $K^{-(n,m)}$, which happens to be the transpose of $K^{(n,m)}$ as well. When the $\vect$-operator is
applied to an output process $(X_{t})_{t\in\MR}$ we obtain an
$\MR_{nm}$-valued process $(\vect(X_{t}))_{t\in\MR}$. By linearity, we see
that the process $(\vect(X_{t}))_{t\in\MR}$ can again be considered as an
output process of a certain continuous-time linear state space model on
$\MR_{nm}$. In the following proposition we show more, namely that the
controller canonical form in~\eqref{eq:bk22-Ap}-\eqref{eq:bk22-Ep}
transforms, under the $\vect$-transformation, into a controller canonical-like
form of an $\MR_{nm}$-valued state space model. 

\begin{proposition}\label{prop:bk22-vectorized-state-space-model}
  Let $p\in\MN$ and $q\in\MN_{0}$ such that $q<p$ and
  $(\cA_{p},E_{p},\cC_{q},L)$ be as in~\eqref{eq:bk22-Ap}-\eqref{eq:bk22-Ep} and $L$ a
  two-sided L\'evy process on $\MM_{n,m}$. Moreover, let the state process
  $(Z_{t})_{t\in\MR}$ be given by~\eqref{eq:bk22-Z-controller} and the output
  process $(X_{t})_{t\in\MR}$ be as in~\eqref{eq:bk22-X-controller}. We set
  $L^{\vect}\df\vect(L)$ and define the
  \emph{state transition matrix} $\hat{\cA}_{p}^{\vect}\in \MM_{pnm}$ by 
  \begin{align}\label{eq:bk22-Ap-vect}
    \hat{\cA}^{\vect}_{p}\df
    \begin{pmatrix}
      \zero_{nm} & K^{-(n,m)} & \zero_{nm} & \ldots  & \zero_{nm} \\
      \zero_{nm} & \zero_{nm} & K^{-(n,m)} &  & \vdots \\
      \vdots &   & \qquad\ddots   & \quad\ddots & \vdots  \\
      \vdots &  & \quad\ddots   & \ddots & \zero_{nm} \\
      \zero_{nm} & \ldots & \ldots & \zero_{nm} & K^{-(n,m)}     \\
      \!\hat{A}^{\vect}_{p} & \hat{A}^{\vect}_{p-1} & \ldots & \ldots &
      \hat{A}^{\vect}_{1} 
    \end{pmatrix},
  \end{align}
  where $\hat{A}^{\vect}_{i}\df
  A_{i}^{\vect}\circ K^{-(n,m)}\in\MM_{nm}$ for every $i=1,\ldots,p$ and
  $A_{i}^{\vect}$ denotes the matrix representation of $\opA_{i}$. Moreover,
  we define the 
  \emph{output matrix} $\cC_{q}^{\vect}\in \MM_{nm,pnm}$ by 
  \begin{align}\label{eq:bk22-Cq-vect}
    \hat{\cC}_{q}^{\vect}\df\left(\hat{C}_{0}^{\vect},\hat{C}_{1}^{\vect},\ldots,\hat{C}_{p-1}^{\vect}\right),    
  \end{align}
  where $\hat{C}^{\vect}_{j}=C_{j}^{\vect}\circ
  K^{-(n,m)}\in\MM_{nm}$ for every $j=0,\ldots,p-1$ and
  $C^{\vect}_{j}$ denotes the matrix representation of $\opC_{j}$. Lastly, we
  define the \emph{input matrix} $\hat{E}_{p}^{\vect}\in \MM_{pnm,nm}$ by
  \begin{align}\label{eq:bk22-Ep-vect}
   \hat{E}_{p}^{\vect}\df e_{p}\otimes K^{-(n,m)}.
  \end{align}
  Then $(\vect(X_{t}))_{t\in\MR}$ is the output process of an $\MR_{nm}$-valued continuous-time state space model associated with
  $(\hat{\cA}^{\vect}_{p},\hat{\cC}_{q}^{\vect},\hat{E}_{p}^{\vect},L^{\vect})$ and such that
  \begin{align}
    \D \vect(Z_{t}^{\intercal})&= \hat{\cA}^{\vect}_{p}\vect(Z_{t}^{\intercal})\D
                                 t+\hat{E}^{\vect}_{p}\D
                                 L^{\vect}_{t},\quad t\in\MR,\label{eq:bk22-vect-X-state-space-rep-1}
  \\ 
    \vect(X_{t})&= \hat{\cC}_{q}^{\vect}\vect(Z_{t}^{\intercal}),\quad t\in\MR.\label{eq:bk22-vect-X-state-space-rep-2}   
  \end{align}
\end{proposition}
\begin{toappendix}
  \begin{proof}[Proof of Proposition~\ref{prop:bk22-vectorized-state-space-model}]
For every $t\in\MR$ write
$Z_{t}=(Z_{t}^{(1)},Z_{t}^{(2)},\ldots,Z_{t}^{(p)})^{\intercal}\in(\MM_{n,m})^{p}$
where for every $i=1,\ldots,p$ we denote by $Z_{t}^{(i)}$ the
$i$-th $n\times m$-block matrix component of $Z_{t}$. Given the state process
$(Z_{t})_{t\in\MR}$ and output process $(X_{t})_{t\in\MR}$ we show that
$(\vect(X_{t}))_{t\in\MR}$ solves~\eqref{eq:bk22-vect-X-state-space-rep-2}
with $(\vect(Z_{t}^{\intercal}))_{t\in\MR}$ being a solution to~\eqref{eq:bk22-vect-X-state-space-rep-1}. By definition we have $X_{t}=\cC_{q}(Z_{t})=\sum_{i=1}^{q+1}\opC_{i-1}(Z_{t}^{(i)})$
and hence by linearity of the $\vect$-operation we see that
\begin{align}\label{eq:bk22-proof-vectorized-state-space-model-1}
  \vect(X_{t})=\vect(\cC_{q}(Z_{t}))=\sum_{i=1}^{q+1}\vect(\opC_{i-1}(Z_{t}^{(i)}))=\sum_{i=1}^{q+1}C_{i-1}^{\vect}\vect(Z_{t}^{(i)}),\quad t\in\MR.
\end{align}
  As before let $K^{(pn,m)}$ be the commutation matrix such that
  $K^{(pn,m)}\vect(F)=\vect(F^{\intercal})$ for every
  $F\in(\MM_{n,m})^{p}$. Then $\vect(Z_{t}^{(i)})=K^{-(n,m)}\vect((Z_{t}^{(i)})^{\intercal})$ and for every $i=1,2,\ldots,p$ we can write 
  \begin{align*}
   \vect((Z_{t}^{(i)})^{\intercal})=(e_{i}^{\intercal}\otimes \MI_{nm})\vect(Z_{t}^{\intercal}). 
  \end{align*}
  Hence for every $i=1,2,\ldots,p$ and $t\in\MR$ we have
  \begin{align}\label{eq:bk22-proof-vectorized-state-space-model-2}
    \vect(Z_{t}^{(i)})&=K^{-(n,m)}(e_{i}^{\intercal}\otimes \MI_{nm})K^{(pn,m)}\vect(Z_{t}).  
  \end{align}
  We thus continue with the term $\vect(Z_{t})$ appearing
  on the right-hand side of~\eqref{eq:bk22-proof-vectorized-state-space-model-2}. Note that
  by linearity, inserting~\eqref{eq:bk22-Z-controller} into $\vect(Z_{t})$ yields
  \begin{align}\label{eq:bk22-proof-vectorized-state-space-model-3}
    \vect(Z_{t})&=\vect(\E^{(t-s)\cA_{p}}Z_{s})+\int_{s}^{t}\vect(\E^{(t-u)\cA_{p}}E_{p}\D
                  L_{u}),\quad s<t\in\MR.
  \end{align}
  We know that $t\mapsto \E^{t\cA_{p}}v_{0}$ is the unique solution to the linear
  equation $\frac{\partial}{\partial t}v(t)=\cA_{p}v(t)$ with
  $v(0)=v_{0}$. Moreover, we see that $\frac{\partial}{\partial
    t}\vect(v(t))=\cA_{p}^{\vect}\vect(v(t))$ with
  $\vect(v(0))=\vect(v_{0})$ is uniquely solved by
  $\vect(v(t))=\E^{t\cA_{p}^{\vect}}\vect(v_{0})$. Hence by linearity, we
  conclude that
  $(\E^{t\cA_{p}})^{\vect}=\E^{t\cA^{\vect}_{p}}$ must hold for all $t\geq
  0$. Therefore, the right-hand side
  in~\eqref{eq:bk22-proof-vectorized-state-space-model-3} coincides with
  \begin{align}\label{eq:bk22-proof-vectorized-state-space-model-4}
    \E^{(t-s)\cA^{\vect}_{p}}\vect(Z_{s})+\int_{s}^{t}\E^{(t-u)\cA^{\vect}_{p}}\vect(E_{p}\D
    L_{u}),\quad s<t\in\MR.
  \end{align}
  Hence by using the $K^{(pn,m)}$-commutation matrix again, we see that for
  all real $s<t$ we have
  \begin{align*}
   \vect(Z_{t})=\E^{(t-s)\cA^{\vect}_{p}}K^{-(pn,m)}\vect(Z_{s}^{\intercal})+\int\limits_{s}^{t}\E^{(t-u)\cA^{\vect}_{p}}K^{-(pn,m)}\vect((E_{p}\D
    L_{u})^{\intercal}),
  \end{align*}
  and hence by~\eqref{eq:bk22-proof-vectorized-state-space-model-2}
  \begin{align}\label{eq:bk22-proof-vectorized-state-space-model-5}
    \vect(Z_{t}^{(i)})&=K^{-(n,m)}(e_{i}^{\intercal}\otimes
    \MI_{nm})\Big(K^{(pn,m)}\E^{(t-s)\cA^{\vect}_{p}}K^{-(pn,m)}\vect(Z_{s}^{\intercal})\nonumber\\
    &\quad
      +\int_{s}^{t}K^{(pn,m)}\E^{(t-u)\cA^{\vect}_{p}}K^{-(pn,m)}\vect((E_{p}\D
      L_{u})^{\intercal})\Big),\quad s<t\in\MR.
  \end{align}
  Now note that
  \begin{align*}
    K^{(pn,m)}\E^{t\cA_{p}^{\vect}}K^{-(pn,m)}=\E^{t
    K^{(pn,m)}\cA_{p}^{\vect}K^{-(pn,m)}},\quad \text{for all }t\geq 0.
  \end{align*}
  Next, we show that
  $K^{(pn,m)}\cA_{p}^{\vect}K^{-(pn,m)}=\hat{\cA}^{\vect}_{p}$. Let
  $\bF_{p}=(F_{1},F_{2},\ldots,F_{p})^{\intercal}\in(\MM_{n,m})^{p}$ then 
  \begin{align*}
    K^{(pn,m)}\cA_{p}^{\vect}K^{-(pn,m)}\vect(\bF_{p}^{\intercal})&=K^{(pn,m)}\vect\big(\big(F_{2},F_{2},\ldots,F_{p},\sum_{i=1}^{p}\opA_{p+1-i}(F_{i})\big)^{\intercal}\big)\\ 
                                                                  &=\vect\Big(\big(F^{\intercal}_{2},F^{\intercal}_{2},\ldots,F^{\intercal}_{p},\sum_{i=1}^{p}\opA_{p+1-i}(F_{i})^{\intercal}\big)\Big)\\
                                                                  &=\big(\vect(F_{2}^{\intercal}),\ldots,\vect(F_{p}^{\intercal}),\sum_{i=1}^{p}\vect(\opA_{p+1-i}(F_{i})^{\intercal})\big)^{\intercal}\\
                                                                  &=\hat{\cA}^{\vect}_{p}(\vect(\bF_{p}^{\intercal})), 
  \end{align*}
  where in the last equation we used that
  $K^{-(n,m)}\vect(F_{i}^{\intercal})=\vect(F_{i})$ and
  \begin{align*}
    \sum_{i=1}^{p}\vect(\opA_{p+1-i}(F_{i})^{\intercal})=\sum_{i=1}^{p}A_{p+1-i}^{\vect}K^{-(n,m)}\vect(F_{i})=\sum_{i=1}^{p}\hat{A}_{p+1-i}^{\vect}\vect(F_{i}). 
  \end{align*}
  This and $\vect((E_{p}\D
  L_{u})^{\intercal})=\hat{E}_{p}^{\vect}\D\vect(L_{u}^{\intercal})$ inserted
  into~\eqref{eq:bk22-proof-vectorized-state-space-model-5} imply
  \begin{align*}
    \vect(Z_{t}^{(i)})&=K^{-(n,m)}(e_{i}^{\intercal}\otimes
    \MI_{nm})\E^{(t-s)\hat{\cA}_{p}^{\vect}}\vect(Z_{s}^{\intercal})\\
    &\quad
      +K^{-(n,m)}(e_{i}^{\intercal}\otimes\MI_{nm})\int_{s}^{t}\E^{(t-u)\hat{\cA}_{p}^{\vect}}\hat{E}_{p}^{\vect}\D\vect(L_{u}^{\intercal}), 
  \end{align*}
  which finally inserted back
  into~\eqref{eq:bk22-proof-vectorized-state-space-model-1}
  proves~\eqref{eq:bk22-vect-X-state-space-rep-1} and
  ~\eqref{eq:bk22-vect-X-state-space-rep-2}.
      
\end{proof}
\end{toappendix}

\begin{remark}
  \begin{enumerate}
  \item[i)]  Note that in order to obtain the correct autoregressive structure
    in~\eqref{eq:bk22-vect-X-state-space-rep-1}, we have to take
    the transpose of the state process $Z_{t}$ for $t\in\MR$. That means we first stack the
    columns of the first block matrix entry $Z_{t}^{(1)}$ below each other, then
    below this real vector of length $nm$ we append the stacked columns of
    $Z_{t}^{(2)}$ and so on until we finally obtain the vector
    $\vect(Z_{t}^{\intercal})=(\vect(Z_{t}^{(1)}),\vect(Z_{t}^{(2)}),\ldots,\vect(Z_{t}^{(p)}))^{\intercal}\in\MR_{pnm}$. 
  \item[ii)] In case of $\MM_{n,1}=\MR_{n}$ the $\vect$-operator is the
    identity and also $\vect(Z_{t}^{\intercal})=\vect(Z_{t})$, i.e. in this
    case Proposition~\ref{prop:bk22-vectorized-state-space-model} is trivial. Note
    further that whenever the state process $(Z_{t})_{t\in\MR}$ takes
    values in $(\MS_{d})^{p}$ driven by a Levy process $L$ with values in
    $\MS_{d}$, then $K^{(d,d)}=\MI_{d^{2}}$ and we see
    that $\hat{A}_{i}^{\vect}=A_{i}^{\vect}$,
    $\hat{C}_{j}^{\vect}=C_{j}^{\vect}$ and
    $\hat{E}_{p}^{\vect}=E_{p}^{\vect}$ for $i=1,\ldots,p$ and
    $j=0,\ldots,q$ and the representations
    in~\eqref{eq:bk22-Ap-vect}-\eqref{eq:bk22-Ep-vect} become considerably
    simpler. Moreover, in this case we could replace $\vect$ by the 
    $\mathrm{vech}$ operation.
    \item[iii)] The controller canonical form for $\MR_{d}$-valued MCARMA
      processes was already used in~\cite{BS13} for estimating the parameters
      of the driving L\'evy process
      $(L_{t})_{t\in\MR}$. Proposition~\ref{prop:bk22-vectorized-state-space-model}
      shows that the results of~\cite{BS13} can be straightforwardly extended to the
      matrix-valued case.
  \end{enumerate}
\end{remark}

Note that the explicit form of the parameter tuple $(\hat{\cA}^{\vect}_{p},\hat{\cC}_{q}^{\vect},\hat{E}_{p}^{\vect},L^{\vect})$
  given by Proposition~\ref{prop:bk22-vectorized-state-space-model} is
  also convenient for the practitioner seeking to transform matrix- to
  vector-valued data, while preserving the correct autoregressive model structure. 
Moreover, from the representation in~\eqref{eq:bk22-vect-X-state-space-rep-1} and
\eqref{eq:bk22-vect-X-state-space-rep-2} we can read off the following second order
structure for the output process $(X_{t})_{t\in\MR}$, see also~\cite{BNS07, FK22}:

\begin{proposition}\label{prop:bk22-second-order-property}
  Let $p\in\MN$ and $q\in\MN_{0}$ such that $q<p$ and assume that
  $(\hat{\cA}^{\vect}_{p},\hat{E}^{\vect}_{p},\hat{\cC}^{\vect}_{q},L^{\vect})$
  is as in~\eqref{eq:bk22-Ap-vect}-\eqref{eq:bk22-Ep-vect} with $L$ being
  square-integrable and where we denote the covariance operator of
  $(L^{\vect}_{t})_{t\in\MR}$ by $\mathcal{Q}^{\vect}$. Then the absolute
  second moment of the process $(\vect(X_{t}))_{t\in\MR}$, given
  by~\eqref{eq:bk22-vect-X-state-space-rep-2}, exists and we have  
\begin{align}
  \label{eq:bk22-X-variance}
  \Var\left[\vect(X_{t})|\cF_s\right]=\hat{\cC}^{\vect}_{q} \Sigma^{\vect}_{t,s}
  (\hat{\cC}^{\vect}_{q})^{\intercal},\quad s<t\in\MR, 
\end{align}
where
\begin{align*}
\Sigma^{\vect}_{t,s}\df\int_{s}^{t}\E^{u\hat{\cA}_{p}^{\vect}}\hat{E}_{p}^{\vect}\mathcal{Q}^{\vect}\big(\hat{E}_{p}^{\vect}\big)^{\intercal}\E^{u(\hat{\cA}_{p}^{\vect})^{\intercal}}\D
  u.  
\end{align*}
Moreover, for every $h\geq 0$ the auto-covariance of
$(\vect(X_{t}))_{t\in\MR}$ satisfies
\begin{align}\label{eq:bk22-auto-covariance}
\Cov\left[\vect(X_{t+h}),\vect(X_{t})\lvert \cF_{s}\right]=\hat{\cC}^{\vect}_{q} \E^{h
  \hat{\cA}_{p}^{\vect}}\Sigma^{\vect}_{t,s}(\hat{\cC}^{\vect}_{q})^{\intercal},\quad
  s<t\in\MR,\, h\geq 0.
\end{align}
If in addition $(X_{t})_{t\in\MR}$ is stable and given by~\eqref{eq:bk22-X-stationary}, then
\begin{align}
  \label{eq:bk22-X-variance-stationary}
  \Var\left[\vect(X_{t})\right]=\hat{\cC}_{q}^{\vect} \Sigma^{\vect}_{\infty}(\hat{\cC}_{q}^{\vect})^{\intercal}\quad \forall t\in\MR,
\end{align}
where $\Sigma^{\vect}_{\infty}\df\int_0^{\infty}\E^{u\hat{\cA}_{p}^{\vect}}\hat{E}_{p}^{\vect}\mathcal{Q}^{\vect}\big(\hat{E}_{p}^{\vect}\big)^{\intercal}\E^{u(\hat{\cA}_{p}^{\vect})^{\intercal}}\D u$
and the auto-covariance is
\begin{align}\label{eq:bk22-autocovariance-stationary}
\Cov\left[\vect(X_{t+h}),\vect(X_{t})\right]=\hat{\cC}_{q}^{\vect}\E^{h
  \hat{\cA}_{p}^{\vect}}\Sigma_{\infty}(\hat{\cC}_{q}^{\vect})^{\intercal},\quad
  t\in\MR,\, h\geq 0.
\end{align} 
\end{proposition}

\vspace{3mm}

Proposition~\ref{prop:bk22-vectorized-state-space-model} tells us that the process
$(\vect(X_{t}))_{t\in\MR}$ can be interpreted as the output process of
an $\MR_{nm}$-valued continuous-time state space model in a controller
canonical form entrywise composited with the commutation matrix $K^{-(n,m)}$. It
is well known that every $\MR_{nm}$-valued MCARMA process possesses a state space
representation. Conversely, the result in \cite[Theorem 3.3]{SS12} describes
precisely those state space specifications such that the associated output
process gives rise to a (causal) MCARMA process. We introduce the following
central notion: We call the function $H\colon \MC\to \cL(\MM_{n,m})$ given by
\begin{align}\label{eq:bk22-transfer-function}
  H(\lambda)\df\cC(\lambda\opI_{p}-\cA)^{-1}\cB,\quad\lambda\in\MC,  
\end{align}
the \emph{transfer function} of the continuous-time linear state space model
associated with $(\cA,\cB,\cC,L)$. Note that the transfer
  function $H$ uniquely identifies an associated MCARMA process, see e.g.~\cite{SS12},
  and is therefore the natural object to work with. In the next lemma, we derive
  a right matrix fraction of the transfer function $H$ associated with the tuple $(\hat{\cA}^{\vect}_{p},\hat{E}^{\vect}_{p},\hat{\cC}^{\vect}_{q},L^{\vect})$
given by Proposition~\ref{prop:bk22-vectorized-state-space-model}.  
 
\begin{lemma}\label{lem:bkk22-right-matrix-description}
  Let $p\in\MN$ and $q\in\MN_{0}$ such that $q<p$ and let
  $(\hat{\cA}^{\vect}_{p},\hat{E}^{\vect}_{p},\hat{\cC}^{\vect}_{q},L^{\vect})$ be as
  in Proposition~\ref{prop:bk22-vectorized-state-space-model}. Then for every
  $\lambda\in\MC$ we have
  \begin{align}\label{eq:bk22-right-matrix-fraction-describtion}
    \hat{\cC}^{\vect}_{q}(\lambda \MI_{pnm}-\hat{\cA}^{\vect}_{p})^{-1}\hat{E}^{\vect}_{p}=\hat{Q}(\lambda)\hat{P}(\lambda)^{-1}, 
  \end{align}
  where $\hat{Q},\hat{P}\in \MM_{nm}(\MR[\lambda])$ for $\lambda\in\MC$ are given by 
  \begin{align}
    \label{eq:bk22-Q-vect}
    \hat{Q}(\lambda)\df
    \hat{C}^{\vect}_{0}+\hat{C}^{\vect}_{1}K^{(n,m)}\lambda +\hat{C}^{\vect}_{2}(K^{(n,m)}\lambda)^{2}+\ldots+\hat{C}^{\vect}_{q}(K^{(n,m)}\lambda)^{q}, 
  \end{align}
  and
  \begin{align}
    \label{eq:bk22-P-vect}
    \hat{P}(\lambda)\df
    (K^{(n,m)}\lambda)^{p}-A^{\vect}_{1}(
    K^{(n,m)}\lambda)^{p-1}-A^{\vect}_{2}(K^{(n,m)}\lambda)^{p-2}-\ldots-A^{\vect}_{p}.
\end{align}  
\end{lemma}
\begin{toappendix}
  \begin{proof}[Proof of Lemma~\ref{lem:bkk22-right-matrix-description}]
    Let $\lambda\in\MC$ and $F\df(F_{1},F_{2},\ldots,F_{p})^{\intercal}\in
    \cL(\MR_{nm},\MR_{pnm})$ with $F_{i}\in \MM_{nm}$ for all
    $i=1,\ldots,p$ and such that
    $F(x)=(F_{1}x,F_{2}x,\ldots,F_{p}x)^{\intercal}$ for all $x\in\MR_{nm}$.
    We solve the matrix equation $(\lambda
    \MI_{pnm}-\hat{\cA}^{\vect}_{p})F=\hat{E}^{\vect}_{p}$, where the left-hand
    side equals
    \begin{align*}
      (\lambda F_{1}-K^{-(n,m)}F_{2},\ldots,\lambda F_{p-1}-K^{-(n,m)}F_{p},
      \lambda F_{p}-\hat{A}^{\vect}_{p}F_{1}-\ldots-\hat{A}^{\vect}_{1}F_{p})^{\intercal}.
    \end{align*}
    Setting this equal to $(\zero_{nm},\ldots,\zero_{nm},K^{-(n,m)})^{\intercal}$
    and solving for $F$ yields
    \begin{align}\label{eq:bk22-lemma-eq-1}
      F_{i}=(\lambda K^{(n,m)})^{-(p-i)}F_{p},\quad\text{for } i=1,\ldots,p-1,
    \end{align}
    and inserting this into the last equation we see that the term
    \begin{align*}
      \lambda F_{p}-\sum_{i=1}^{p}\hat{A}^{\vect}_{p+1-i}F_{i}&=
                                                                \Big(\lambda^{p} (K^{(n,m)})^{p-1}\!-\sum_{i=1}^{p}\hat{A}^{\vect}_{i}(\lambda
                                                       K^{(n,m)})^{p-i}\Big)(\lambda
                                                                K^{(n,m)})^{1-p}F_{p},
    \end{align*}
 must coincide with $K^{-(n,m)}$, which by definition of $\hat{P}(\lambda)$ is equivalent to $F_{p}=(\lambda
 K^{(n,m)})^{p-1}\hat{P}(\lambda)^{-1}$. Hence by~\eqref{eq:bk22-lemma-eq-1} we see
 that $F_{i}=(\lambda K^{(n,m)})^{i-1}\hat{P}(\lambda)^{-1}$, i.e.
 \begin{align*}
   F=(\hat{P}(\lambda)^{-1},\lambda K^{(n,m)}
   \hat{P}(\lambda)^{-1},\ldots,(\lambda
   K^{(n,m)})^{p-1}\hat{P}(\lambda)^{-1})^{\intercal}.  
 \end{align*}
 From this and the definition of $\hat{Q}$ we conclude that
 \begin{align*}
   \hat{\cC}_{q}^{\vect}(\lambda
   \MI_{pnm}-\hat{\cA}^{\vect}_{p})^{-1}\hat{E}^{\vect}_{p}=\hat{\cC}_{q}^{\vect}F=\hat{Q}(\lambda)\hat{P}(\lambda)^{-1},\quad \lambda\in\MC, 
 \end{align*}
 which proves~\eqref{eq:bk22-right-matrix-fraction-describtion}.
\end{proof}
\end{toappendix}

The following proposition is the key result for the definition of matrix-valued
MCARMA processes.

\begin{proposition}\label{prop:bk22-state-space-MCARMA}
  Let $p\in\MN$ and $q\in\MN_{0}$ such that $q<p$ and assume that
  $(\hat{\cA}^{\vect}_{p},\hat{E}^{\vect}_{p},\hat{\cC}^{\vect}_{q},L^{\vect})$
  is as in Proposition~\ref{prop:bk22-vectorized-state-space-model}. Moreover, let
  $\hat{Q}(\lambda)$ and $\hat{P}(\lambda)$ be given by~\eqref{eq:bk22-Q-vect}
  and~\eqref{eq:bk22-P-vect}, respectively. Then there exist
  two matrix polynomials $\tilde{Q},\tilde{P}\in\MM_{nm}(\MR[\lambda])$ such that
  \begin{align}
    \tilde{Q}(\lambda)=\tilde{C}_{0}\lambda^{q}+\tilde{C}_{1}\lambda^{q-1}+\ldots+\tilde{C}_{q},\quad
    \lambda\in\MC,  
  \end{align}
  with $\tilde{C}_{j}\in\MM_{nm}$ for $j=0,\ldots,q$ and
  \begin{align}
    \tilde{P}(\lambda)=\MI_{nm}\lambda^{p}-\tilde{A}_{1}\lambda^{p-1}-\ldots-\tilde{A}_{p},\quad \lambda\in\MC,  
  \end{align}
  with $\tilde{A}_{i}\in\MM_{nm}$ for $i=1,\ldots,p$ satisfying
  \begin{align}\label{eq:bk22-left-right-matrix-fraction-describtion}
    \tilde{P}(\lambda)^{-1}\tilde{Q}(\lambda)=\hat{Q}(\lambda)\hat{P}(\lambda)^{-1}\quad\text{for
    all }\lambda\in\MC,
  \end{align}
  and $\det(\tilde{P}(\lambda))=0$ if and only if
  $\det(\hat{P}(\lambda))=0$. If moreover, the L\'evy process
  $(L^{\mathrm{vec}}_{t})_{t\in\MR}$ satisfies
  $\EX{\log(\norm{L_{1}^{\vect}}_{nm})}<\infty$ and $\hat{P}$ satisfies
  \begin{align}\label{eq:bk22-stationary-solution-condition}
  \set{\lambda \in\MC\colon
    \det(\hat{P}(\lambda))=0}\subseteq \MR\setminus\set{0}+\I\MR,  
  \end{align}
  then equation~\eqref{eq:bk22-vect-X-state-space-rep-2} has a stationary solution,
  unique in law, given by 
  \begin{align}\label{eq:bk22-stationary-solution}
    \vect(X_{t})=\int_{-\infty}^{\infty}g(t-s)\D L^{\vect}_{s},\quad t\in\MR,
  \end{align}
  where $g(t)\df\frac{1}{2\pi}\int_{-\infty}^{\infty}\E^{\I \xi
    t}\tilde{P}(\I \xi)^{-1}\tilde{Q}(\I\xi)\D \xi$ for all $t\in\MR$.
\end{proposition}
\begin{toappendix}
  \begin{proof}[Proof of Proposition~\ref{prop:bk22-state-space-MCARMA}]
    The existence of $\tilde{Q},\tilde{P}\in\MM_{nm}(\MR[\lambda])$ such
    that~\eqref{eq:bk22-left-right-matrix-fraction-describtion} holds with
  $\det(\tilde{P}(\lambda))=0$ if and only if $\det(\hat{P}(\lambda))=0$
  follows immediately from~\cite[Lemma 6.3-8]{Kai80} and
  Lemma~\ref{lem:bkk22-right-matrix-description}. From~\eqref{eq:bk22-right-matrix-fraction-describtion}
  it then follows that $(\tilde{P},\tilde{Q})$ is a left-matrix fraction
  description of the transfer function
  $H(\lambda)=\hat{\cC}^{\vect}_{q}(\lambda\opI_{pnm}-\hat{\cA}^{\vect}_{p})^{-1}\hat{E}^{\vect}_{p}$,
  i.e. 
  \begin{align}\label{eq:bk22-state-space-MCARMA-1}
   \hat{\cC}^{\vect}_{q}(\lambda
    \MI_{pnm}-\hat{\cA}^{\vect}_{p})^{-1}\hat{E}^{\vect}_{p}=\tilde{P}(\lambda)^{-1}\tilde{Q}(\lambda),\quad
    \forall \lambda\in\MC.
  \end{align}

  We define the two kernels $g_{1},g_{2}\colon\MR\to\cL(\MR_{nm},\MR_{nmp})$ by
  \begin{align}\label{eq:bk22-kernels-g}
   g_{i}(t)\df \frac{1}{2\pi \I}\int_{\rho_{i}}\E^{\lambda t}(\lambda
    \MI_{pnm}-\hat{\cA}^{\vect}_{p})^{-1}\hat{E}^{\vect}_{p}\D \lambda,\quad
    t\in\MR\text{ and } i=1,2,  
  \end{align}
  where we integrate anti-clockwise over simple closed curves $\rho_{1}$ and
  $\rho_{2}$ in the open left- and right-half plane of the complex field
  encircling the zeroes of the map $\lambda\mapsto(\lambda
  \MI_{pnm}-\hat{\cA}^{\vect}_{p})^{-1}\hat{E}^{\vect}_{p}$. More
  specifically, let $\rho_{1}$ be the rectangle in the left-half plane with
  width $M$ and height $2R$ (with $M$ and $R$ large enough such that all
  eigenvalues in the left-half plain are encircled) and such that the line segment
  \begin{align*}
  \rho_{\I R}\df\set{\lambda\in\MC\colon \Re(\lambda)=0\text{ and }|\Im(\lambda)|\leq R},
  \end{align*}
  forms an edge of $\rho_{1}$. We then define the curve $\rho_{2}$ as the
  reflection of $\rho_{1}$ over the imaginary axis. Moreover, for $i=1,2$ we
  denote by $\hat{\rho}_{i}$ the curve $\rho_{i}$ without the line segment on the imaginary axis.\par{}
  The complex integrals in~\eqref{eq:bk22-kernels-g} are well defined, since by
  assumption~\eqref{eq:bk22-stationary-solution-condition} there is no singularity of
  $\lambda\mapsto(\lambda
  \MI_{pnm}-\hat{\cA}^{\vect}_{p})^{-1}\hat{E}^{\vect}_{p}=(\MI_{nm},\lambda
  \MI_{nm},\ldots,\lambda^{p}\MI_{nm})^{\intercal}\tilde{P}(\lambda)^{-1}$ on the
  imaginary axis. For every $t\in\MR$ we set 
    \begin{align}\label{eq:bk22-Z-stationary-solution}
      Z_{t}\df\int_{-\infty}^{t}g_{1}(t-u)\D L_{u}-\int_{t}^{\infty}g_{2}(t-u)\D L_{u},  
    \end{align}
    and show in the following that the integrals over the kernels $g_{1}$ and
    $g_{2}$ are well defined as the limit of integrals
    $\lim_{T\to\infty}\int_{-T}^{t}g_{1}(t-u)\D L_{u}$ and
    $\lim_{T\to\infty}\int_{t}^{T}g_{2}(t-u)\D L_{u}$, respectively. Moreover, we show
    that $(Z_{t})_{t\in\MR}$ is the unique stationary solution 
    of~\eqref{eq:bk22-vect-X-state-space-rep-1}. First, note that for $i=1,2$ the
    kernel $g_{i}$ satisfies the equation $\frac{\D}{\D
      t}g_{i}(t)=\hat{\cA}^{\vect}_{p}g_{i}(t)$, which can be seen by similar
    arguments as in Lemma~\ref{lem:bkk22-right-matrix-description}. Indeed, note
    that we have
    \begin{align*}
      \hat{\cA}^{\vect}_{p}g_{i}(t)=\frac{1}{2\pi \I}\int_{\rho_{i}}\E^{\lambda
      t}(\lambda \MI_{pnm}-\hat{\cA}^{\vect}_{p})^{-1}\hat{\cA}^{\vect}_{p}\hat{E}^{\vect}_{p}\D \lambda,
    \end{align*}
    and the term $(\lambda
    \MI_{pnm}-\hat{\cA}^{\vect}_{p})^{-1}\hat{\cA}^{\vect}_{p}\hat{E}^{\vect}_{p}$
    can be computed by solving the following linear matrix equation
    \begin{align*}
      (\lambda \MI_{pnm}-\hat{\cA}^{\vect}_{p})^{-1}F=\hat{\cA}^{\vect}_{p}\hat{E}^{\vect}_{p},  
    \end{align*}
    for $F\df(F_{1},F_{2},\ldots,F_{p})^{\intercal}\in\cL(\MR_{nm},\MR_{pnm})$
    with $F_{i}\in \MM_{nm}$ for all $i=1,\ldots,p$. Since
    $\hat{\cA}^{\vect}_{p}\hat{E}^{\vect}_{p}=(\zero_{nm},\ldots,\zero_{nm},K^{-(n,m)},\hat{A}_{1}^{\vect})$
    we can argue similarly to the proof of
    Lemma~\ref{lem:bkk22-right-matrix-description} and obtain  
    \begin{align*}
      F=(\lambda\MI_{nm},\ldots,\lambda^{p+1}(K^{(n,m)})^{p}\MI_{nm})^{\intercal}\hat{P}(\lambda)^{-1}-(\zero_{nm},\ldots,\zero_{nm},K^{(n,m)})^{\intercal}.
    \end{align*}
    Note that by integrating over the closed curves $\rho_{i}$ for $i=1,2$, the
    integral over the latter term vanishes and hence for all $t\in\MR$ we obtain
  \begin{align*}
    \frac{\D}{\D t}g_{i}(t)&=\frac{1}{2\pi \I}\int_{\rho_{i}}\E^{\lambda
    t}(\lambda\MI_{nm},\lambda^{2}K^{(n,m)},\ldots,\lambda^{p+1}(K^{(n,m)})^{p}\MI_{nm})^{\intercal}\hat{P}(\lambda)^{-1}\D
    \lambda\\
    &=\hat{\cA}^{\vect}_{p}g_{i}(t),\quad i=1,2.
  \end{align*}
  Since the homogeneous linear equation $\frac{\D}{\D
    t}v(t)=\hat{\cA}^{\vect}_{p}v(t)$ is uniquely solved by
  $\E^{t\hat{\cA}^{\vect}_{p}}v_{0}$, we see that
  $g_{i}(t)=\E^{t\hat{\cA}^{\vect}_{p}}g_{i}(0)$ for $i=1,2$ and
  $t\in\MR$. This has the following consequences: First, it follows that there
  exist $K>0$ and $\delta>0$ such that for all $u\leq 0$ we have
  $\norm{g_{1}(-u)}_{\cL(\MM_{n,m},(\MM_{n,m})^{p})}\leq K \E^{-\delta |u|}$
  and for all $u\geq 0$ we have
  $\norm{g_{2}(-u)}_{\cL(\MM_{n,m},(\MM_{n,m})^{p})}\leq K \E^{-\delta
    |u|}$. This together with $\EX{\log(\norm{L_{1}}_{nm})}<\infty$ implies
  the existence of the integrals  $\int_{-\infty}^{t}g_{1}(t-u)\D L_{u}$ and
  $\int_{t}^{\infty}g_{2}(t-u)\D L_{u}$, respectively, as limits of integrals over the
  intervals $(-T,t]$, resp. $[t,T)$, for $T\to\infty$, see
  also~\cite{CM87}.\par{} Next, we show that $(Z_{t})_{t\in\MR}$ is a solution
  of~\eqref{eq:bk22-vect-X-state-space-rep-1}, where as a second consequence from
  $g_{i}(t)=\E^{t\hat{\cA}^{\vect}_{p}}g_{i}(0)$ for $i=1,2$ and $t\in\MR$ we
  conclude that for every $s<t\in\MR$ the following equality holds true:
  \begin{align}\label{eq:bk22-state-space-MCARMA-2}
    \E^{(t-s)\hat{\cA}_{p}^{\vect}}Z_{s}&=\E^{(t-s)\hat{\cA}_{p}^{\vect}}\Big(\int_{-\infty}^{s}g_{1}(s-u)\D
                                          L_{u} -\int_{s}^{\infty}g_{2}(s-u)\D
                                          L_{u} \Big)\nonumber\\
    &= \E^{t\hat{\cA}_{p}^{\vect}}\Big(\int_{-\infty}^{s}g_{1}(-u)\D
      L_{u}-\int_{s}^{\infty}g_{2}(-u)\D L_{u}\Big).
  \end{align}
  Now, by setting $\rho=\rho_{1}+\rho_{2}$, the spectral representation of the
  matrix exponential $\E^{t\hat{\cA}^{\vect}_{p}}$, see
  e.g.~\cite[Theorem 17.5]{Lax02}, yields
  \begin{align*}
   \E^{t\hat{\cA}^{\vect}_{p}}\hat{E}^{\vect}_{p}=\frac{1}{2\pi\I}\int_{\rho}\E^{\lambda
    t}(\lambda\MI_{pnm}-\hat{\cA}^{\vect}_{p})^{-1}\hat{E}^{\vect}_{p}\D \lambda,\quad \forall t\in\MR,
  \end{align*}
  
  and hence for every $t\in\MR$ we have
  \begin{align}\label{eq:bk22-state-space-MCARMA-3}
    \int_{s}^{t}\E^{(t-u)\hat{\cA}_{p}^{\vect}}\hat{E}^{\vect}_{p}\D L_{u}=
    \E^{t\hat{\cA}_{p}^{\vect}}\Big(\int_{s}^{t}g_{1}(-u)\D
    L_{u}+\int_{s}^{t}g_{2}(-u)\D L_{u}\Big). 
  \end{align}
  By summing up~\eqref{eq:bk22-state-space-MCARMA-2}
  and~\eqref{eq:bk22-state-space-MCARMA-3} we obtain 
  \begin{align*}
  \E^{t\hat{\cA}_{p}^{\vect}}\left(\int_{-\infty}^{t}g_{1}(-u)\D
         L_{u}-\int_{t}^{\infty}g_{2}(-u)\D L_{u}\right)=Z_{t},
  \end{align*}
  which proves that $(Z_{t})_{t\in\MR}$ is a solution
  of~\eqref{eq:bk22-vect-X-state-space-rep-1}. Moreover, it is easy to see
  that $(Z_{t})_{t\in\MR}$ is also stationary and unique in law. Now, let us
  set $Y_{t}\df \hat{\cC}_{q}^{\vect}(Z_{t})$ and $h(t)\df
  g_{1}(t)\one_{[0,\infty)}(t)-g_{2}(t)\one_{(-\infty,0)}(t)$ for all $t\in\MR$, then
  \begin{align*}
    Y_{t}=\int\limits_{-\infty}^{\infty}\!\hat{\cC}_{q}^{\vect}h(t-u)\D L_{u}=\int\limits_{-\infty}^{\infty}\!\big(\frac{1}{2\pi\I}\int_{\rho}\!\E^{\lambda(t-u)}\hat{\cC}_{q}^{\vect}(\lambda\MI_{nm}-\hat{\cA}_{p}^{\vect})^{-1}\hat{E}_{p}^{\vect}\D
           \lambda\big)\D L_{u},
  \end{align*}
  which by~\eqref{eq:bk22-state-space-MCARMA-1} equals
  \begin{align}\label{eq:bk22-complex-integral-2}
    Y_{t}=\int_{-\infty}^{\infty}\Big(\frac{1}{2\pi\I}\int_{\rho}\E^{\lambda(t-u)}\tilde{P}^{-1}(\lambda)\tilde{Q}(\lambda)\D
    \lambda\Big) \D L_{u},\quad t\in\MR,
  \end{align}
  and if
  \begin{align}\label{eq:bk22-complex-integral}
    \int_{-\infty}^{\infty}\Big(\frac{1}{2\pi\I}\int_{\rho}\E^{\lambda(t-u)}\tilde{P}^{-1}(\lambda)\tilde{Q}(\lambda)\D\lambda\Big)=\int_{-\infty}^{\infty}g(t-u)\D
    L_{u},\quad t\in\MR,
  \end{align}
  then by uniqueness we conclude that $\vect(X_{t})=Y_{t}$ for $t\in\MR$ and
  that the representation~\eqref{eq:bk22-stationary-solution} holds true. It is
  therefore left to prove that the identity in~\eqref{eq:bk22-complex-integral}
  holds true for every $t\in\MR$. In order to prove this, we set $K(z,t)\df\E^{-z
    t}\tilde{P}(z)^{-1}\tilde{Q}(z)$ and use the integration paths from above,
  i.e. $\rho_{i}=\hat{\rho}_{i}\pm\rho_{\I R}$.\par{} We then compute the left
  hand-side in~\eqref{eq:bk22-complex-integral} as follows:
  \begin{align*}
    h(t)&=g_{1}(t)\one_{[0,\infty)}(t)-g_{2}(t)\one_{(-\infty,0)}(t)\\
        &=\frac{1}{2\pi\I}\left(\int_{\rho_{1}}K(z,t)\D
          z\one_{[0,\infty)}(t)-\int_{\rho_{2}}K(z,t)\D z\one_{(-\infty,0)}(t)
          \right)\\
        &=\frac{1}{2\pi\I}\left(\int_{\hat{\rho}_{1}}K(z,t)\D
          z\one_{[0,\infty)}(t)-\int_{\hat{\rho}_{2}}K(z,t)\D z\one_{(-\infty,0)}(t)
          \right)\\
        &\quad+\frac{1}{2\pi\I}\left(\int_{-\I R}^{\I R}K(z,t)\D
          z\one_{[0,\infty)}(t)-\int_{\I R}^{-\I R}K(z,t)\D z\one_{(-\infty,0)}(t)
          \right)\\
        &=\frac{1}{2\pi\I}\left(\int_{\hat{\rho}_{1}}K(z,t)\D
          z\one_{[0,\infty)}(t)-\int_{\hat{\rho}_{2}}K(z,t)\D z\one_{(-\infty,0)}(t)
          \right)\\
        &\quad+\frac{1}{2\pi}\left(\int_{-R}^{R}K(\I\xi,t)\D
          \xi\one_{[0,\infty)}(t)+\int_{-R}^{R}K(\I\xi,t)\D \xi\one_{(-\infty,0)}(t)
          \right)\\
        &=\frac{1}{2\pi}\int_{-R}^{R}\!K(\I\xi,t)\D
          \xi+\frac{1}{2\pi\I}\big(\int\limits_{\hat{\rho}_{1}}\!K(z,t)\D
          z\one_{[0,\infty)}(t)-\int\limits_{\hat{\rho}_{2}}\!K(z,t)\D z\one_{(-\infty,0)}(t)
          \big).
  \end{align*}
  Now by letting $R\to\infty$, we see that for every $t\in\MR{}$ the first
  term converges to $g(t)=\frac{1}{2\pi}\int_{-\infty}^{\infty}K(\I\xi,t)\D
  \xi$ and it remains to show that the latter term converges to zero. For this, we first
  consider the term $\int_{\hat{\rho}_{1}}K(z,t)\D z\one_{[0,\infty)}(t)$ for
  $t\geq 0$ and note that the integral over $\hat{\rho}_{1}$ can
  be split into three separate integrals: The first is
  \begin{align*}
    \int_{\I R-M}^{-\I R-M}\E^{\lambda
    t}\tilde{P}(\lambda)^{-1}\tilde{Q}(\lambda)\D\lambda=\int_{-R}^{R}\E^{-\I \xi t}\E^{-t
    M}\tilde{P}(-\I\xi-M)^{-1}\tilde{Q}(-\I\xi-M)\D\xi,
  \end{align*}
  where the term $\E^{-t M}$ in the integral dictates the convergence to zero as
  $M\to\infty$ for arbitrary $R$. The two other terms are
  given by:
  \begin{align*}
    \int_{\I R}^{\I R-M}\E^{\lambda
    t}\tilde{P}(\lambda)^{-1}\tilde{Q}(\lambda)\D\lambda=-\int_{0}^{M}\E^{(\I R-
    \zeta)t}\tilde{P}(\I R- \zeta)^{-1}\tilde{Q}(\I R- \zeta)\D\zeta,\quad \text{and}\\
    \int_{-\I R-M}^{-\I M}\E^{\lambda
    t}\tilde{P}(\lambda)^{-1}\tilde{Q}(\lambda)\D\lambda=\int_{0}^{-M}\E^{(-\I R+\zeta)t}\tilde{P}(-\I R+\zeta)^{-1}\tilde{Q}(-\I R+\zeta)\D\zeta,
  \end{align*}
  where in both cases the integrals exists for arbitrary large $M$ and due to
  the fact that
  $p>q$ implies that $\tilde{P}(-\I R+\zeta)^{-1}\tilde{Q}(-\I R+\zeta)\to 0$
  as $R\to\infty$, we conclude that both integrals converge to zero as
  $R\to\infty$. Similarly, for the integral over $\hat{\rho}_{2}$ and all
  $t<0$, we see that
  \begin{align*}
    \int_{-\I R+M}^{\I R+M}\E^{\lambda
    t}\tilde{P}(\lambda)^{-1}\tilde{Q}(\lambda)\D\lambda=\int_{-R}^{R}\E^{-\I \xi t}\E^{t
    M}\tilde{P}(-\I\xi-M)^{-1}\tilde{Q}(-\I\xi-M)\D\xi,
  \end{align*}
  which, as before, converges to zero as $M\to\infty$ for every $R\geq 0$ since
  $t<0$ and the matrix exponential is the dominating term. For the remaining
  parts, we have  
  \begin{align*}
    \int_{\I R+M}^{\I R}\E^{\lambda 
    t}\tilde{P}(\lambda)^{-1}\tilde{Q}(\lambda)\D\lambda=-\int_{0}^{-M}\E^{(\I R- \zeta)t}\tilde{P}(\I R- \zeta)^{-1}\tilde{Q}(\I R- \zeta)\D\zeta,
  \end{align*}
  and
  \begin{align*}
    \int_{-\I R}^{-\I R+M}\E^{\lambda
    t}\tilde{P}(\lambda)^{-1}\tilde{Q}(\lambda)\D\lambda=-\int_{0}^{M}\E^{-(\I
    R-\zeta)t}\tilde{P}(\zeta-\I R)^{-1}\tilde{Q}(\zeta-\I R)\D\zeta,
  \end{align*}
  where we see that for every $M>0$ and $t<0$, since $p>q$ the integrals converge to
  zero, whenever $R\to\infty$. Thus the only remaining term of $h(t)$ when
  expanding the integration domain is
  $g(t)=\frac{1}{2\pi}\int_{-\infty}^{\infty}K(\I\xi,t)\D \xi$ which
  proves~\eqref{eq:bk22-complex-integral} and we conclude the assertions of
  Proposition~\ref{prop:bk22-state-space-MCARMA}.  
\end{proof}
\end{toappendix}

Following~\cite[Theorem 3.22]{MS07}, every $\MR_{nm}$-valued MCARMA process
$(Y_{t})_{t\in\MR}$ with \emph{moving-average polynomial matrix}
$\check{Q}\in\MM_{nm}(\MR[\lambda])$, \emph{autoregressive polynomial matrix}
$\check{P}\in\MM_{nm}(\MR[\lambda])$ and input L\'evy process $\check{L}$ on
$\MR_{nm}$ is given by 
\begin{align*}
  Y_{t}=\frac{1}{2\pi}\int_{-\infty}^{\infty}\int_{-\infty}^{\infty}\E^{\I \xi
  (t-s)}\check{P}(\I\xi)^{-1}\check{Q}(\I\xi)\D \xi\D \check{L}_{s},\quad t\in\MR, 
\end{align*}
where $\check{P}$ satisfies $\set{\lambda\in\MC\colon\,
  \det(\check{P}(\lambda))=0}\subseteq \MR\setminus\set{0}+\I\MR$ and
$\check{L}$ is such that $\EX{\log(\norm{\check{L}_{1}}_{nm})}<\infty$. It thus
follows from Proposition~\ref{prop:bk22-state-space-MCARMA}, that the unique stationary process
$(\vect(X_{t}))_{t\geq 0}$ in~\eqref{eq:bk22-stationary-solution} is an
$\MR_{mn}$-valued MCARMA process with moving-average polynomial matrix
$\tilde{Q}$, autoregressive polynomial matrix $\tilde{P}$ and L\'evy noise
$L^{\vect}$. Moreover,
by~\cite[Remark 3.19]{MS07} it can be interpreted as the solution of the
higher order stochastic differential
equation~\eqref{eq:bk22-MCARMA-higher-order-SDE} with $L$ replaced by
$L^{\vect}$. This justifies the following definition of a matrix-valued multivariate continuous-time autoregressive
moving-average process (by means of transformed $\MR_{nm}$-valued MCARMA processes):  

\begin{definition}\label{def:bk22-CARMA-p-q}
  Let $p\in\MN$ and $q\in\MN_{0}$ such that $q<p$ and
  $(\cA_{p},E_{p},\cC_{q},L)$ be as in~\eqref{eq:bk22-Ap}-\eqref{eq:bk22-Ep}. If
  moreover $\EX{\log(\norm{L_{1}}_{n,m})}<\infty$
  and~\eqref{eq:bk22-stationary-solution-condition} is satisfied, then we call 
  the unique output process $(X_{t})_{t\in\MR}$ in~\eqref{eq:bk22-X-controller} for which
  $(\vect(X_{t}))_{t\in\MR}$ is the stationary solution to~\eqref{eq:bk22-vect-X-state-space-rep-2}, an \emph{$\MM_{n,m}$-valued continuous-time
    autoregressive moving-average (MCARMA) process of order $(p,q)$}. In the
  special case where $q=0$ and $C_{0}=\opI$,
  i.e. $\cC_{0}=[\opI,\op0,\ldots,\op0]$, we call $(X_{t})_{t\in\MR}$
  an \emph{$\MM_{n,m}$-valued MCAR process of order $p$} instead. Moreover, if
  in the situation above we have $\tau(\cA_{p})<0$, or equivalently
  \begin{align*}
    \set{\lambda\in\MC\colon\, \det(\hat{P}(\lambda))=0}\subset (-\infty,0)+\I\MR, 
  \end{align*}
  then we say that $(X_{t})_{t\in\MR}$ is a \emph{causal} MCAR(MA) process of
  order $(p,q)$ (resp. $p$). Otherwise, if $\tau(\cA_{p})\geq
  0$ we say that $(X_{t})_{t\in\MR}$ is \emph{non-causal}. 
\end{definition}

\begin{remark}\label{rem:MA-and-AR-polynomials}
  \begin{enumerate}
  \item[i)] Note that by Definition~\ref{def:bk22-CARMA-p-q}
      the state space representation of a matrix-valued MCARMA process is
      given in controller canonical form. This differs from the
      definition of the state space
      representation of the classical MCARMA in~\cite{MS07}, which is given in
      the \emph{observer canocial form}, see also~\cite{SS12}. However, it
      follows from~\cite[Theorem 3.2]{BS13} that every MCARMA process also
      possesses a representation in controller canonical form. 
  \item[ii)] By Definition~\ref{def:bk22-CARMA-p-q} we see that every causal
    MCARMA process with values in $\MM_{n,m}$ is the output process of a stable
    linear state space model. In contrast, non-causal MCARMA processes
    correspond to non-stable linear state space models. Moreover, it can be seen from the
    representation~\eqref{eq:bk22-stationary-solution} (see also the
    decomposition~\eqref{eq:bk22-Z-stationary-solution} in the proof of
    Proposition~\ref{prop:bk22-state-space-MCARMA}), that non-causal MCARMA processes are not adapted to the
    natural filtration $\MF^{L}$, since for $t\in\MR$, $X_{t}$ depends on
    $\sigma(L_{s}\colon s>t)$.  
  \item[iii)] We want to emphasize here that the vectorized versions $\hat{P}$
    and $\hat{Q}$ of the operator polynomials $\opP$ and $\opQ$
    in~\eqref{eq:bk22-P} and~\eqref{eq:bk22-Q} can in general not be interpreted as the
    moving-average, respectively, autoregressive polynomials of the MCARMA
    process. Instead, the polynomials $\tilde{Q}$ and $\tilde{P}$
    in~\eqref{eq:bk22-left-right-matrix-fraction-describtion} can be naturally
    interpreted as such, i.e. by means of
    equation~\eqref{eq:bk22-MCARMA-higher-order-SDE}. If, however, $\hat{P}$ and
    $\hat{Q}$ commute, then $\hat{P}(\lambda)=\tilde{P}(\lambda)$ and
    $\hat{Q}(\lambda)=\tilde{Q}(\lambda)$ for all $\lambda\in\MC$.     
  \end{enumerate}  
\end{remark}

\section{Cone valued MCARMA processes}\label{sec:bk22-posit-mcarma-proc}

In this section we study \emph{cone-invariance} of $\MM_{n,m}$-valued MCARMA
processes of order $(p,q)$, as they were defined in
Definition~\ref{def:bk22-CARMA-p-q} above. Our main examples for multivariate
cones are $\MR_{d}^{+}$ and $\MS_{d}^{+}$ for $d\in\MN$, where in both cases we
sometimes speak of \emph{positive} processes (instead of cone valued). This
section is divided into the case of causal MCARMA processes treated in
Section~\ref{sec:bk22-stat-mcarma-proc} and the case of non-causal MCARMA
processes in Section~\ref{sec:bk22-non-stat-mcarma}. A careful distinction
between the two cases is justified as the cone-invariance constraints may
interact with the stability conditions. The main results of this section are
Theorem~\ref{thm:bk22-positive-stationary} and
Theorem~\ref{thm:bk22-positivite-non-stable-MCARMA} below which establish
sufficient and/or necessary conditions for the cone-invariance of
$\MM_{n,m}$-valued causal and non-causal MCARMA processes.  

\subsection{Positive operators and increasing L\'evy
  processes}\label{sec:bk22-cone-valued-levy}
Before we study cone-invariance of $\MM_{n,m}$-valued MCARMA processes in the
next two sections, we recall some additional preliminaries concerning
\emph{convex cones}, \emph{(quasi)-positive} operators and \emph{increasing L\'evy processes}. For $n,m\in\MN$, we consider the inner product space
$(\MM_{n,m},(\cdot,\cdot)_{n,m})$ and, as usual, identify $\MM_{n,m}$
  with its dual space $\MM_{n,m}^{*}$. Throughout this section we assume that
$\cK$ is a (convex) cone in
$\MM_{n,m}$, i.e. $\cK\subseteq \MM_{n,m}$ is such that $\cK+\cK\subseteq \cK$,
$\lambda \cK\subseteq \cK$ for all $\lambda\in\MRplus$ and $\cK\cap
(-\cK)=\set{\zero_{n,m}}$, and assume that the interior of
  $\cK$ is non-empty. Note that some authors call $\cK$ a
  \emph{pointed} or \emph{proper cone}, in case that the zero element is contained in
  $\cK$. Moreover, we write ``$\leq_{\cK}$'' for the partial-ordering on $\MM_{n,m}$ induced by $\cK$, i.e. for $x,y\in
\MM_{n,m}$: $x\leq_{\cK} y$ if and only if $y-x\in \cK$.

\subsubsection*{Positive and quasi-positive operators}
We denote by $\pi(\cK)\subseteq \cL(\MM_{n,m})$ the set of all linear operators
leaving the cone $\cK$ invariant, i.e. 
\begin{align*}
\pi(\cK)=\set{A\in\cL(\MM_{n,m})\colon A(u)\geq_{\cK} 0\,\text{for all }u\geq_{\cK} 0}.  
\end{align*}
We call elements in $\pi(\cK)$ \emph{positive operators} on $\MM_{n,m}$. Note that the set
$\pi(\cK)$ is a  \emph{convex algebra cone}, that means it is a convex cone such that
$B_{1},B_{2}\in \pi(\cK)$ implies $B_{1}B_{2}\in \pi(\cK)$. We denote by
``$\preceq$'' the partial ordering on $\cL(\MM_{n,m})$ induced by
$\pi(\cK)$. Moreover, we call an element $A\in\cL(\MM_{n,m})$
\emph{quasi-positive}, if $\exp(A t)\succeq 0$ for all $t\geq 0$, where $\exp(A
t)$ denotes the operator exponential of $A t$.  
  \begin{remark}
    The notion of a quasi-positive operator is closely related to the concept of
    \emph{quasi-monotone increasing (qmi)} functions. Indeed, an operator $A$ is
    quasi-positive if and only if the map $u\mapsto A u$ is \emph{quasi monotone
      increasing}, i.e. $u\leq_{\cK} v$ and $x\in\cK^{*}\df\set{v\in\MM_{n,m}\colon
      \langle v,u\rangle_{n,m} \geq 0\,,\forall u\in\cK}$ with $\langle
    v-u,x\rangle_{n,m}=0$ implies $\langle Av-Au,x\rangle_{n,m}\geq 0$, see,
    e.g.,~\cite{HL98, Els74}. Note that the concept of qmi functions is also
    crucial in the theory of affine processes on cones, where
    it is known that the right-hand side functions of the associated Riccati
    equations must be quasi-monotone increasing to ensure the cone-invariance of
    its solution, see~\cite{CMET16, CKK22a}.
\end{remark}

The following two cases are our main examples for convex cones in $\MM_{n,m}$:
\begin{enumerate}
\item[a)] For $m=1$ and $n=d$ we have
  $(\MR_{d},(\cdot,\cdot)_{d})=(\MM_{d,1},(\cdot,\cdot)_{d,1})$. On $\MR_{d}$
  we consider the positive orthant $\cK=\MR^{+}_{d}$, which is a convex cone and
  we denote its induced partial
  ordering by ``$\leq_{d}$''. It is well known that the cone $\pi(\MR_{d}^{+})\subseteq \MM_{d}$ of
  $\MR_{d}^{+}$-preserving linear maps (matrices) is given by the set of all
  positive matrices (more precisely \emph{non-negative matrices}), i.e.
  \begin{align*}
    \pi(\MR_{d}^{+})=\set{(a_{i,j})_{1\leq i,j\leq d}\in\MM_{d}\colon\,
    a_{i,j}\geq 0\,\, \forall i,j=1,\ldots,d}.    
  \end{align*}
  It is also well-known that a matrix $A=(a_{i,j})_{1\leq i,j\leq d}\in\MM_{d}$ is
  quasi-positive (sometimes called \emph{cross-positive}) if and only if
  $a_{i,j}\geq 0$ for all $i,j=1,\ldots,d$ such that $i\neq j$, i.e. all
  off-diagonal elements are non-negative and the diagonal ones can be arbitrary, see e.g.~\cite{HL98}.
\item[b)] In case of $n=m=d$ for some $d\in\MN$ we consider the space of real
  $d\times d$-matrices $\MM_{d}$. On $\MM_{d}$ we consider the convex cone of
  all symmetric and positive-semi definite matrices $\cK=\MS^{+}_{d}$ and denote
  the induced partial ordering by ``$\leq_{\MS_{d}^{+}}$''. As far as we know
  there is no analogous characterization of the set $\pi(\MS_{d}^{+})$
  known. Partial results were achieved in this direction, see
  e.g.~\cite{LT92} for some related results in the theory of linear preserver
  problems.   
\end{enumerate}
\subsubsection*{Multivariate L\'evy processes on cones} Let
$L=(L_{t})_{t\in\MR}$ denote a two-sided L\'evy process on $\MM_{n,m}$ defined on some
filtered probability space $(\Omega,\cF,\MF,\MP)$ and let $\cK$ be a convex cone
in $\MM_{n,m}$. We define an \emph{increasing L\'evy process} as follows:
  \begin{definition}
    We call a two-sided L\'evy process $(L_{t})_{t\in\MR}$ on $\MM_{n,m}$
    $\cK$-\emph{increasing}, whenever $L_{t}-L_{s}\in \cK$ for all $t,s\in\MR$ such
    that $s<t$. If $L$ has representation~\eqref{eq:two-sided-Levy}, then this is
    equivalent to $L^{(1)}$ and $L^{(2)}$ being $\cK$-valued L\'evy processes.  
  \end{definition}
The characteristic exponent~\eqref{eq:bk22-chracteristic-Levy} of a
$\cK$-increasing L\'evy process is given by  
\begin{align*}
  \varphi_{L}(z)=\I \langle
  \gamma_{L},z\rangle_{n,m}+\int_{\cK}(\E^{\I\langle
  \xi,z\rangle_{n,m}}-1)\nu_{L}(\D\xi),\quad z\in \MM_{n,m} ,
\end{align*}
where $\gamma_{L}\in \cK$ and the L\'evy measure $\nu_{L}$ is concentrated on
$\cK\setminus\set{\zero_{n,m}}$ and satisfies $\int_{\cK}(1\wedge \xi)\nu(\D\xi)<\infty$, hence jumps, small or
large, of the L\'evy process are $\cK$-valued and of finite-variation. Moreover,
note that compared to~\eqref{eq:bk22-chracteristic-Levy} the diffusion part
vanishes, i.e. an increasing L\'evy process is of pure-jump type. We refer
to~\cite{BNPA08} for more information on matrix-valued (increasing) L\'evy processes.

\subsection{Cone valued causal MCARMA processes}\label{sec:bk22-stat-mcarma-proc}
Throughout this section we fix $n,m\in\MN$ and let $\cK$ be a cone in
$\MM_{n,m}$. Moreover, we assume that $L=(L_{t})_{t\in\MR}$ is a $\cK$-increasing
two-sided L\'evy process defined on the filtered probability space
$(\Omega,\cF,\MF,\MP)$ such that $\EX{\log(\norm{L_{1}}_{n,m})}<\infty$. We let
$\cA_{p}$, $E_{p}$ and $\cC_{q}$ be as in~\eqref{eq:bk22-Ap}-\eqref{eq:bk22-Ep}
and assume that $\tau(\cA_{p})<0$. This of course implies that the stable output
process $(X_{t})_{t\in\MR}$, associated with $(\cA_{p},E_{p},\cC_{q},L)$, is a
causal $\MM_{n,m}$-valued MCARMA process given by  
\begin{align}\label{eq:bk22-stationary-carma}
  X_{t}=\int_{-\infty}^{t}\cC_{q} \E^{(t-s)\cA_{p}}E_{p}\D L_{s},\quad t\in\MR.
\end{align}

We recall the following connection between quasi-positive operators and their
resolvent, see~\cite{HL98, Els74}:

\begin{lemma}\label{lem:bkk22-inverse-quasi-positive}
  Let $V$ be a linear space and $\cK\subseteq V$ a convex cone. Then for every
  quasi-positive $\opA\in\cL(V)$ with $\tau(\opA)<0$ we have $-\opA^{-1}\succeq
  0$.  
\end{lemma}

From the representation~\eqref{eq:bk22-stationary-carma} we see that $X_{t}\in \cK$ for
every $t\in\MR$, whenever the L\'evy process $(L_{t})_{t\in\MR}$ is
$\cK$-increasing and $g(s)=\cC_{q}\E^{s\cA_{p}}E_{p}\in\pi(\cK)$ holds true for every
$s\geq 0$. In the following lemma we prove a
particular form of the Laplace transform of the kernel function $g$. The main
part of the proof, the computation of the transfer function,  is similar to
the matrix representation case in Lemma~\ref{lem:bkk22-right-matrix-description}. 

\begin{lemma}\label{lem:bkk22-stationary-Laplace-transform}
The Laplace transform $\varphi\colon \MRplus\to\cL(\MM_{n,m})$ of the kernel
$g(s)=\cC_{q}\E^{s\cA_{p}}E_{p}$ exists and is given by 
\begin{align}
  \label{eq:bk22-stationary-Laplace-transform}
  \varphi(\lambda)=\opQ(\lambda)\opP(\lambda)^{-1},\quad \lambda\geq 0.
\end{align}  
\end{lemma}
\begin{proof}
  Since $\tau(\cA_{p})<0$ we see that for $\lambda\geq 0$ the resolvent
  $R(\lambda,\cA_{p})=(\lambda\opI-\cA_{p})^{-1}$ is given by the Laplace
  transform of the matrix semigroup $\E^{t\cA_{p}}$, i.e. $R(\lambda,
  \cA_{p})=\int_{0}^{\infty}\E^{-\lambda s}\E^{s\cA_{p}}\,\D s$. 
  We thus compute
  \begin{align}\label{eq:bk22-stationary-Laplace-transform-1}
   \varphi(\lambda)=\int_{0}^{\infty}\E^{-\lambda s}g(s)\,\D
  s=\cC_{q}\left(\int_{0}^{\infty}\E^{-\lambda
                     s}\E^{s\cA_{p}}\,\D s\right)E_{p}=\cC_{q} R(\lambda,\cA_{p})E_{p}.
  \end{align}
  For $\lambda\in\MC$ we compute the term $R(\lambda,\cA_{p})E_{p}$ as
  follows: Let $\bF_{p}\df[F_{1},F_{2},\ldots,F_{p}]^{\intercal}$ in $\cL(\MM_{n,m},(\MM_{n,m})^{p})$ with
  $F_{i}\in\cL(\MM_{n,m})$ and such that $\bF_{p}$ applied to $x$ satisfies
  $\bF_{p}(x)=(F_{1}x,F_{2}x,\ldots,F_{p}x)^{\intercal}$. Moreover, set $\opA\df
  [\opA_{p},\opA_{p-1},\ldots,\opA_{1}]$ and consider $ (\lambda
  \opI-\cA_{p})\bF_{p}=E_{p}$, which is equivalent to
 \begin{align*}
  [\lambda F_{1}-F_{2},\lambda F_{2}-F_{3},\ldots,\lambda F_{p-1}-F_{p},
   \lambda F_{p}-\opA\bF_{p}]=[0,\ldots,\opI],.
 \end{align*}
 Solving for $\bF_{p}$ yields $F_{1}=\lambda^{-1}F_{2}$, $F_{2}=\lambda^{-1}F_{3},\ldots
 ,F_{p-1}=\lambda^{-1}F_{p}$ and thus for every $i=1,\ldots,p-1$ we have
 $F_{i}=\lambda^{-(p-i)}F_{p}$. Moreover, the last equation reads as
 \begin{align*}
   \big(\lambda^{p}\opI-\opA_{p}-\opA_{p-1}\lambda-\ldots-\opA_{1}\lambda^{p-1}\big)\lambda^{-(p-1)}F_{p}=\opI, 
 \end{align*}
 and hence by definition of $\opP(\lambda)$ we see that
 $F_{p}=\lambda^{p-1}\opP(\lambda)^{-1}$ and therefore
 $F_{i}=\lambda^{i-1}\opP(\lambda)^{-1}$ for $i=1,\ldots,p-1$. In vector
 notation this means
 \begin{align*}
   \bF_{p}=[\opI\circ\opP(\lambda)^{-1},\lambda
   \opI\circ\opP(\lambda)^{-1},\ldots,\lambda^{p-1}\opI\circ\opP(\lambda)^{-1}]= (1,
   \lambda,\ldots,\lambda^{p-1})\otimes \opP(\lambda)^{-1}.
 \end{align*}
 Thus inserting $\bF_{p}=R(\lambda,\cA_{p})E_{p}$ back
 into~\eqref{eq:bk22-stationary-Laplace-transform-1} and by the definition of
 $\mathbf{Q}$ in~\eqref{eq:bk22-Q}, we obtain
 \begin{align*}
   \cC_{q}\bF_{p}=\opC_{0}\opP(\lambda)^{-1}+\lambda \opC_{1}\opP(\lambda)^{-1}+\ldots+\lambda^{p-1}\opC_{p-1}\opP(\lambda)^{-1}=\opQ(\lambda)\opP(\lambda)^{-1}.
 \end{align*}
\end{proof}

Next, we introduce the fundamental property of the Laplace transforms of the
kernel function $s\mapsto g(s)$ that will ensure the cone-invariance of the
associated causal MCARMA processes. The following definition is adapted from~\cite[Definition 5.4]{Are84}.
\begin{definition}
We call a function $f\colon \MRplus\to \cL(\MM_{n,m})$ \emph{completely
  monotone} with respect to $\pi(\cK)$, if $f$ is infinitely often
differentiable and $(-1)^{n}f^{(n)}(\lambda)\succeq 0$ for all $\lambda>0$ and
$n\in\MN_{0}$.
\end{definition}

The following theorem is our main result on the cone-invariance of
causal matrix-valued MCARMA processes, which builds on the notion of completely
monotone functions. 

\begin{theorem}\label{thm:bk22-positive-stationary}
  Let $(X_{t})_{t\in\MR}$ be an $\MM_{n,m}$-valued causal MCARMA process of order $(p,q)$
  given by~\eqref{eq:bk22-stationary-carma}. Moreover, let $\opP(\lambda)$ and
  $\opQ(\lambda)$ be the operator polynomials in~\eqref{eq:bk22-P}
  and~\eqref{eq:bk22-Q}, respectively. Then the following holds true:
  \begin{theoremenum}
  \item\label{item:bk22-positive-stationary-1} $(X_{t})_{t\in\MR}$ is $\cK$-valued if
    and only if the map $\lambda\mapsto \opQ(\lambda)\opP(\lambda)^{-1}$ is
    completely monotone with respect to $\pi(\cK)$.
  \item\label{item:bk22-positive-stationary-2} If $\lambda\mapsto\opQ(\lambda)$ is
    completely monotone with respect to $\pi(\cK)$ and $\opP(\lambda)$ can be
    decomposed into linear factors as follows: 
    \begin{align}\label{eq:bk22-decomposition-linear}
      \opP(\lambda)=\prod_{i=1}^{p}(\lambda \opI- \hat{\opA}_{i}), 
    \end{align}
    where for all $i=1,\ldots,p$ the operator $\hat{\opA}_{i}\in\cL(\MM_{n,m})$ is
    quasi-positive and $\tau(\hat{\opA}_{i})<0$, then $(X_{t})_{t\in\MR}$ is
    $\cK$-valued.     
  \end{theoremenum}
\end{theorem}
\begin{proof}
  By Lemma~\ref{lem:bkk22-stationary-Laplace-transform} the Laplace transform $\phi$
  of the kernel $g(s)=\cC_{q}\E^{s\cA_{p}}E_{p}$ is given by the operator
  rational function $\lambda\mapsto \opQ(\lambda)\opP(\lambda)^{-1}$. By a
  vector valued version of Bernstein's theorem, see \cite[Theorem 5.5]{Are84},
  and Lemma~\ref{lem:bkk22-stationary-Laplace-transform} it follows that $s\mapsto
  g(s)$ is positive. Indeed, by Bernstein's theorem we have $g(s)\in\pi(\cK)$ for all
  $s\geq 0$ if and only if its Laplace transform is completely monotone with
  respect to $\pi(\cK)$, but since its Laplace transform is given by
  $\lambda\mapsto\opQ(\lambda)\opP(\lambda)^{-1}$
  \cref{item:bk22-positive-stationary-1} follows.\\
  For the second statement~\cref{item:bk22-positive-stationary-2}, note that from
  part i) it follows that the MCARMA process $(X_{t})_{t\geq 0}$ is $\cK$-valued if and
  only if $\lambda\mapsto \opQ(\lambda)\opP(\lambda)^{-1}$ is completely
  monotone with respect to $\pi(\cK)$. Note further that for every
  $i=1,\ldots,p$ the resolvent $R(\lambda,\hat{\opA}_{i})$ exists for every
  $\lambda>0$ and is completely monotone with respect to $\pi(\cK)$. Indeed, by
  assumption $\hat{\opA}_{i}$ is quasi-positive and
  $\tau(\hat{\opA}_{i})<0$. This implies that also
  $-\lambda\opI+\hat{\opA}_{i}$ is quasi-positive and
  $\tau(-\lambda\opI+\hat{\opA}_{i})<0$ for every $\lambda\geq 0$. Indeed,
  note that the operator $-\lambda \opI$ is always quasi-positive for
  $\lambda>0$, since whenever $\langle
  u,v\rangle_{n,m}=0$ we have $\langle -\lambda \opI u,v\rangle_{n,m}=-\lambda \langle
  u,v\rangle_{n,m}=0$. By an application of
  Lemma~\ref{lem:bkk22-inverse-quasi-positive} it follows that $(\lambda
  \opI-\hat{\opA}_{i})^{-1}\succeq 0$. Moreover, for every $n\in\MN$ we have
  \begin{align}\label{eq:bk22-prop:bk22-positive-stationary-proof-1}
    (-1)^{n}\frac{\D^{n} }{\D \lambda^{n}}(\lambda \opI-\hat{\opA}_{i})^{-1}=(\lambda
    \opI-\hat{\opA}_{i})^{-(1+n)},
  \end{align}
  where the right-hand side of~\eqref{eq:bk22-prop:bk22-positive-stationary-proof-1} is
  again positive, since $\pi(\cK)$ is an algebra cone. It thus follows that for
  every $i=1,\ldots,p$ the linear factor $(\lambda\opI-\hat{\opA}_{i})^{-1}$
  in the decomposition~\eqref{eq:bk22-decomposition-linear} is completely monotone
  and by assumption $\opQ$ is completely monotone as well. Now the assertion
  follows since the product of completely monotone functions is again
  completely monotone (use the general Leibniz's rule here) and
  hence~\cref{item:bk22-positive-stationary-2} follows from part i).   
\end{proof}

\begin{remark}
  \begin{enumerate}
  \item[i)] In the univariate case an analog of~\cref{item:bk22-positive-stationary-1} was shown
    in~\cite[Theorem 2]{TC05}. Here we extend the result to the
    multivariate setting. However, it is to be noted that the result can be
    extended even beyond finite-dimensions. In fact, also for certain
    Hilbert-valued CARMA processes as introduced in~\cite{BS18} a similar
    characterization can be shown and will be discussed in future work.
  \item[ii)] The factorization of operator polynomials into the
    form~\eqref{eq:bk22-decomposition-linear} is well studied in the
    literature, see, e.g.~\cite{Lei86} and in particular for matrix
    polynomials \cite{GLR09}. For instance, a sufficient criteria for $\opP(\lambda)$ to admit
    a factorization of the form~\eqref{eq:bk22-decomposition-linear} is that
    the transition operator $\cA_{p}$ is diagonizable. If an operator
    polynomial is factorizable, then the operators $\hat{\opA}_{i}$ can be
    computed by iterated operator division and the additional positivity conditions can be checked thereafter.
  \end{enumerate}
\end{remark}

The strength of~\cref{item:bk22-positive-stationary-2}, however, lays in the fact that it
provides us with a simple method to construct cone valued stationary MCARMA
processes by choosing suitable operators $\hat{\op{A}}_{i}$ for
$i=1,\ldots,p$. This is explained in the following example by means of a second
order MCARMA process: 

\begin{example}\label{ex:positive-MCAR}
  Let $n=m=d$ for some $d\in\MN$, $p=2$ and $\opQ(\lambda)=\opI$. Moreover, let
  $\hat{\opA}_{1},\hat{\opA}_{2}\in\cL(\MM_{d})$ be quasi-positive with
  $\tau(\hat{\opA}_{1}),\tau(\hat{\opA}_{2})<0$. For example, we could choose
  $\hat{\opA}_{i}x\df \hat{A}_{i}x+x\hat{A}_{i}^{*}$ for some matrix $A_{i}\in
  \MM_{d}$ with $\tau(\hat{A}_{i})<0$ for $i=1,2$. Note that in this case for
  $i=1,2$ we have
  $\sigma(\hat{\opA}_{i})=\sigma(\hat{A}_{i})+\sigma(\hat{A}_{i})$,
  see~\cite{Ros56}. If we set $\opA_{1}\df \hat{\opA}_{1}+\hat{\opA}_{2}$ and
  $\opA_{2}\df\hat{\opA}_{1}\hat{\opA}_{2}$, we see that 
  \begin{align*}
   \opP(\lambda)=(\lambda\opI-\hat{\opA}_{1})(\lambda\opI-\hat{\opA}_{2})=\lambda^{2}\opI-\lambda\opA_{1}-\opA_{2}. 
  \end{align*}
  Hence following \cref{item:bk22-positive-stationary-2} and due to the fact
  that $\tau(\cA_{2})=\tau(\opA_{2})$, this specification gives rise to a
  $\cK$-valued causal MCAR process of order $p=2$, whenever
  $\tau(\opA_{2})=\tau(\hat{\op{A}}_{1}\hat{\op{A}}_{2})<0$.      
\end{example}

\vspace{3mm}

Note that the conditions in \cref{item:bk22-positive-stationary-2} are not necessary
for causal MCARMA processes to be cone valued. Indeed, in some situations
we can check the conditions of \cref{item:bk22-positive-stationary-1} directly as
illustrated by the following example:

\begin{example}
  Let $p=2$, $\opC_{0}\in \pi(\cK)$ and $\opA\in\cL(\MM_{n,m})$ be invertible and
  such that $-\opA^{2}$ is quasi-positive. If we define
  \begin{align*}
  \cA_{2}\df
    \begin{bmatrix}
      \op0 & \opI \\
      -\opA^{2}& \op0
    \end{bmatrix},\quad \text{and} \quad  \cC_{0}\df [\opC_{0},\op0],
  \end{align*}
  then $\tau(\cA_{2})=\tau(-\opA^{2})<0$ and the associated operator
  polynomials are $\opP(\lambda)=\lambda^{2}\opI+\opA^{2}$ and
  $\opQ(\lambda)=\opC_{0}$. We thus see that $\lambda\mapsto
  \opQ(\lambda)\opP(\lambda)^{-1}$ is completely monotone, although no
  factorization of the form in~\eqref{eq:bk22-decomposition-linear} is available.
\end{example}

In the following Proposition~\ref{prop:bk22-CARMA-positive} we extend another sufficient
positivity criteria, known in the univariate case in~\cite[Theorem 1
e)]{TC05}, to $\MR_{d}$-valued MCARMA processes. In
order to this, we adapt the proof of \cite[Theorem 1]{Bal94} to
$\pi(\MR_{d}^{+})$-valued rational functions. Unsurprisingly, this
multivariate version is more involved and requires some additional and rather
technical assumptions.      

\begin{proposition}\label{prop:bk22-CARMA-positive}
  Let  $p\in\MN$ and $q\in\MN_{0}$ such that $q<p$ and let $(X_{t})_{t\geq 0}$
  be an $\MR_{d}$-valued causal MCARMA given by~\eqref{eq:bk22-stationary-carma}
  with parameters $(\cA_{p},E_{p},\cC_{q},L)$. Let $P(\lambda)$ and $Q(\lambda)$
  denote the matrix polynomials in~\eqref{eq:bk22-P} and~\eqref{eq:bk22-Q},
  respectively, and assume that 
  $\set{\hat{C}_{i}\in\MM_{d}\colon i=0,\ldots,q}$ and
  $\set{\hat{A}_{i}\in\MM_{d}\colon i=1,\ldots,p}$ are two commutative
  families of positive matrices such that
  \begin{align}\label{eq:bk22-positive-stationary}
    Q(\lambda)=\prod_{i=0}^{q}(\lambda \opI-\hat{C}_{i})\quad\text{and}\quad
    P(\lambda)=\prod_{i=1}^{p}(\lambda \opI-\hat{A}_{i}). 
  \end{align}
  Moreover, assume that for every $\lambda>0$  we have
  $Q(\lambda)P(\lambda)^{-1}=P(\lambda)^{-1}Q(\lambda)$, the matrix
  logarithms $\log(P(\lambda))$ and $\log(Q(\lambda))$ exist in $\MM_{d}$ and
  for every $i=0,\ldots,q$ there exist
  $l^{(i,1)},l^{(i,2)},\ldots,l^{(i,q+1)}\in\pi(\MR_{d}^{+})$ such that
  for every $n\in\MN$ we have $(l^{(i,j+1)})^{n}\preceq l^{(i,j+1)}$ for
  all $j=0,\ldots,q$, 
  \begin{align}\label{eq:bk22-bistochastic-smaller-one}
    \hat{C}_{i}\preceq \sum_{j=0}^{q}l^{(i,j+1)}\odot\hat{A}_{j+1}\quad\text{and}\quad\sum_{i=0}^{q}l^{(i,j+1)}=\one_{d}.
  \end{align}
  If $p>q+1$, we assume in addition that $\tau(\hat{A}_{i})<0$ for all
  $i=q+2,\ldots,p$. Then $(X_{t})_{t\geq 0}$ is an $\MR_{d}^{+}$-valued causal
  MCARMA process of order $(p,q)$ with associated moving average polynomial
  $Q$ and autoregressive polynomial $P$.  
\end{proposition}
\begin{proof}
  Following~\cref{item:bk22-positive-stationary-1}, it suffices to show that
  $\lambda\mapsto Q(\lambda)P(\lambda)^{-1}$ is completely monotone on
  $\MRplus$ with respect to the cone
  $\pi(\MR_{d}^{+})$. By~\eqref{eq:bk22-positive-stationary} we have 
  \begin{align*}
   Q(\lambda)P(\lambda)^{-1}=
    \left(\prod_{i=0}^{q}(\lambda \opI-
    \hat{C}_{i})\right)\left(\prod_{i=1}^{p}(\lambda \opI-
    \hat{A}_{i})\right)^{-1},\quad \forall\lambda>0.
  \end{align*}
  Since the multiplication of monotone functions is again completely monotone it
  suffices to show the complete monotonicity of $\lambda\mapsto
  Q(\lambda)\check{P}(\lambda)^{-1}$ for 
  \begin{align*}
    \check{P}(\lambda)=\prod_{i=1}^{q}(\lambda \opI-\hat{A}_{i}),\quad\lambda\in\MC,   
  \end{align*}
  since by~\cref{item:bk22-positive-stationary-2} we conclude that the maps
  $(\lambda\opI-\hat{A}_{i})^{-1}$ are completely monotone for all $i=q+2,\ldots,p$. Moreover, it
  suffices to prove that the map
  $\lambda\mapsto\log\big(Q(\lambda)\check{P}(\lambda)^{-1}\big)$ is
  completely monotone, since for all $\lambda>0$ we have 
  \begin{align}\label{eq:bk22-CARMA-positive-sum-1}
  Q(\lambda)\check{P}(\lambda)^{-1}=\exp\big(\log(Q(\lambda)\check{P}(\lambda)^{-1})\big)=\sum_{n=1}^{\infty}\frac{\big(\log(Q(\lambda)\check{P}(\lambda)^{-1})\big)^{n}}{n!}.   
  \end{align}
  By assumption we know that the matrix logarithms of $Q(\lambda)$ and
  $\check{P}(\lambda)^{-1}$ exist for every $\lambda>0$ and moreover the matrices
  $\hat{C}_{0},\ldots,\hat{C}_{q}$, the matrices
  $\hat{A}_{1},\ldots,\hat{A}_{q+1}$ and the matrix rational function
  $Q(\lambda)$ and $\check{P}(\lambda)^{-1}$ mutually commute, thus we find:
  \begin{align}
    \log\big(Q(\lambda)\check{P}(\lambda)^{-1}\big)&=\log\big(Q(\lambda)\big)-\log\big(\check{P}(\lambda)\big)\nonumber\\
                                   &=\sum_{j=0}^{q}\log\big(\lambda
                                     \opI-\hat{C}_{j}\big)-\sum_{i=1}^{q+1}\log\big(\lambda
                                     \opI-\hat{A}_{i}\big)\nonumber\\
                                   &=\int_{0}^{\infty}\E^{-\lambda
                                     s}s^{-1}\sum_{i=0}^{q}(\E^{s\hat{A}_{i+1}}-\E^{s\hat{C}_{i}})\D
                                     s. \label{eq:bk22-CARMA-positive-sum}
  \end{align}
  From~\eqref{eq:bk22-CARMA-positive-sum} we see that it suffices to show that
  $\sum_{i=0}^{q}(\E^{s\hat{A}_{i+1}}-\E^{s\hat{C}_{i}})\succeq 0$ holds for all $s\geq
  0$. For this note that the matrix exponential is monotone with respect
  to $\pi(\MR_{d}^{+})$, i.e. for $G_{1},G_{2}\in\MM_{d}$ such that
  $G_{1}\preceq G_{2}$ we have $\E^{G_{1}}\preceq \E^{G_{2}}$. This can be
  seen from the definition of the matrix exponential and due to the fact that
  monomials are monotone with respect to $\pi(\MR_{d}^{+})$. Note further that
  the matrices 
  $l^{(i,j+1)}$ are in $\pi(\MR_{d}^{+})$, i.e. entrywise non-negative, and the
  same holds true for $\hat{C}_{i}$ for $i=1,\ldots,q+1$ and
  $j=0,\ldots,q$. Thus by Lemma~\ref{lem:hadamard-product} we have $(l^{(i,j+1)}\odot \hat{C}_{i})^{n}\preceq
  (l^{(i,j+1)})^{n}\odot (\hat{C}_{i})^{n}$ for every $n\in\MN$ and hence for
  $i=1,\ldots,q+1$ and $j=0,\ldots,q$ we see that
  \begin{align}\label{eq:bk22-CARMA-positive-sum-3}
    \E^{ l^{(i,j+1)}\odot \hat{A}_{i}}=\sum_{n\in\MN}\frac{(l^{(i,j+1)}\odot
    \hat{A}_{i})^{n}}{n!}\preceq \sum_{n\in\MN}\frac{(l^{(i,j+1)})^{n}\odot
    \hat{A}_{i}^{n}}{n!}\preceq
    l^{(i,j+1)}\odot\sum_{n\in\MN}\frac{\hat{A}_{i}^{n}}{n!}. 
  \end{align}
  This together with assumption~\eqref{eq:bk22-bistochastic-smaller-one} and
  the convexity of the matrix exponential imply 
  \begin{align*}
    \sum_{i=0}^{q}\E^{s\hat{C}_{i}}\preceq \sum_{i=0}^{q}\E^{s\sum_{j=0}^{q}l^{(i,j+1)}\odot\hat{A}_{j+1}}&\preceq
                                                                                                            \sum_{i=0}^{q}\sum_{j=0}^{q}l^{(i,j+1)}\odot\E^{s\hat{A}_{j+1}}\\
    &=
                                      \sum_{j=0}^{q}\Big(\sum_{i=0}^{q}l^{(i,j+1)}\Big)\odot\E^{s\hat{A}_{j+1}}\\
    &=\sum_{j=0}^{q}\E^{s\hat{A}_{j+1}},\quad \forall s\geq 0.
  \end{align*}
  Hence, from~\eqref{eq:bk22-CARMA-positive-sum} it follows that
  $\log\big(Q(\lambda)\check{P}(\lambda)^{-1}\big)$ is given by the Laplace transform
  of the $\pi(\MR_{d}^{+})$-valued kernel $s\mapsto
  s^{-1}\sum_{i=0}^{q}(\E^{s\hat{A}_{i+1}}-\E^{s\hat{C}_{i}})$, which by
  Bernstein's theorem yields the complete monotonicity of
  $\log\big(Q(\lambda)\check{P}(\lambda)^{-1}\big)$ and hence following the reasoning
  above we conclude by~\cref{item:bk22-positive-stationary-1} that
  $(X_{t})_{t\in\MR}$ is $\MR_{d}^{+}$-valued. That $Q$ and $P$ are the
  moving-average, respectively, autoregressive polynomials of
  $(X_{t})_{t\in\MR}$ then follows from the commutativity of $Q(\lambda)$ and
  $P(\lambda)$ and Remark~\ref{rem:MA-and-AR-polynomials}.
\end{proof}

\begin{remark}
  \begin{enumerate}
  \item[i)] Following \cite[Theorem 6.4.15 c)]{HJ91} the (real) logarithm of the
    matrix $P(\lambda)$ exists if and only if $P(\lambda)$ is non-singular and
    has an even number of Jordan blocks of each size for every negative
    eigenvalue.
  \item[ii)] The commutativity assumptions in
      Proposition~\ref{prop:bk22-CARMA-positive} may appear as quite
      restrictive, however, similarly to \cref{item:bk22-positive-stationary-2}
      before, we can use this proposition to construct
      $\MR_{d}^{+}$-valued causal MCARMA processes by selecting appropriate
      matrices $(\hat{C}_{i})_{i=0,\ldots,q}$ and
      $(\hat{A}_{i})_{i=1,\ldots,p}$, e.g. by choosing diagonal matrices
      satisfying the extra condition~\eqref{eq:bk22-bistochastic-smaller-one}. 
  \item[iii)] Note that the technical assumption of
    Proposition~\ref{prop:bk22-CARMA-positive} is best understood when departing
    from the condition   
    \begin{align}\label{eq:bk22-bistochastic-condition-2}
      \sum_{i=0}^{q}\hat{A}_{i+1}\succeq \sum_{i=0}^{q}\hat{C}_{i}.
    \end{align}
    Indeed, note that if~\eqref{eq:bk22-bistochastic-condition-2} holds for all
    $1\leq k,n\leq d$, then in particular $\sum_{i=0}^{q}(\hat{A}_{i+1})_{k,n}\geq
    \sum_{i=0}^{q}(\hat{C}_{i})_{k,n}$. Following the Hardy-Littlewood rearrangement
    inequality and Hall's marriage theorem, see also \cite[Equation 3]{Bal94}
    and references therein, we see that for all $i=0,\ldots,q$ and
    $k,n=1,\ldots,d$ there exist $(l^{(i,j+1)}_{k,n})_{j=0,\ldots,q}$ with $0\leq l^{(i,j+1)}_{k,n}\leq 1$ such that 
    $(\hat{C}_{i})_{k,n}\leq
    \sum_{j=0}^{q}l^{(i,j+1)}_{k,n}(\hat{A}_{j+1})_{k,n}$ and
    $\sum_{i=0}^{q}l^{(i,j)}_{k,n}=1$ for every $1\leq k,n\leq d$. Hence,
    setting $l^{(i,j)}=(l^{(i,j)}_{k,n})_{1\leq k,n\leq d}$,  we see that the
    conditions in Proposition~\ref{prop:bk22-CARMA-positive} are met if we
    assume~\eqref{eq:bk22-bistochastic-condition-2} together with
    $(l^{(i,j)})^{n}\preceq l^{(i,j)}$ for every $i=0,\ldots,q$ and
    $j=1,\ldots,q+1$ any $n\in\MN$. 
    Condition~\eqref{eq:bk22-bistochastic-condition-2} is the analogous
    multivariate condition compared to the univariate case in~\cite[Theorem 1
    e)]{TC05}. In our case, the additional
    assumption~\eqref{eq:bk22-bistochastic-smaller-one} is needed
    in~\eqref{eq:bk22-CARMA-positive-sum-3}. Otherwise, the matrix exponential is
    not convex with respect to the Hadamard product. Note further that it follows
    from~\eqref{eq:bk22-CARMA-positive-sum}, that the stronger, but more accessible,
    condition $\sum_{i=0}^{q}(\E^{s\hat{A}_{i+1}}-\E^{s\hat{C}_{i}})\succeq 0$ for all
    $s\geq 0$ is sufficient and could replace the condition
    in~\eqref{eq:bk22-bistochastic-smaller-one}. 
  \end{enumerate}
\end{remark}

\subsection{Cone valued non-stable output processes}\label{sec:bk22-non-stat-mcarma}

In this section we study cone valued non-stable output and non-causal MCARMA
processes. As before we assume that $\cK$ is a cone in $\MM_{n,m}$ for
$n,m\in\MN$ and let $L=(L_{t})_{t\in\MR}$ be a $\cK$-increasing
two-sided L\'evy process defined on the filtered probability space
$(\Omega,\cF,\MF,\MP)$ such that $\EX{\log(\norm{L_{1}}_{n,m})}<\infty$. Moreover, we let $\cA_{p}$,
$E_{p}$ and $\cC_{q}$ be as in~\eqref{eq:bk22-Ap}-\eqref{eq:bk22-Ep} and assume throughout this section
that $\tau(\cA_{p})\geq 0$. Analogous to the stable case in the previous
section, we are interested in sufficient and necessary conditions for the
non-stable output processes to be cone valued. Recall that every non-stable output process
$(X_{t})_{t\in\MR}$, associated with $(\cA_{p},E_{p},\cC_{q},L)$, has the representation    
\begin{align}\label{eq:bk22-CARMA-p-q-variation-of-constant}
  X_{t}=\cC_{q}\E^{(t-s)\cA_{p}}Z_{s}+\int_{s}^{t}\cC_{q}\E^{(t-u)\cA_{p}}E_{p}\D
  L_{s},\quad s<t.
\end{align}
If, in addition to the above, condition~\eqref{eq:bk22-stationary-solution-condition} is
satisfied, then there exists a unique stationary solution
to~\eqref{eq:bk22-CARMA-p-q-variation-of-constant} which is the associated
non-causal MCARMA process. However, in contrast to the causal case, a non-causal
MCARMA process does not admit representation~\eqref{eq:bk22-stationary-carma}. Indeed, following the proof of
Proposition~\ref{prop:bk22-state-space-MCARMA}, we see that the stationary representation of
$(\vect(X_{t}))_{t\in\MR}$ in~\eqref{eq:bk22-stationary-solution}, and hence
also of $(X_{t})_{t\in\MR}$, is considerably more complicated due to the
kernel $g_{2}$ in decomposition~\eqref{eq:bk22-complex-integral-2}, which in
the non-causal case does not vanish. We therefore assess the positivity of non-stable output processes
$(X_{t})_{t\in\MR}$ given by~\eqref{eq:bk22-CARMA-p-q-variation-of-constant}
as follows: We look for conditions on the model parameters
$\opA_{1},\ldots,\opA_{p}$ and $\opC_{0},\ldots, \opC_{q}$ such that for every
$s<t$ the process $(X_{t})_{t\geq s}$ is $C$-valued whenever $Z_{s}\in \cK^{p}$
and $(L_{t})_{t\in \MR}$ is $\cK$-increasing.
\begin{remark}
  \begin{enumerate}
  \item[i)] In the deterministic control literature a similar type of
    positivity is often referred to as \emph{internal positivity}, see,
    e.g.~\cite{FR00}. We call a state space model \emph{internal positive}, if
    for every non-negative initial value $Z_{s}$ and every $\cK$-valued input
    process $(L_{t})_{t\in\MR}$ the output process $(X_{t})_{t\geq s}$ assumes
    only values in $\cK$.  
  \item[ii)] Again we may assume that $Z_{s}$ is $\cF_{s}$-measurable,
    such that the process $(X_{t})_{t\geq s}$ is adapted to the natural
    filtration. In this case, however, $(X_{t})_{t\geq s}$ is not necessarily an MCARMA
    process anymore. Moreover, even if a unique stationary solution (MCARMA
    process) exists, conditions such that $(X_{t})_{t\geq s}$ is cone valued for
    all initial values $Z_{s}\in \cK^{p}$ are in general neither
    sufficient nor necessary for the MCARMA process to be cone valued. Indeed,
    the stationary distribution of $(Z_{t})_{t\in\MR}$, if it exists, might
    not be supported on $\cK^{p}$, see Section~\ref{sec:bk22-posit-semi-defin-1}
    for an example.
  \end{enumerate}  
\end{remark}

Let $J\subseteq \set{1,\ldots,p}$ and denote by $\cK^{J,p}$ the wedge\footnote{A
wedge $W\subseteq \MM_{n,m}$ is a convex cone without the property
$W\cap(-W)=\set{\zero_{n,m}}$.} in $(\MM_{n,m})^{p}$ consisting of the cone
$\cK$ in the in the $\cK$-th Cartesian coordinates (and $\MM_{n,m}$ otherwise). In
particular for $J=\set{1,\ldots,p}$ we have $\cK^{J,p}=\cK^{p}$ and
$\cK^{\set{1},p}=\cK\times \MM_{n,m}\times\dots\times \MM_{n,m}$. To ensure the
positivity of the term $\cC_{q}\E^{(t-s)\cA_{p}}Z_{s}$
in~\eqref{eq:bk22-CARMA-p-q-variation-of-constant} for every $Z_{s}\in
\cK^{p}$, we see that for all $t>s$ the operator $\cC_{q}\E^{(t-s)\cA_{p}}$ must
map $\cK^{p}$ into $\cK$. Note that in the case of a MCAR process, i.e. where
$\cC_{0}=[\opI,\op0,\ldots,\op0]$, one would think that $\E^{s\cA_{p}}$ has to
map $\cK^{p}$ onto $\cK^{\set{1},p}$ only, as the output operator $\cC_{q}$
only projects down onto the first block matrix component anyway. However, we
have the following:    

\begin{lemma}\label{lem:bkk22-pos-necessary-car}
  Let $(\cA_{p},E_{p},\cC_{q},L)$ be as above and let $(X_{t})_{t\geq s}$ be
  the associated non-stable output process
  in~\eqref{eq:bk22-CARMA-p-q-variation-of-constant}. Then the following holds
  true:   
  \begin{lemenum}
  \item\label{item:bk22-pos-necessary-car-1} If $\opC_{j}\in\pi(\cK)$ for all
    $j=0,\ldots,q$ and $\cA_{p}$ is quasi-positive with respect to $\cK^{p}$,
    then $(X_{t})_{t\geq s}$ is $\cK$-valued for all initial values $Z_{s}\in
    \cK^{p}$ and every $\cK$-increasing L\'evy process $(L_{t})_{t\in\MR}$.
  \item\label{item:bk22-pos-necessary-car-2} If $\cA_{p}$ is quasi-positive with
    respect to $\cK^{J,p}$ then $J=\set{1,2,\ldots,p}$.
  \item\label{item:bk22-pos-necessary-car-3} Conversely, if $(X_{t})_{t\geq s}$ is
    $\cK$-valued for all initial values $Z_{s}\in \cK^{p}$ and every $\cK$-increasing
    L\'evy process $(L_{t})_{t\in\MR}$, then $\opC_{j}\in\pi(\cK)$ for all
    $j=0,\ldots,q$. If even $\cK=\opC_{j}^{-1}(\cK)$ holds for all
    $j=0,\ldots,q$, then $\cA_{p}$ must be quasi-positive with respect to
    $\cK^{p}$. 
  \end{lemenum}
\end{lemma}
\begin{proof}
  We proof part i) first. Suppose $\cA_{p}$ is quasi-positive with respect
  to $\cK^{p}$, then by definition $\E^{t\cA_{p}}\opz\in \cK^{p}$
  for every $\opz\in \cK^{p}$ and $t\geq 0$. Hence for $Z_{s}\in \cK^{p}$ we have
  $\E^{(t-s)\cA_{p}}Z_{s}\in \cK^{p}$ and by definition of a $\cK$-increasing L\'evy
  process, we have $E_{p}(L_{s}-L_{s'})\in \cK^{p}$ for all $s>s'$. This implies
  that $\int_{s}^{t}\E^{(t-s)\cA_{p}}E_{p}\D L_{s}\in \cK^{p}$ for all $t>s$. If
  moreover $\opC_{j}\in\pi(\cK)$ for all $j=0,\ldots,q$ then
  $\cC_{q}(\E^{(t-s)\cA_{p}}Z_{s})\in \cK$ and
  $\cC_{q}(\int_{s}^{t}\E^{(t-s)\cA_{p}}E_{p}\D L_{s})\in \cK$, which according
  to~\eqref{eq:bk22-CARMA-p-q-variation-of-constant} yields $X_{t}\in \cK$ for all $t\geq s$.\\
  Next we show that the particular form of $\cA_{p}$ implies that it can only
  be quasi-positive on $(\MM_{n,m})^{p}$ with respect to the wedge $\cK^{J,p}$ if
  $J=\set{1,2,\ldots,p}$. Recall that quasi-positivity of $\cA_{p}$ with
    respect to $\cK^{J,p}$ can be equivalently characterized by the property
    that $\llangle \cA_{p}\bx, \bu \rrangle_{p}\geq 0$, whenever $\llangle
    \bx,\bu\rrangle_{p}=0$ with $\bx\in \cK^{J,p}$ and $\bu\in(\cK^{*})^{J,p}$ where
    we denote the inner-product on $(\MM_{n,m})^{p}$ by $\llangle
    \cdot,\cdot\rrangle_{p}$.\clearpage{}
    Therefore, let $\bx\in \cK^{J,p}$, $\bu\in(\cK^{*})^{J,p}$ and note that
    \begin{align}\label{eq:bk22-lemma-quasi-monotone} 
             \llangle \cA_{p}\bx,\bu\rrangle_{p}&=\llangle (x_{2},x_{3},\ldots,x_{p},
                                   \sum_{i=1}^{p}\opA_{i}(x_{i}))^{\intercal}, \bu\rrangle_{p}\nonumber\\
                                       &=\sum_{j=1}^{p-1}\langle x_{j+1}, u_{j}\rangle_{n,m}
   +\sum_{i=1}^{p}\langle \opA_{p+1-i}(x_{i}),u_{p}\rangle_{n,m}.  
\end{align}
Set $J^{c}\df\set{1,2,\ldots,p}\!\setminus\! J$ and suppose that $J^{c}$ is
non-empty. If $p\in J$, let $\bx\in\cK^{J,p}$ and $\bu \in (\cK^{*})^{J,p}$ be such that $u_{j}=0$ for
every $j\in J$ and $\langle x_{i},u_{i}\rangle_{n,m}=0$ for $i\in
J^{c}$, then clearly $\llangle \bx, \bu \rrangle_{p}=0$,
but from~\eqref{eq:bk22-lemma-quasi-monotone} we see that $\llangle
\cA_{p}\bx,\bu\rrangle_{p}=\sum_{j\in J^{c}}\langle x_{j+1},u_{j}\rangle_{n,m}$
which can be negative since we only assumed that $u_{j}\in \MM_{n,m}$ and
$x_{j+1}$ can be chosen arbitrary as long as $\langle
x_{j+1},u_{j+1}\rangle_{n,m}=0$ for $j+1\in J^{c}$. The case $p\in J^{c}$ then
follows by a similar argument and we see that $J^{c}$ must be empty or
otherwise $\cA_{p}$ can not be quasi-positive with respect to $\cK^{J,p}$. This
implies that we must have $J=\set{1,2,\ldots,p}$.   
\\
Lastly, for for the necessary direction in part iii). We suppose that
$X_{t}\in \cK$ for all $t\geq s$, $Z_{s}\in \cK^{p}$ and every $\cK$-increasing
L\'evy process $(L_{t})_{t\in\MR}$. In particular, at $t=s$ we have
$X_{s}=\cC_{q}Z_{s}=\sum_{i=1}^{q+1}\opC_{i-1}(Z^{(i)}_{s})\in \cK$ for all
$Z_{s}=(Z_{s}^{(1)},Z_{s}^{(2)},\ldots,Z_{s}^{(p)})\in \cK^{p}$. Therefore, if
we let $z\in \cK$ be arbitrary and set $z_{s}^{j}\df e_{j}\otimes z$ for
$j=1,\ldots,q$, then $z_{s}^{j}\in \cK^{p}$ and by assumption we must have
$X_{s}=\opC_{j-1}(z)\in \cK$ for all $j=1,\ldots,q+1$. Since $z\in \cK$ was
arbitrary we conclude that $\opC_{j-1}\in\pi(\cK)$ for all
$j=1,\ldots,q+1$. Next, assume that even $\cK=\opC_{j}^{-1}(\cK)$ holds for all
$j=0,\ldots,q$ and show that in this case $\cA_{p}$ must be quasi-positive with
respect to the cone $\cK^{p}$. For this let $Z_{s}\in \cK^{p}$ be fixed, but
arbitrary. Since the constant zero L\'evy process is $\cK$-increasing as well, it
follows by assumption that $(X_{t})_{t\in\MR}$ given by
$X_{t}=\cC_{q}(\E^{(t-s)\cA_{p}}Z_{s})$ is $\cK$-valued for all $s<t\in\MR$ and every
$Z_{s}\in \cK^{p}$. Since by the first part $\opC_{j}\in\pi(\cK)$ and by assumption
even $\cK=\opC_{j}^{-1}(\cK)$ for all $j=0,\ldots,q$, we see that for every $s<t$ there
exist a $J\subseteq \set{1,\ldots,p}$ such that $\set{1,\ldots,q}\subseteq J$
and $\E^{(t-s)\cA_{p}}Z_{s}\in \cK^{J,p}$. But since this holds for every
$s<t\in\MR$, we find in every neighborhood infinitely many time points
$(t_{j})_{j\in\MN}$ such that $\E^{(t_{j}-s)\cA_{p}}Z_{s}\in \cK^{J,p}$ holds
for the same $J$. From this we conclude that already
$\E^{t\cA_{p}}(\cK^{J,p})\subseteq \cK^{J,p}$ must hold for all $t\geq 0$ and some
set $J\subseteq \set{1,\ldots,p}$ (first prove this for rational time points and
then extend to irrational ones by continuity). This, however, by definition means that
$\cA_{p}$ is quasi-positive with respect to $\cK^{J,p}$ and thus we conclude
from part ii) that already $J=\set{1,2,\ldots,p}$ holds, i.e. that
$\cA_{p}$ must be quasi-positive with respect to $\cK^{p}$.    
\end{proof}

\vspace{-0.5mm}

The following theorem is our main result on the positivity of non-stable
output processes on $\MM_{n,m}$. Under an extra condition on the stationary
distribution of the non-stable output process, it also provides a sufficient
condition for the positivity of non-causal MCARMA processes.

\begin{theorem}\label{thm:bk22-positivite-non-stable-MCARMA}
  Let $p\in\MN$ and $q\in\MN_{0}$ with $q<p$. For $s<t\in\MR$, let
  $(X_{t})_{t\geq s}$ be the output process
  in~\eqref{eq:bk22-CARMA-p-q-variation-of-constant}, associated with
  $(\cA_{p},E_{p},\cC_{q},L)$ such that $\tau(\cA_{p})\geq 0$. Then the
  following holds true:     
\begin{theoremenum}
\item\label{item:bk22-positivite-non-stable-MCARMA-1} If $\opA_{1}\in\cL(\MM_{n,m})$
  is quasi-positive and $\opA_{i},\opC_{j}\in\pi(\cK)$ for $i=2,\ldots,p$ and
  $j=0,\ldots,q$, then $(X_{t})_{t\geq s}$ is $\cK$-valued for every initial
  value $Z_{s}\in \cK^{p}$ and $\cK$-increasing L\'evy process $(L_{t})_{t\in\MR}$.  
\item\label{item:bk22-positivite-non-stable-MCARMA-2} If
    $\cK=\opC_{j}^{-1}(\cK)$ for all $j=0,\ldots,q$. Then $(X_{t})_{t\geq s}$ is
    $\cK$-valued for every initial value $Z_{s}\in \cK^{p}$ and every
    $\cK$-increasing L\'evy process $(L_{t})_{t\in\MR}$ if and only if
    $\opA_{1}\in\cL(\MM_{n,m})$ is quasi-positive and $\opA_{i}\in\pi(\cK)$ for
    $i=2,\ldots,p$.     
\item\label{item:bk22-positivite-non-stable-MCARMA-3}If $\opC_{j}\in\pi(\cK)$
  for all $j=0,\ldots,q$ and there exists a stationary distribution of
  $(Z_{t})_{t\in\MR}$ supported on $\cK^{p}$, then the associated non-causal
  MCARMA process $(X_{t})_{t\in\MR}$ is $\cK$-valued.
\end{theoremenum}
\end{theorem}
\begin{proof}
We begin with the proof of part i). By~\cref{item:bk22-pos-necessary-car-1} it
is enough to show that $\opA_{i}\in\pi(\cK)$ and $\opA_{1}$ is quasi-positive with
respect to $\cK$ implies that $\cA_{p}$ is quasi-positive with respect to
$\cK^{p}$. Indeed, let $\bx=(x_{1},x_{2},\ldots,x_{p})^{\intercal}\in \cK^{p}$ and
$\bu=(u_{1},u_{2},\ldots,u_{p})^{\intercal}\in (\cK^{*})^{p}$ such that
$\llangle \bx, \bu \rrangle=0$ we then want to show that $\llangle \cA_{p}\bx,\bu\rrangle\geq 0$.
As before we have
\begin{align}\label{eq:bk22-quasi-monotonicitiy}
  \llangle \cA_{p}\bx,\bu\rrangle&=\sum_{j=1}^{p-1}\langle x_{j+1},
                                   u_{j}\rangle_{n,m}+\sum_{i=1}^{p}\langle
                                   \opA_{p-i+1}(x_{i}),u_{i}\rangle_{n,m},   
\end{align}
and due to the fact that $x_{j+1}\in \cK$, $u_{j}\in\cK^{*}$ for all $j=1,\ldots,p-1$, we see that $\langle x_{j+1},
u_{j}\rangle_{n,m}\geq 0$ and hence the first sum
in~\eqref{eq:bk22-quasi-monotonicitiy} is non-negative. Moreover, we see that $\langle
\opA_{p-i+1}(x_{i}),u_{i}\rangle\geq 0$ for $i=1,\ldots,p-1$ by assumption that
$\opA_{i}(\cK)\subseteq \cK$ for $i=2,\ldots,p$ and thus the remaining term of the second sum is 
$\langle \opA_{1}(x_{p}),u_{p}\rangle_{n,m}$. By assumption we have $\llangle
\bx,\bu \rrangle=\sum_{j=1}^{p}\langle x_{j},u_{j}\rangle_{n,m}=0$, and in
particular $\langle x_{p},u_{p}\rangle_{n,m}=0$, which by the quasi-positivity of
$\opA_{1}$ implies $\langle \opA_{1}(x_{p}),u_{p}\rangle_{n,m}\geq 0$. Hence
we see that $\llangle \cA_{p}\bx,\bu\rrangle_{p}\geq 0$ whenever $\llangle \bx, \bu
\rrangle_{p}=0$, which proves the quasi-positivity of $\cA_{p}$ with respect
to $\cK^{p}$.\\
The second assertion is a consequence
of~\cref{item:bk22-pos-necessary-car-3} and it is only left to prove that the
quasi-positivity of $\cA_{p}$ with respect to $\cK^{p}$ implies that
$\opA_{1}$ is quasi-positive and $\opA_{i}\in\pi(\cK)$ for $i=2,\ldots,p$. For
this, suppose that $\cA_{p}$ is quasi-positive with respect to
$\cK^{p}$, then for every $\bx\in \cK^{p}$, $\bu\in(\cK^{*})^{p}$ with $\llangle \bx, \bu
\rrangle_{p}=0$ we find that the term in~\eqref{eq:bk22-quasi-monotonicitiy}
is non-negative. If we let $x_{1}\in \cK$ and $u_{p}\in\cK^{*}$ be such that $\langle
x_{1},u_{p}\rangle_{n,m}=0$ and if we set $\bu=(0,\ldots,0,u_{p})^{\intercal}$ and
$\bx=(x_{1},0,\ldots,0)^{\intercal}$, then we observe that $\llangle \bx, \bu
\rrangle_{p}=0$ and~\eqref{eq:bk22-quasi-monotonicitiy} reduces to $\langle \opA_{1}(x_{p}),u_{p}\rangle_{n,m}\geq 0$.
Since $x_{1}$ and $u_{p}$ satisfy $\langle x_{1},u_{p}\rangle_{n,m}=0$ but are
otherwise arbitrary, it follows that $\opA_{1}$ is quasi-positive. Moreover,
we see that $\opA_{i}\in\pi(\cK)$ for $i=2,\ldots,p$ follows by a similar argument. The
third assertion follows immediately from~\eqref{eq:bk22-CARMA-p-q-variation-of-constant}, since by assumption
$Z_{t}$ is supported on $\cK^{p}$ for all $t\in\MR$ and $\cC_{q}$ maps by
assumption $\cK^{p}$ into $\cK$. 
\end{proof}

\begin{remark}
Note that the positivity condition
in~\cref{item:bk22-positivite-non-stable-MCARMA-1} only applies in the
non-stable case, i.e. for $\tau(\cA_{p})\geq 0$. Indeed, suppose that
$\cA_{p}$ is quasi-positive with respect to $\cK^{p}$ and $\tau(\cA_{p})<0$,
then by an application of Lemma~\ref{lem:bkk22-inverse-quasi-positive} to
$\cA_{p}$ and the cone $\cK^{p}$ implies that the inverse $\cA_{p}^{-1}$ exists and must map positive vectors into negatives,
i.e. $\cA_{p}^{-1}(\cK^{p})\subseteq -\cK^{p}$. However, the inverse of $\cA_{p}$ is
explicitly known as
\begin{align}
  \cA_{p}^{-1}=
  \begin{bmatrix}
    \opA_{p}^{-1}\opA_{p-1} & \opA_{p}^{-1}\opA_{p-2} & \ldots & \opA_{p}^{-1}\\
    \opI & \op0 & \ldots & \op0 \\
    \op0 & \opI & \ddots & \vdots \\
    \vdots & \ddots   & \ddots & \op0 \\
    \op0 & \ldots & \op0 & \opI  
  \end{bmatrix},                  
\end{align}
and obviously never satisfies $\cA_{p}^{-1}(\cK^{p})\subseteq -\cK^{p}$. Conversely,
note that if $\opA_{i}$ for $i=1,\ldots,p$ satisfies the assumptions
in~\cref{item:bk22-positivite-non-stable-MCARMA-1}, then $\cA_{p}$ can not
satisfy $\tau(\cA_{p})<0$. This is so since $\sigma(\cA_{p})=\sigma(\opA_{p})$
and by the Perron-Frobenius theorem, see also~\cite[Proposition 1]{HL98}, we
know that there exists at least one leading eigenvalue of $\opA_{p}$ which is
non-negative  and hence $\tau(\cA_{p})\geq 0$.     
\end{remark}

In the following two corollaries we concretize the positivity criteria in
Theorem~\ref{thm:bk22-positivite-non-stable-MCARMA} for the case of
$\MR_{d}$-valued and $\MS_{d}$-valued non-stable output processes.

\begin{corollary}\label{coro:non-stationary-positive-Rd}
  Assume that $m=1$ and $n=d$ for some $d\in\MN$ and let $p\in\MN$,
  $q\in\MN_{0}$ such that $q<p$. Moreover, let $(L_{t})_{t\in\MR}$ be an
  $\MR_{d}^{+}$-valued L\'evy process and for $s<t\in\MR$ we denote by $(X_{t})_{t\geq
    s}$ the non-stable output process associated with
  $(\cA_{p},E_{p},\cC_{q},L)$ such that $\tau(\cA_{p})\geq 0$ and
  $(A_{i})_{i=1,\ldots,p}$ and $(C_{j})_{j=0,\ldots,q}$ satisfy the following conditions:
  \begin{enumerate}
  \item[i)] For every $i=2,\ldots,p$ we have $A_{i}=(a_{k,n}^{(i)})_{1\leq
      k,n\leq d}$ such that $a_{k,n}^{(i)}\geq 0$ for all $1\leq k,n\leq d$;
  \item[ii)] $A_{1}=(a^{(1)}_{k,n})_{1\leq k,n\leq d}$ is such that
    $a^{(1)}_{k,n}\geq 0$ for all $1\leq k,n\leq d$ with $k\neq n$;
  \item[iii)] For every $j=0,\ldots,q$ we have $C_{j}=(c_{k,n}^{(j)})_{1\leq
      k,n\leq d}$ such that $c_{k,n}^{(j)}\geq 0$ for all $1\leq k,n\leq d$.
  \end{enumerate}
  Then $(X_{t})_{t\geq s}$ is $\MR_{d}^{+}$-valued, whenever $Z_{s}\in \MR_{pd}^{+}$. 
\end{corollary}

Following Corollary~\ref{coro:non-stationary-positive-Rd} the condition on the
transition matrix $\cA_{p}\in\MM_{pd}$ can be visualized as follows:

{\small \begin{align*}
  \cA^{\vect}_{p}=
  \left(\begin{array}{ccccc}
    \zero_{d} & \MI_{d} & \zero_{d} & \ldots & \zero_{d} \\
    \zero_{d} & \zero_{d} & \MI_{d} &   & \vdots \\
    \vdots &  &     & \ddots \qquad & \vdots \\
          \vdots &  & \qquad\ddots   &   & \zero_{d} \\
    \zero_{d} & \ldots & \ldots & \zero_{d} & \MI_{d}     \\
     & & & & \\
    a^{(p)}_{1,1} \ldots a^{(p)}_{1,d} & a^{(p-1)}_{1,1}
    \ldots a^{(p-1)}_{1,d} & \ldots & \ldots & a^{(1)}_{1,1}
    \ldots a^{(1)}_{1,d}\\
    \vdots\quad\ddots\quad\vdots & \vdots\quad\ddots\quad\vdots & \quad \ldots \quad  & \quad\ldots \quad & \vdots\quad\ddots\quad\vdots\\
    \underbrace{a^{(p)}_{d,1} \ldots
      a^{(p)}_{d,d}}_{a^{(p)}_{k,n}\geq 0,\;\forall k,n} & \underbrace{a^{(p-1)}_{d,1}
    \ldots a^{(p-1)}_{d,d}}_{a^{(p-1)}_{k,n}\geq 0, \;\forall k,n} &
  \ldots & \ldots & \underbrace{ a^{(1)}_{d,1} \ldots
    a^{(1)}_{d,d}}_{a^{(1)}_{k,n}\geq 0,\;\forall k\neq n}
  \end{array}\right)
\end{align*}}

We obtain an analogous result for $\MS_{d}^{+}$-valued non-stable output processes:

\begin{corollary}\label{coro:CAR-p-positive}
  Assume that $m=n=d$ for some $d\in\MN$ and let $p\in\MN$,
  $q\in\MN_{0}$ such that $q<p$. For $s<t\in\MR$ we denote by $(X_{t})_{t\geq
    s}$ the output process associated with
  $(\cA_{p},E_{p},\cC_{q},L)$ such that $\tau(\cA_{p})\geq 0$ and
  $(\opA_{i})_{i=1,\ldots,p}$ and $(\opC_{j})_{j=0,\ldots,q}$ satisfy the
  following conditions:
  \begin{enumerate}
  \item[i)] For every $i=2,\ldots,p$ there exists $a_{i}\in \MM_{d}$ such that
    $\opA_{i}(x)=a_{i}xa_{i}^{*}$ for every $x\in\MS_{d}$;
  \item[ii)] There exists an $a_{1}\in\MM_{d}$ such that $\opA_{1}(x)=a_{1}x+xa_{1}^{*}$ for
    every $x\in\MS_{d}$;
  \item[iii)] For every $j=0,\ldots,q$ there exists $c_{j}\in \MM_{d}$ such that
    $\opC_{j}(x)=c_{j}xc_{j}^{*}$ for every $x\in\MS_{d}$.
  \end{enumerate}
  Then $(X_{t})_{t\geq 0}$ is $\MS_{d}^{+}$-valued, whenever $Z_{s}\in
  (\MS_{d}^{+})^{p}$. Moreover, the state-space representation of the process
  $(\vect(X_{t}))_{t\geq 0}$
  in~\eqref{eq:bk22-vect-X-state-space-rep-1}-\eqref{eq:bk22-vect-X-state-space-rep-2}
  holds with state transition operator $\cA_{p}^{\vect}$ given by 
  {\small\begin{align*}
    \cA^{\vect}_{p}\df
    \begin{pmatrix}
      \zero_{d^{2}} & \MI_{d^{2}} & \zero_{d^{2}} & \ldots & \zero_{d^{2}} \\
      \zero_{d^{2}} & \zero_{d^{2}} & \MI_{d^{2}} & \ddots & \vdots \\
      \vdots &  &  & \ddots & \vdots \\
      \vdots &  & \ddots   &   & \zero_{d^{2}} \\
      \zero_{d^{2}} & \ldots & \ldots & \zero_{d^{2}} & \MI_{d^{2}}     \\
      a_{p}\otimes a_{p} & a_{p-1}\otimes a_{p-1} & \ldots & \ldots &
      \MI_{d}\otimes a_{1}+a_{1}\otimes \MI_{d} 
    \end{pmatrix},
  \end{align*}}
  output operator $\cC_{q}^{\vect}$ given by $\cC_{q}^{\vect}=[c_{0}\otimes c_{0}, c_{1}\otimes
  c_{1},\ldots,c_{q}\otimes c_{q},\zero_{d^{2}},\ldots,\zero_{d^{2}}]$ and the
  input operator $E_{p}^{\vect}$ given by $E_{p}^{\vect}=e_{p}\otimes \MI_{d^{2}}$. 
\end{corollary}
\begin{proof}
This follows from Theorem~\ref{thm:bk22-positivite-non-stable-MCARMA} and the
fact that maps of the form $x\mapsto axa^{*}$ for $a\in\MM_{d}$ are in
$\pi(\MS_{d}^{+})$ and maps of the form $x\mapsto ax+xa^{*}$ are
quasi-positive. Moreover, note that 
$A^{\vect}_{1}(x)=\big(\MI_{d}\otimes a_{1}+a_{1}\otimes \MI_{d}\big)\vect(x)$
and $A_{i}^{\vect}(x)=a_{i}\otimes a_{i}\vect(x)$ for $i=2,\ldots,p$ and the
analogous assertions hold for the operators $(C_{j})_{j=0,\ldots,q}$. 
\end{proof}

\section{MCARMA based stochastic covariance models}\label{sec:bk22-mcarma-based-mult} 

In this section we study multivariate stochastic covariance models based on
symmetric and positive semi-definite MCARMA processes. We begin this section
with a brief introduction to general stochastic covariance modeling and we
recall some results on the second order structure of the return and covariance
processes in Section~\ref{sec:bk22-stoch-covar-models}. Subsequently, in
Section~\ref{sec:bk22-posit-semi-defin-1}, we discuss the multivariate
Barndorff--Nielsen and Shepard (BNS) volatility model from~\cite{PS09b} as a
particular example of an MCAR based stochastic covariance model of order one. We
motivate the use of higher-order MCARMA based covariance models to capture
certain short-lag memory effects in realized variance and cross-covariance
time-series. In Section~\ref{sec:bk22-posit-semi-defin-2}, we demonstrate the
gained flexibility in the second order structure when using higher-order MCARMA
processes through an exemplary analysis of \emph{positive semi-definite well-balanced Ornstein-Uhlenbeck processes}.   

\subsection{Stochastic covariance models and their second order structure}\label{sec:bk22-stoch-covar-models}
Let $d\in\MN$. We call an $\MR_{d}\times\MS_{d}^{+}$-valued process
$(Y_{t},X_{t})_{t\geq 0}$ a \emph{stochastic covariance model} on $\MR_{d}$,
whenever it consists of a \emph{spot covariance process} $(X_{t})_{t\geq 0}$
on $\MS_{d}^{+}$ and a $d$-dimensional logarithmic asset price process
$(Y_{t})_{t\geq 0}$ given by a stochastic differential equation of the form  
\begin{align}\label{eq:bk22-log-asset-price-model}
  \begin{cases}
    \D Y_{t}= (\alpha +X_{t}\beta)\D t+X_{t}^{1/2}\D W_{t},\quad t>0,\\
    Y_{0}=y\in\MR_{d}, 
  \end{cases}
\end{align}
where $\alpha\in\MR_{d}$ is the drift, $\beta\in\MR_{d}$ the risk-premia,
$(W_{t})_{t\geq 0}$ denotes an $\MR_{d}$-valued (standard) Brownian motion, see also~\cite{BNS13, PS09b}. The spot covariance process $(X_{t})_{t\geq 0}$ is assumed to
be integrable with respect to $(W_{t})_{t\geq 0}$ and $\MS_{d}^{+}$-valued
such that for all $t\geq 0$ the matrix square-root $X_{t}^{1/2}$ exists and the
stochastic integral in~\eqref{eq:bk22-log-asset-price-model} is
well-defined. The main challenge in stochastic covariance modeling is the appropriate specification of the spot
covariance process $(X_{t})_{t\geq 0}$ such that the model
$(Y_{t},X_{t})_{t\geq 0}$ represents the stylized facts of financial data, while being sufficiently tractable for, e.g. simulations, statistical inference
or option pricing. For option pricing in multivariate
  stochastic covariance model of OU-type, see~\cite{MKPS12}. Moreover, note that
  from a modeling perspective it would be reasonable to add a leverage term
  $\rho(\D L_{t})$ to the dynamics of $(Y_{t})_{t\geq 0}$, where $\rho\colon
  \MS_{d}\to\MR_{d}$ is some operator and $(L_{t})_{t\geq 0}$ is the driving
  L\'evy process of the spot covariance process $(X_{t})_{t\geq 0}$, see also
  ~\cite{MKPS12}. However, in these notes we omit the leverage term in
  model~\eqref{eq:bk22-log-asset-price-model} for the
  sake of simplicity.\par{} 
In contrast to the logarithmic price process $(Y_{t})_{t\geq 0}$, the spot
covariance process $(X_{t})_{t\geq 0}$ is not directly observable in markets. However, it can be measured indirectly from the \emph{realized
  covariance} of the (squared) return process as follows: First, assume that
$(X_{t})_{t\geq 0}$ is square-integrable and stationary and for any real
positive number $\Delta$ we define the discrete-time process
$(Y_{n}^{\Delta})_{n\in\MN}$ by
\begin{align*}
  Y_{n}^{\Delta}\df Y_{n\Delta}-Y_{(n-1)\Delta}=\int_{(n-1)\Delta}^{n\Delta}(\alpha+X_{t}\beta)\D
  t+\int_{(n-1)\Delta}^{n\Delta}X_{t}^{1/2}\D W_{t},\quad n\in\MN.   
\end{align*}
The process $(Y_{n}^{\Delta})_{n\in\MN}$ is the sequence of logarithmic returns over
the intervals $[(n-1)\Delta,n\Delta]$ and it can be seen that for every
$n\in\MN$ the random variable $Y^{\Delta}_{n}$ given $X^{\Delta}_{n}$ is
normal distributed. More precisely, for every $n\in\MN$ we have
\begin{align*}
  Y^{\Delta}_{n}|X^{\Delta}_{n}\sim
  \mathcal{N}(\alpha\Delta+X^{\Delta}_{n}\beta,X^{\Delta}_{n})\quad\text{with
  }X_{n}^{\Delta}\df \int_{(n-1)\Delta}^{n\Delta}X_{t}\D
  t=X_{n\Delta}^{+}-X_{(n-1)\Delta}^{+},
\end{align*}
where $X_{t}^{+}\df \int_{0}^{t}X_{t}\D s$ for $t\geq 0$, see
also~\cite{BNS01}. We define the auto-covariance function
$\mathrm{acov}_{X_{t}}\colon \MRplus\to \cL(\MM_{d})$ of the process
$(X_{t})_{t\geq 0}$ at $t\geq 0$ by $\mathrm{acov}_{X_{t}}(h)\df
\mathrm{Cov}[\vect(X_{t+h}),\vect(X_{t})]$ where $h\geq 0$;
Analogously, we define the autocovariances $\mathrm{acov}_{X_{n}^{\Delta}}$ and
$\mathrm{acov}_{Y_{n}^{\Delta}(Y_{n}^{\Delta})^{\intercal}}(h)$ for the
processes $Y_{n}^{\Delta}(Y_{n}^{\Delta})^{\intercal}$ and
$(X^{\Delta}_{n})_{n\in\MN}$, respectively, in which case we restrict to $n,h\in\MN$. Following~\cite{PS09b}, we observe that the second order structure
of the spot covariance process $(X_{t})_{t\geq 0}$, respectively its discrete
difference process $(X^{\Delta}_{n})_{n\in\MN}$, is inherited by the second order
structure of the squared logarithmic return process
$(Y_{n}^{\Delta}(Y_{n}^{\Delta})^{\intercal})_{n\in\MN}$ such that
\begin{align}\label{eq:bk22-auto-covariance-Yn}
\mathrm{acov}_{Y_{n}^{\Delta}(Y_{n}^{\Delta})^{\intercal}}(h)=\mathrm{acov}_{X_{n}^{\Delta}}(h),\quad\text{for
  all }h,n\in\MN.
\end{align}
Moreover, we recall from~\cite[Theorem 3.2]{PS09b} that for general square-integrable stationary spot covariance processes
$(X_{t})_{t\geq 0}$ we have 
\begin{align}\label{eq:bk22-auto-covariance-Xn}
  \mathrm{acov}_{X_{n}^{\Delta}}(h)=r^{++}(h\Delta+\Delta)-2r^{++}(h\Delta)+r^{++}(h\Delta-\Delta),\quad
  \forall h,n\in\MN,
\end{align}
where $r^{++}\colon \MRplus\to \cL(\MM_{d})$ is given by
\begin{align}\label{eq:bk22-r++}
r^{++}(t)\df\int_{0}^{t}\int_{0}^{s}\mathrm{acov}_{X_{n}}(u)\D u\D s,\quad t\geq 0.
\end{align}

As we noted before, a prominent feature observed in many realized
(co)variance time-series is the \emph{memory effect}, which means that the
observed auto-covariances of the (squared) returns exhibits slower than pure
exponential decay or allows non-monotone configurations, see, e.g.~\cite{BL13, Bro14} and
the reference therein. To capture a great variety of different memory effects, we
propose to model the spot covariance process $(X_{t})_{t\geq 0}$
in~\eqref{eq:bk22-log-asset-price-model} by $\MS^{+}_{d}$-valued higher-order
MCARMA processes as they were introduced in this work. In the univariate case the idea to model the spot
variance process in stochastic volatility models by positive (higher-order)
CARMA processes goes back to~\cite{BL13}. To demonstrate the potential of
MCARMA based stochastic covariance models we compare the two auto-covariances
in~\eqref{eq:bk22-auto-covariance-Yn} and~\eqref{eq:bk22-auto-covariance-Xn}
in the following two cases: First, for the multivariate BNS model in
Section~\ref{sec:bk22-posit-semi-defin-1}, which was studied
in~\cite{PS09b}. Secondly, for a stochastic covariance model based on a
\emph{positive semi-definite well-balanced OU} processes introduced in
Section~\ref{sec:bk22-posit-semi-defin-2}. 

\subsection{Positive semi-definite Ornstein-Uhlenbeck type processes}\label{sec:bk22-posit-semi-defin-1} 
Stochastic covariance models based on positive semi-definite OU type processes
were studied extensively in~\cite{BNS07, PS09, PS09b} and we refer to it as
the (multivariate) BNS stochastic volatility model. The class of
matrix-valued OU processes is included in the class of matrix-valued MCARMA
processes as they form the class of MCAR processes of order one, see
Definition~\ref{def:bk22-CARMA-p-q}. We show that the positivity criteria in
Theorem~\ref{thm:bk22-positive-stationary} in case of a causal MCAR
process of order one coincides with the well-known (sufficient) criteria for
symmetric and positive semi-definite OU processes
in~\cite{BNS07}. Subsequently, we then recall the second order structure of OU
based stochastic covariance models from~\cite{PS09b}.\par{} 

Let $p=1$, $q=0$ and set $\cA_{1}=-\opA$ for some
$\opA\in\cL(\MM_{d})$ such that~\eqref{eq:bk22-stationary-solution-condition} is
satisfied and let $\cC_{q}=\opI$. Moreover, let $L$ be a two-sided L\'evy
process on $\MM_{d}$ and assume that $\EX{\log(\norm{L_{1}}_{d^{2}})}<\infty$. We
denote the MCAR process associated with the state space representation
$(\cA_{1},\opI,\opI,L)$ by $(X_{t})_{t\in\MR}$. This process is
an Ornstein-Uhlenbeck type process on $\MM_{d}$ and has the following
representation:  
\begin{align*}
  X_{t}=\E^{-(t-s)\opA}Z_{s}+\int_{s}^{t}\E^{-(t-u)\opA}\D L_{u},\quad s<t\in\MR.
\end{align*}
If we further assume that $\tau(-\opA)<0$, then
$X_{t}=\int_{-\infty}^{t}\E^{-(t-s)\opA}\D L_{s}$ is the unique stationary OU process adapted to the natural filtration of
$(L_{t})_{t\in\MR{}}$. It follows from~\cref{item:bk22-positive-stationary-2}, that
whenever $(L_{t})_{t\in\MR}$ is $\MS_{d}^{+}$-increasing, a sufficient
condition for $(X_{t})_{t\in\MR}$ to be $\MS_{d}^{+}$-valued is the
quasi-positivity of the operator $-\opA$ with respect to $\MS_{d}^{+}$. By recalling the definition of quasi-positivity this means
$\E^{-t\opA}(\MS_{d}^{+})\subseteq \MS_{d}^{+}$ for all $t\geq 0$ and we see
that this criteria coincides with the usual sufficient condition
for OU processes to be symmetric and positive semi-definite,
see~\cite[Proposition 4.1]{BNS07}. In the following we recall some statistical properties of $\MS_{d}^{+}$-valued
stationary OU process that follow from Propositions~\ref{prop:bk22-output-process-properties}
and~\ref{prop:bk22-second-order-property} and \cite[Proposition 4.7]{BNS07}:\par{} For
all $t\geq 0$ the mean of $X_{t}$ is given by 
\begin{align}\label{eq:bk22-mean-OU}
  \EX{X_{t}}=\opA^{-1}\mu_L=\opA^{-1}\left(\gamma_{L}+\int_{\MS_{d}^{+}\cap \set{\norm{\xi}_{d}>1}}\xi\,\nu^{L}(\D\xi)\right),
\end{align}
and the auto-covariance of $(X_{t})_{t\geq 0}$ at $t\geq 0$ is given by
\begin{align}\label{eq:bk22-covariance-OU}
  \Cov
  \left[\vect(X_{t}),\vect(X_{t+h})\right]=\E^{-h A^{\vect}}\cD^{-1}Q^{\vect},\quad
  h\geq 0,
\end{align}
where $\cQ^{\vect}=\vect\circ \cQ \circ \vect^{-1}$ and $\cD\in\cL(\MM_{d^{2}})$ is given by
\begin{align}\label{eq:bk22-D-operator}
\cD(X)=A^{\vect}X+X(A^{\vect})^{\intercal}\quad\text{with }A^{\vect}=\vect\circ\opA\circ\vect^{-1}.
\end{align}
Moreover, we have
\begin{align}
  \int_{0}^{t}X_{t}\D t=-\opA^{-1}(X_{t}-X_{0}-L_{t}),\quad t\in\MR,
\end{align}
and 
\begin{align}\label{eq:bk22-r++OU}
r^{++}(t)=\Big((A^{\vect})^{-2}\big(\E^{-A^{\vect} t}-\MI_{d^{2}}\big)+(A^{\vect})^{-1}t\Big)\cD^{-1}\cQ^{\vect} ,\quad t\geq 0, 
\end{align}
which by~\eqref{eq:bk22-auto-covariance-Yn}
and~\eqref{eq:bk22-auto-covariance-Xn} for every $n\in\MN$ yield
\begin{align}\label{eq:bk22-auto-covariance-squared-returns-OU}
  \mathrm{acov}_{Y_{n}^{\Delta}(Y_{n}^{\Delta})^{\intercal}}(h)=\E^{-A^{\vect}\Delta(h-1)}(A^{\vect})^{-2}\big(\MI_{d^{2}}-\E^{-A^{\vect}\Delta}\big)^{2}\cD^{-1}\cQ^{\vect},\quad
  h\in\MN. 
\end{align}

\subsection{Positive semi-definite well-balanced Ornstein-Uhlenbeck
  processes}\label{sec:bk22-posit-semi-defin-2}

In this section we introduce \emph{matrix-valued well-balanced OU processes} that extend the univariate
well-balanced OU processes from~\cite{SW11} to $\MS_{d}$-valued processes. The
authors in~\cite{SW11} proposed positive well-balanced OU processes as a model
for the spot variance process in stochastic volatility models. Here, we extend
this idea to the multivariate setting by studying $\MS_{d}^{+}$-valued
well-balanced OU processes and show that this class is well-suited to model the
spot covariance process in multivariate stochastic covariance
models. Moreover, we show that stochastic covariance models based on
well-balanced OU processes exhibit auto-covariance functions of the squared
logarithmic returns $Y_{n}^{\Delta}$ that have slower decay compared to the multivariate
BNS model.\par{} 
Let $\opA\in\cL(\MS_{d})$ with $\tau(-\opA)<0$ and define $\cA_{2}\in
\cL((\MM_{d})^{2})$, $E_{2}\in\cL(\MM_{d},(\MM_{d})^{2})$ and $\cC_{0}\in \cL((\MM_{d})^{2},\MM_{d})$ by
\begin{align}\label{eq:bk22-well-balanced-OU-state-space}
  \cA_{2}\df \begin{bmatrix}
    \op0 & \opI \\
    \opA^{2} & \op0
  \end{bmatrix},\quad E_{2}\df
                         \begin{bmatrix}
                           \op0\\
                           \opI
                         \end{bmatrix}
\quad\text{and}\quad \cC_{0}\df[-2 \opA,\op0].
\end{align}
Moreover, let $(L_{t})_{t\in\MR}$ be a square-integrable
$\MS_{d}^{+}$-increasing L\'evy process with expectation
$\mu_{L}\in\MS_{d}^{+}$ and covariance operator $\cQ\in \cL(\MS_{d})$.  
Let $(X_{t})_{t\geq 0}$ denote the output process of the state space model
associated with $(\cA_{2},E_{2},\cC_{0},L)$ and initial value $Z_{0}$, where
$Z_{0}$ is a random element in $(\MS_{d})^{2}$ (not necessarily
$\cF_{0}$-measurable). The output process $(X_{t})_{t\geq 0}$ is given by
\begin{align}\label{eq:bk22-well-balanced-OU}
  X_{t}&=\cC_{0}\E^{t\cA_{2}}Z_{0}+\int_{0}^{t}\cC_{0}\E^{(t-s)\cA_{2}}E_{2}\D
         L_{s},\quad t\geq 0,
\end{align}
and we note that since $\sigma(\cA_{2})=\sigma(\opA^{2})\subseteq
\MRplus\setminus\set{0}+\I\MR $ the process $(X_{t})_{t\geq 0}$ is non-stable,
but~\eqref{eq:bk22-stationary-solution-condition} is satisfied and thus following
Proposition~\ref{prop:bk22-state-space-MCARMA}, there exist a unique stationary
solution to~\eqref{eq:bk22-well-balanced-OU}.\par{}
In the following proposition we
define \emph{positive semi-definite well-balanced OU processes} as the unique
stationary and positive semi-definite solution
to~\eqref{eq:bk22-well-balanced-OU}. Moreover, we specify the stationary
distribution and present sufficient positivity conditions for $(X_{t})_{t\geq 0}$. 

\begin{proposition}\label{prop:bk22-positive-well-balanced-OU}
Let $(\cA_{2},E_{2},\cC_{0},L)$ be as above and in addition assume that $-\opA$
is quasi-positive. Let $(X_{t})_{t\geq 0}$ be as
in~\eqref{eq:bk22-well-balanced-OU} with $Z_{0}=(Z_{0}^{(1)},Z_{0}^{(2)})$
given by $Z_{0}^{(1)}=-\opA^{-1}\frac{(\pi_{1}+\pi_{2})}{2}$ and
$Z_{0}^{(2)}=\frac{\pi_{1}-\pi_{2}}{2}$ where $\pi_{1}\df\int_{-\infty}^{0}\E^{s\opA}\D L_{s}$ and
$\pi_{2}\df\int_{0}^{\infty}\E^{-s\opA}\D L_{s}$. Then $(X_{t})_{t\geq 0}$
is stationary, $\MS_{d}^{+}$-valued and can be represented as
\begin{align}\label{eq:bk22-well-balanced-representation}
 X_{t}=\int_{-\infty}^{t}\E^{-(t-s)\opA}\D
  L_{s}+\int_{t}^{\infty}\E^{-(s-t)\opA}\D L_{s},\quad t\geq 0.
\end{align}
Moreover, for all $t\geq 0$ we have
\begin{align}\label{eq:bk22-well-balanced-expectation}
  \EX{X_{t}}= 2 \opA^{-1}\mu_{L},
\end{align}
and the auto-covariance of $(X_{t})_{t\geq 0}$ at $t\geq 0$ is given by
\begin{align}\label{eq:bk22-well-balanced-auto-covariance}
  \Cov\left[\vect(X_{t+h}),
  \vect(X_{t})\right]&=\E^{-hA^{\vect}}\big(\E^{h\hat{\cD}}-\opI\big)\hat{\cD}^{-1}\cQ^{\vect}\nonumber\\
                     &\quad +2 \E^{-h A^{\vect}}\cD^{-1}\cQ^{\vect},\quad h\geq 0,   
\end{align}
where $\hat{\cD}\in\cL(\MS_{d})$ is given by
$\hat{\cD}(X)=A^{\vect}X-X(A^{\vect})^{\intercal}$, $\cD$ is as
in~\eqref{eq:bk22-D-operator} and $A^{\vect}\df\vect\circ
\opA\circ\vect^{-1}$. We call the process
$(X_{t})_{t\geq 0}$ in~\eqref{eq:bk22-well-balanced-representation} a
\emph{positive semi-definite well-balanced OU process}.    
\end{proposition}
\begin{proof}
From the particular anti-diagonal form of $\cA_{2}$ we see that for every
$k\in\MN$ the following holds true:
\begin{align*}
  \cA_{2}^{2k}=\begin{bmatrix}
    (\opA^{2})^{k}& \op0 \\
    \op0 & (\opA^{2})^{k} 
  \end{bmatrix}\quad \text{and} \quad   \cA_{2}^{2k+1}=\begin{bmatrix}
    \op0 & (\opA^{2})^{k} \\
    (\opA^{2})^{k+1} & \op0
  \end{bmatrix},
\end{align*}
which gives
\begin{align*}
  \E^{t\cA_{2}}=\sum_{k=0}^{\infty}\frac{(t\cA_{2})^{2k}}{(2k)!}+\sum_{k=0}^{\infty}\frac{(t\cA_{2})^{2k+1}}{(2k+1)!}=\begin{bmatrix}
    \cosh(t\opA)  & \sinh(t\opA)\opA^{-1} \\
    \sinh(t\opA)\opA & \cosh(t\opA)
  \end{bmatrix},\quad t\geq 0,
\end{align*}
where $\cosh(t\opA)\df
\sum_{k=0}^{\infty}\frac{(t\opA)^{2k}}{2k!}$
and $\sinh(t\opA)\df
\sum_{k=0}^{\infty}\frac{(t\opA)^{2k+1}}{(2k+1)!}$. Note further, that
$\sinh(\opA)$, $\cosh(\opA)$ and $\opA$ all commute mutually. Hence
by~\eqref{eq:bk22-well-balanced-OU} we obtain
\begin{align}
  X_{t}&=[-2\opA,\op0]
         \begin{pmatrix}
           \cosh(t\opA)Z_{0}^{(1)}+\sinh(t\opA)\opA^{-1}Z_{0}^{(2)}\nonumber\\
           \sinh(t\opA)\opA Z_{0}^{(1)}+\cosh(t\opA)Z_{0}^{(2)}
         \end{pmatrix}
  \\
  &\quad+\int_{0}^{t}[-2\opA,\op0]
  \begin{pmatrix}
    \sinh((t-s)\opA)\opA^{-1}\D L_{s}\nonumber\\
    \cosh((t-s)\opA)\D L_{s}
  \end{pmatrix}\\
       &=-2\opA\cosh(t\opA )Z_{0}^{(1)}-2\sinh(\opA)Z_{0}^{(2)}-2\int_{0}^{t}\sinh((t-s)\opA)\D
         L_{s}\nonumber\\
       &=-\opA \E^{t\opA}Z_{0}^{(1)}-\opA\E^{-t\opA}Z_{0}^{(1)}+X_{t}^{(1)}-X_{t}^{(2)},\label{eq:bk22-well-balanced-stationary-1}
\end{align}
where in the last line~\eqref{eq:bk22-well-balanced-stationary-1} we used that $\cosh(t\opA)=\frac{1}{2}(\exp(t\opA)+\exp(-t\opA))$ and
$\sinh(t\opA)=\frac{1}{2}(\exp(t\opA)-\exp(-t\opA))$ and set
$X_{t}^{(1)}\df\E^{-t \opA}Z_{0}^{(2)}+\int_{0}^{t}\E^{-(t-s)\opA}\D L_{s}$, for
$t\geq 0$, as well as $X_{t}^{(2)}\df\E^{t\opA}Z_{0}^{(2)}+\int_{0}^{t}\E^{(t-s)\opA}\D
L_{s}$.\par{} Now, recall $\pi_{1}=\int_{-\infty}^{0}\E^{s\opA}\D L_{s}$ and
$\pi_{2}=\int_{0}^{\infty} \E^{-s\opA}\D L_{s}$, respectively, and inserting
the initial state $Z_{0}^{(1)}=-\opA^{-1}\frac{1}{2}(\pi_{1}+\pi_{2})$ and
$Z_{0}^{(2)}= \frac{1}{2}(\pi_{1}-\pi_{2})$
into~\eqref{eq:bk22-well-balanced-stationary-1} yields   
\begin{align}
  X_{t}&= \E^{t\opA}\pi_{2}+\int_{0}^{t}\E^{-(t-s)\opA}\D
         L_{s}+\E^{-t\opA}\pi_{1}-\int_{0}^{t}\E^{(t-s)\opA}\D
         L_{s}\nonumber \\
       &=\int_{0}^{\infty}\E^{(t-s)\opA}\D L_{s}+\int_{0}^{t}\E^{-(t-s)\opA}\D
         L_{s}+\int_{-\infty}^{0}\E^{-(t-s)\opA}\D
         L_{s}-\int_{0}^{t}\E^{(t-s)\opA}\D L_{s}\nonumber\\
       &= \int_{-\infty}^{t}\E^{-\opA(t-s)}\D L_{s}+\int_{t}^{\infty}\E^{-\opA(s-t)}\D L_{s},\label{eq:bk22-well-balanced-stationary-2} 
\end{align}
which proves representation~\eqref{eq:bk22-well-balanced-representation}. We
set $\tilde{X}^{(1)}_{t}=\int_{-\infty}^{t}\E^{-\opA(t-s)}\D L_{s}$ and
$\tilde{X}^{(2)}_{t}=\int_{t}^{\infty}\E^{-\opA(s-t)}\D L_{s}$. It follows from
Section~\ref{sec:bk22-posit-semi-defin-1}, that $\tilde{X}^{(1)}$ and
$\tilde{X}^{(2)}$ are stationary OU processes and due to
  the fact that $-\opA$ is assumed to be quasi-positive, it also follows that
both $\tilde{X}^{(1)}_{t}$ and $\tilde{X}^{(2)}_{t}$ are positive semi-definite
for all $t\geq 0$. Thus we conclude that also $(X_{t})_{t\geq 0}$ must be stationary, positive semi-definite and
possesses the representation~\eqref{eq:bk22-well-balanced-representation}. We
continue with the computation of the expectation and auto-covariance. It is
easy to see that for all $t\geq 0$ we have 
\begin{align*}
  \EX{\tilde{X}_{t}^{(1)}}=\EX{\tilde{X}_{t}^{(2)}}=\opA^{-1}\mu_{L},
\end{align*}
which implies~\eqref{eq:bk22-well-balanced-expectation}. It is left to
prove that the auto-covariance of $(X_{t})_{t\geq 0}$ at $t\geq 0$
satisfies~\eqref{eq:bk22-well-balanced-auto-covariance}. For this, we recall that for all
$t,h\geq 0$ we have 
\begin{align*}
 \Cov\left[\vect(X_{t+h}),\vect(X_{t})\right]=\EX{\vect(X_{t+h})\vect(X_{t})^{\intercal}}-\EX{\vect(X_{t+h})}\EX{\vect(X_{t})^{\intercal}},
\end{align*}
where according to~\eqref{eq:bk22-well-balanced-expectation} and linearity of the
expectation we see that
\begin{align*}
  \EX{\vect(X_{t+h})}=2(A^{\vect})^{-1}\EX{\vect(L_{1})},  
\end{align*}
and $\EX{\vect(X_{t})^{\intercal}}=2\EX{\vect(L_{1})}^{\intercal}((A^{\vect})^{\intercal})^{-1}$. Thus
by~\eqref{eq:bk22-well-balanced-representation}, we are left with the terms
\begin{align}\label{eq:bk22-well-balanced-stationary-3}
\EX{\vect(\tilde{X}_{t+h}^{(i)})\vect(\tilde{X}_{t}^{(j)})^{\intercal}},\quad i,j=1,2.
\end{align}
Note that for $i=j=1$ this term is the auto-covariance of
the stationary OU type process $(\tilde{X}_{t}^{(1)})_{t\geq
  0}$ adjusted by the following outer-square of the expectation:
\begin{align*}
\EX{\vect(\tilde{X}^{(1)}_{t+h})}\EX{\vect(\tilde{X}^{(1)}_{t})^{\intercal}}=(A^{\vect})^{-1}\EX{\vect(L_{1})}\EX{\vect(L_{1})^{\intercal}}((A^{\vect})^{-1}))^{\intercal}.  
\end{align*}
Note that following~\eqref{eq:bk22-covariance-OU}, the auto-covariance of
the process $(\vect(\tilde{X}^{(1)}_{t}))_{t\geq 0}$ is given by 
\begin{align*}
 \Cov\left[\vect(\tilde{X}^{(1)}_{t+h}),\vect(\tilde{X}^{(1)}_{t})\right]=\E^{-h
  A^{\vect}}\mathcal{D}^{-1}\cQ^{\vect},\quad h\geq 0. 
\end{align*}
Similarly, for $i=j=2$ a straightforward computation shows that the
auto-covariance of $(\tilde{X}_{t}^{(2)})_{t\geq 0}$ is the same as the
auto-covariance of $(\tilde{X}^{(1)}_{t})_{t\geq 0}$. 
For $i=2$ and $j=1$, we see that by the definition of a two-sided L\'evy
process we have 
\begin{align*}
\Cov\left[\vect(\tilde{X}^{(2)}_{t+h}),\vect(\tilde{X}^{(1)}_{t})\right]=0.   
\end{align*}
Hence we are left with the last term
in~\eqref{eq:bk22-well-balanced-stationary-3}, that is $i=1$ and
$j=2$. Note first, that for every $h\geq 0$ we have
\begin{align}\label{eq:covariance-1} \vect(\tilde{X}_{t+h}^{(1)})\vect(\tilde{X}_{t}^{(2)})^{\intercal}&=\int\limits_{-\infty}^{t+h}\E^{-(t+h-u)A^{\vect}}\D\vect(L_{u})\int\limits_{t}^{\infty}\D\vect(L_{s})^{\intercal}\E^{-(s-t)(A^{\vect})^{\intercal}}.
\end{align}
Due to the fact that $(L_{t})_{t\in\MR}$ has stationary
  increments we note that the term on the right-hand side of~\eqref{eq:covariance-1} is in distribution equal to  
  \begin{align*} \int\limits_{-\infty}^{h}\int\limits_{0}^{\infty}\E^{-(h-u)A^{\vect}}\D\vect(L_{u})\D\vect(L_{s})^{\intercal}\E^{-s(A^{\vect})^{\intercal}}.
  \end{align*}
  Due to the independence of the increments of $(L_{t})_{t\in\MR}$ and $\D \vect(L_{u})\D \vect(L_{s})^{\intercal}=\D
  \vect(L_{u})\D\vect(L_{u})^{\intercal}+\D\big(\vect(L_{u})-\vect(L_{s})\big)\D\vect(L_{s})^{\intercal}$,
  it follows that 
  \begin{align*}    \EX{\vect(\tilde{X}_{t+h}^{(1)})\vect(\tilde{X}_{t}^{(2)})^{\intercal}}&=\int\limits_{0}^{h}\E^{-(h-u)   A^{\vect}}\EX{\D\vect(L_{u})\D\vect(L_{u})^{\intercal}}\E^{-u(A^{\vect})^{\intercal}}. 
  \end{align*}
  Thus, by compensating $\vect(\bar{L}_{u})\df \vect(L_{u})-u\EX{\vect(L_{1})}$
  we see that the quadratic variation $\llangle\,\cdot\,\rrangle$ satisfies $\EX{\D \vect(\bar{L}_{u})\D\vect(\bar{L}_{u})^{\intercal}}=\EX{\D \llangle
  \vect(\bar{L})\rrangle_{u}}= \cQ^{\vect}\D u$ and
\begin{align*}  \EX{\vect(\tilde{X}_{t+h}^{(1)})\vect(\tilde{X}_{t}^{(2)})^{\intercal}}&=\E^{-A^{\vect}h}\int\limits_{0}^{h}\E^{u
                                                                           A^{\vect}}\cQ^{\vect}\E^{-u(A^{\vect})^{\intercal}}\D
                                                                           u\\
                                                                         &\quad+(A^{\vect})^{-1}\EX{\vect(L_{1})}\EX{\vect(L_{1})^{\intercal}}\big((A^{\vect})^{\intercal})\big)^{-1}\\
                                                                         &=\E^{-h
                                                                           A^{\vect}}\int_{0}^{h}\E^{u\hat{\cD}}\cQ^{\vect}\D
                                                                           u\\
  &\quad +(A^{\vect})^{-1}\EX{\vect(L_{1})}\EX{\vect(L_{1})^{\intercal}}((A^{\vect})^{-1}))^{\intercal}.                                                                    
\end{align*}
The integral in the last equation can be computed as
\begin{align*}
  \int_{0}^{h}\E^{u\hat{\cD}}\cQ^{\vect}\D u=\hat{\cD}^{-1}\big(\E^{h\hat{\cD}}-\opI\big)\cQ^{\vect}.
\end{align*}
Hence by collecting all the terms
in~\eqref{eq:bk22-well-balanced-stationary-3} and due to the fact that $\cD$, $\hat{\cD}$ and
$\exp(\opA)$ mutually commute, we obtain~\eqref{eq:bk22-well-balanced-auto-covariance}.  
\end{proof}

\begin{remark}
Following our terminology from Definition~\ref{def:bk22-CARMA-p-q}, we see that
positive semi-definite well-balanced OU processes are non-causal MCARMA
processes of order $(2,0)$. Proposition~\ref{prop:bk22-positive-well-balanced-OU} shows that the
positivity criteria for non-stable state space models from
Theorem~\ref{thm:bk22-positivite-non-stable-MCARMA} is indeed not
necessary. As noted before, this happens as for stationary processes (such
as the well-balanced OU) it would suffice to ensure the positivity
only for its stationary states and not, as in
Theorem~\ref{thm:bk22-positivite-non-stable-MCARMA}, for all positive initial values $Z_{0}$. Moreover, note that the stationary distribution does not have to be
supported on the positive cone $(\MS_{d}^{+})^{p}$ for the output process
$(X_{t})_{t\geq 0}$ to be positive. Lastly, recall that the
  spot-covariance process is not directly observable in markets and
  independent of the Wiener noise $(W_{t})_{t\geq 0}$. In particular, the
  non-causality of the well-balanced OU is generally speaking not a concern. If, for some reason, also the spot covariance
  process shall be causal, we can follow the arguments in~\cite{BL21} to obtain a
  causal CARMA process, having the same second order moment
  structure, that could possibly substitute the former, if it is also
  $\MS_{d}^{+}$-valued.
\end{remark}
 
\subsection{Auto-covariance structure of positive semi-definite well-balanced OU
  based stochastic covariance models}
In this section we study the second order structure of stochastic covariance
models with spot covariance process modeled by a positive
semi-definite well-balanced OU process. In particular, we compare the
obtained auto-covariance of the squared (logarithmic)-return process with the
corresponding auto-covariance
in~\eqref{eq:bk22-auto-covariance-squared-returns-OU} in the multivariate BNS
model. In the next lemma we compute the function $r^{++}$
from~\eqref{eq:bk22-r++OU} for the positive semi-definite well-balanced OU process.

\begin{lemma}\label{lem:bkk22-r++-well-balanced-ou}
  Let $\opA\in\cL(\MM_{d})$ be such that $-\opA$ is quasi-positive and
  $\tau(-\opA)<0$ and denote by $(X_{t})_{t\geq 0}$ the associated positive
  semi-definite well-balanced OU process as in
  Proposition~\ref{prop:bk22-positive-well-balanced-OU}. For every $t\geq
  0$ we have:
  \begin{align}\label{eq:bk22-r++wbOU}
    r^{++}(t)=\left((A^{\vect})^{-2}\E^{-t
    A^{\vect}}\big(\cG\E^{t\hat{\cD}}-\opI\big)-\cD_{1}t+\cD_{2}\right)\hat{\cD}^{-1}\cQ^{\vect}+2r^{++}_{OU}(t),
  \end{align}
  where $r^{++}_{OU}(t)$ is as in~\eqref{eq:bk22-r++OU}, $\cG\in
  \cL(\MM_{d^{2}})$ is defined by
  $\cG(x)\df(A^{\vect})^{2}x\big((A^{\vect})^{\intercal})^{-2}$ and
  $\cD_{i}\in\cL(\MM_{d^{2}})$ for $i=1,2$ is given by $\cD_{i}(x)\df (A^{\vect})^{-i}x-x\big((A^{\vect})^{\intercal})^{-i}$.
\end{lemma}
\begin{proof}
Let $T\geq 0$, then by definition of $r^{++}$ in~\eqref{eq:bk22-r++} we have
to compute
\begin{align*}
  r^{++}(t)&=\int_{0}^{t}\int_{0}^{s}\mathrm{acov}_{X_{T}}(u)\D u\D s,\quad t\geq 0,
\end{align*}
for $\mathrm{acov}_{X_{T}}(u)=\Cov\left[X_{T+u},X_{T}\right]$ given
by~\eqref{eq:bk22-well-balanced-auto-covariance}. By~\eqref{eq:bk22-covariance-OU}
we see that for all $u\geq 0$ the auto-covariance function
$\mathrm{acov}_{X_{T}}(u)$ is the sum of 
$\E^{-hA^{\vect}}\big(\E^{h\hat{\cD}}-\opI\big)\hat{\cD}^{-1}\cQ^{\vect}$ and
two times the auto-covariance function of a classical OU type process. We thus
see that
\begin{align*}
  r^{++}(t)&=\int_{0}^{t}\int_{0}^{s}\E^{-uA^{\vect}}\big(\E^{u\hat{\cD}}-\opI\big)\hat{\cD}^{-1}\cQ^{\vect}\D
             u\D s+2 r^{++}_{OU}(t)\\
           &=\int_{0}^{t}-\hat{\cD}^{-1}\cQ^{\vect}\big((A^{\vect})^{\intercal}\big)^{-1}\big(\E^{-s(A^{\vect})^{\intercal}}-\MI_{d^{2}}\big)\\
           &\qquad+(A^{\vect})^{-1}\big(\E^{-sA^{\vect}}-\MI_{d^{2}}\big)\hat{\cD}^{-1}\cQ^{\vect}\D
             s+2r_{OU}^{++}(t)\\
           &=\hat{\cD}^{-1}\cQ^{\vect}\big((A^{\vect})^{\intercal}\big)^{-2}(\E^{-t(A^{\vect})^{\intercal}}\!\!-\MI_{d^{2}})-(A^{\vect})^{-2}(\E^{-t
             A^{\vect}}\!\!-\MI_{d^{2}})\hat{\cD}^{-1}\cQ^{\vect}\\
  &\qquad
    +\hat{\cD}^{-1}\cQ^{\vect}\big((A^{\vect})^{\intercal}\big)^{-1}t-(A^{\vect})^{-1}\hat{\cD}^{-1}\cQ^{\vect}+2r^{++}_{OU}(t)\\
           &=\big((A^{\vect})^{-2}\E^{-t
             A^{\vect}}\big(\cG\E^{t\hat{\cD}}-\opI\big)-\cD_{1}t+\cD_{2}\big)\hat{\cD}^{-1}\cQ^{\vect}+2r^{++}_{OU}(t),  
\end{align*}
which proves~\eqref{eq:bk22-r++wbOU}.
\end{proof}

From Lemma~\ref{lem:bkk22-r++-well-balanced-ou}
and~\eqref{eq:bk22-auto-covariance-Xn} we see that the auto-covariance
function of the squared
(logarithmic)-returns~\eqref{eq:bk22-auto-covariance-Yn} in the positive
semi-definite well-balanced OU based stochastic covariance model is given by 
\begin{align}\label{eq:bk22-auto-covariance-squared-returns-wb-OU} \mathrm{acov}_{Y_{n}^{\Delta}(Y_{n}^{\Delta})^{\intercal}}(h)&=\E^{-A^{\vect}\Delta(h-1)}(A^{\vect})^{-2}(\cG\E^{\Delta(h-1)\hat{D}}\!-\opI)(\MI_{d^{2}}\!-\!\E^{-A^{\vect}\Delta})^{2}\hat{\cD}^{-1}\cQ^{\vect}\nonumber\\
  &\quad +
    2\E^{-A^{\vect}\Delta(h-1)}(A^{\vect})^{-2}\big(\MI_{d^{2}}-\E^{-A^{\vect}\Delta}\big)^{2}\cD^{-1}\cQ^{\vect},\quad
    h\in\MN, 
\end{align}
for every $n\in\MN$ and $\Delta>0$. Note that the term in the second line
of~\eqref{eq:bk22-auto-covariance-squared-returns-wb-OU} is simply two times
the auto-covariance of the multivariate BNS model (compare
with~\eqref{eq:bk22-auto-covariance-squared-returns-OU}). Thus, the interesting term
is the first one
in~\eqref{eq:bk22-auto-covariance-squared-returns-wb-OU}. Indeed, since
$\mathbf{A}\in\cL(\MS^{d})$ it follows from~\cite{PH81} that $A^{\vect}$ is symmetric. Then note that for every
$\Delta(h-1)>0$ we obtain the following expansion of the non-linear term
in the first line of~\eqref{eq:bk22-auto-covariance-squared-returns-wb-OU}:
\begin{align*}
(A^{\vect})^{-2}(\cG\E^{\Delta(h-1)\hat{\cD}}-\opI)\hat{\cD}^{-1}=\Delta(h-1)+\frac{1}{2}\hat{\cD}(\Delta(h-1))^{2}+\cO\big((\Delta(h-1))^{3}\big).  
\end{align*}
Therefore, we see that the auto-covariance function of the squared
  (logarithmic)-returns in the positive semi-definite well-balanced OU based model
  in~\eqref{eq:bk22-auto-covariance-squared-returns-wb-OU} admits \emph{polynomial
    diminished exponential decay} compared to (pure) exponential decay in case
  of the multivariate BNS model in~\eqref{eq:bk22-auto-covariance-squared-returns-OU}.  
Moreover, we observe the following small-time asymptotics of the auto-covariance
function $\mathrm{acov}_{Y_{n}^{\Delta}(Y_{n}^{\Delta})^{\intercal}}(h)$:
\begin{align*}  \mathrm{acov}_{Y_{n}^{\Delta}(Y_{n}^{\Delta})^{\intercal}}(h)\approx\Delta(h\!-\!1)\E^{-A^{\vect}\Delta(h-1)}\cQ^{\vect}\!+\mathcal{O}\big(\Delta(h\!-\!1))^{2}\E^{-^{\vect}\Delta(h-1)}\cQ^{\vect}\big),    
\end{align*}
for small values of $\Delta(h-1)$.\par{}
This shows that in the short-time lags the auto-covariance function of the
squared returns in well-balanced OU based models behave like $h\mapsto\!\Delta(h-1)\E^{-A^{\vect}\!\Delta(h-1)}\cQ^{\vect}$ (versus
$h\mapsto\E^{-A^{\vect}\Delta(h-1)}\cQ^{\vect}$ in the classical OU
case). Now, for simplicity assume that $\cQ^{\vect}$ is the
  identity-matrix. Then we note that although the auto-covariance function of
  classical OU based models do admit non-monotone behavior in certain
  directions, it is always monotone along a vector of maximal variance, i.e. in
  spectral norm $\norm{\cdot}_{2}$ the auto-covariance function is monotone. In
  contrast, we see that the auto-covariance function in the well-balanced OU
  case allows for non-monotone configurations along a vector of maximal variance
  as $h\mapsto \Delta(h\!-\!1)\E^{\tau(-A^{\vect})\Delta(h-1)} $ is non-monotone
  (note that $\tau(-A^{\vect})<0$ by assumption and
  $\norm{\E^{-A^{\vect}t}}_{2}=\E^{\tau(-A^{\vect}) t}$ due to the fact that 
  $A^{\vect}$ is diagonalizable). Note further that the attainable
auto-covariance functions of higher-order MCARMA based models are also
beyond the ones obtained from models based on superposition of positive semi-definite OU
processes, see~\cite{BNS11}.\par{}
In summary, the example of the well-balanced OU process demonstrates that the
class of higher-order MCARMA based stochastic covariance models is indeed a very
flexible model class providing multiple modeling options to capture short-memory
effects observed in financial and non-financial data.



\nosectionappendix{}

\begin{toappendix}

\section{A submultiplicativity property of the Hadamard product}

\begin{lemma}\label{lem:hadamard-product}
    Let $d\in\MN$ and $A,B\in \MM_{d}$ be non-negative matrices,
    i.e. $A,B\in\pi(\MR_{d}^{+})$. Then for all $n\in\MN$ we have $(A\odot
    B)^{n}\preceq A^{n}\odot B^{n}$. 
  \end{lemma}
  \begin{proof}
    We write $A=(a_{i,j})_{1\leq i,j\leq d}$ and $B=(b_{i,j})_{1\leq i,j\leq d}$
    and prove the statement by induction over $n\in\MN$. For $n=1$ the
    statement is trivial. In case of $n=2$, it follows from the non-negativity of $(a_{i,j})_{1\leq i,j\leq d}$ and
    $(b_{i,j})_{1\leq i,j\leq d}$ that for every $i,j\in \set{1,\ldots,d}$ we have
    \begin{align*}
      ((A\odot B)^{2})_{i,j}=\sum_{k=1}^{d}a_{i,k}b_{i,k}a_{k,j}b_{k,j}\leq
      \!\big(\sum_{k=1}^{d}a_{i,k}a_{j,k}\big)\big(\sum_{k=1}^{d}b_{i,k}b_{j,k}\big)\!=((A^{2}\odot B^{2}))_{i,j},
    \end{align*}
    which proves the statement for $n=2$. Now, suppose that $(A\odot B)^{n}\preceq
    A^{n}\odot B^{n}$ for some $n\in\MN$. We show that also $(A\odot B)^{n+1}\preceq
    A^{n+1}\odot B^{n+1}$ holds. Indeed, we have
    \begin{align*}
     ((A\odot B)^{n+1})_{i,j}=((A\odot B)(A\odot B)^{n})_{i,j}\leq ((A\odot B)A^{n}\odot B^{n})_{i,j}, 
    \end{align*}
    where we note that $((A\odot B)A^{n}\odot
    B^{n})_{i,j}=\sum_{k=1}^{d}a_{i,k}(A^{n})_{k,j}b_{i,k}(B^{n})_{k,j}$ and
    hence
    \begin{align*}
     ((A\odot B)^{n+1})_{i,j}&\leq
      \sum_{k=1}^{d}a_{i,k}(A^{n})_{k,j}b_{i,k}(B^{n})_{k,j}\\
      &\leq
      \big(\sum_{k=1}^{d}a_{i,k}(A^{n})_{k,j}\big)\big(\sum_{k=1}^{d}b_{i,k}(B^{n})_{k,j}\big)=(A^{n+1}\odot
      B^{n+1})_{i,j},
    \end{align*}
    which yields the desired inequality $(A\odot B)^{n+1}\preceq A^{n+1}\odot B^{n+1}$. 
  \end{proof}
    
\end{toappendix}

\bibliographystyle{unsrt}

\end{document}